\documentclass[a4paper,11pt]{amsart}

\usepackage[left=2.7cm,right=2.7cm,top=3.5cm,bottom=3cm]{geometry}
\usepackage{graphicx}
\newcommand{\Bmu}{\mbox{$\raisebox{-0.59ex}
  {$l$}\hspace{-0.18em}\mu\hspace{-0.88em}\raisebox{-0.98ex}{\scalebox{2}
  {$\color{white}.$}}\hspace{-0.416em}\raisebox{+0.88ex}
  {$\color{white}.$}\hspace{0.46em}$}{}}

\numberwithin{equation}{section}

\usepackage{amsthm,amssymb,amsmath,amsfonts,mathrsfs,amscd,amsbsy,dsfont,verbatim,relsize}
\usepackage{stmaryrd,dsfont}
\usepackage[utf8]{inputenc}
\usepackage[T1]{fontenc}
\usepackage[all,cmtip]{xy}  
\usepackage[new]{old-arrows}
\usepackage{latexsym}
\usepackage{longtable}
\usepackage{mathtools}
\usepackage{marginnote}

\usepackage{graphicx} 
\newcommand*{\Scale}[2][4]{\scalebox{#1}{$#2$}}%
\usepackage{stmaryrd}
\usepackage{amsmath}
\usepackage{amsthm}
\usepackage{amssymb}
\usepackage{mathrsfs}
\usepackage{tikz-cd}
\usepackage{verbatim}

\usepackage[pagebackref]{hyperref}

\mathtoolsset{showonlyrefs}

\theoremstyle{plain}
\newtheorem{theorem}{Theorem}[section]
\newtheorem{proposition}[theorem]{Proposition}
\newtheorem{lemma}[theorem]{Lemma}

\newtheorem{conjecture}[theorem]{Conjecture}

\newcounter{lettera}

\newtheorem{itheorem}[lettera]{Theorem}

\theoremstyle{definition}
\newtheorem{definition}[theorem]{Definition}
\newtheorem{assumption}[theorem]{Assumption}

\theoremstyle{remark}
\newtheorem{remark}[theorem]{Remark}

\newtheorem{remark/notation}[theorem]{Remark/Notation}
\newtheorem{notation/convention}[theorem]{Notation/Convention}

\newcommand{\longmono}{\mbox{\;$\lhook\joinrel\longrightarrow$\;}}

\newfont{\cyr}{wncyr10 scaled 1100}
\newcommand{\Sha}{\mbox{\cyr{X}}}

\newcommand{\cyc}{\mathrm{cyc}}
\newcommand{\unr}{\mathrm{unr}}

\DeclareMathOperator{\Fitt}{Fitt}
\DeclareMathOperator{\Gal}{Gal}
\DeclareMathOperator{\Sel}{Sel}
\DeclareMathOperator{\ind}{ind}
\DeclareMathOperator{\ord}{ord}

\DeclareMathOperator{\loc}{loc}
\DeclareMathOperator{\length}{length}
\DeclareMathOperator{\Frob}{Frob}

\DeclareMathOperator{\Frac}{Frac}
\DeclareMathOperator{\divv}{div}
\DeclareMathOperator{\Hom}{Hom}

\DeclareMathOperator{\im}{im}
\DeclareMathOperator{\M}{M}
\DeclareMathOperator{\charr}{char}

\DeclareMathOperator{\tors}{tors}

\DeclareMathOperator{\Gr}{Gr}
\DeclareMathOperator{\GL}{GL}

\DeclareMathOperator{\ks}{\mathbf{KS}}
\DeclareMathOperator{\uks}{\overline{\mathbf{KS}}}
\DeclareMathOperator{\tr}{tr}
\DeclareMathOperator{\rel}{rel}
\DeclareMathOperator{\arith}{arith}
\DeclareMathOperator{\cont}{cont}
\DeclareMathOperator{\Aut}{Aut}
\DeclareMathOperator{\Stub}{Stub}
\DeclareMathOperator{\s}{s}
\DeclareMathOperator{\str}{str}
\DeclareMathOperator{\fs}{fs}
\DeclareMathOperator{\EX}{EX}
\DeclareMathOperator{\f}{f}

\newcommand{\defeq}{\vcentcolon=}

\newcommand{\Q}{\mathds Q}
\newcommand{\N}{\mathds N}
\newcommand{\Z}{\mathds Z}

\newcommand{\C}{\mathds C}

\newcommand{\F}{\mathds F}

\newcommand{\cO}{\mathcal O}
\newcommand{\p}{\mathfrak p}

\newcommand{\ms}{\mathfrak{s}}
\newcommand{\mt}{\mathfrak{t}}

\newcommand{\francesco}[1]{{\color{blue} $\clubsuit \clubsuit \clubsuit$ Francesco: #1}}
\newcommand{\matteo}[1]{{\color{magenta} $\clubsuit \clubsuit \clubsuit$ Matteo: #1}}

\begin{document}

\title{Variation of Shafarevich--Tate groups in Hida families}
\author{Enrico Da Ronche, Matteo Longo, Stefano Vigni and Francesco Zerman}

\thanks{The authors are partially supported by the GNSAGA group of INdAM. The research by the first and third authors is partially supported by the MUR Excellence Department Project awarded to Dipartimento di Matematica, Universit\`a di Genova, CUP D33C23001110001. The last named author was supported by the \emph{ERC Consolidator Grant ShimBSD: Shimura varieties and the BSD conjecture}.}


\address{Dipartimento di Matematica, Universit\`a di Genova, Via Dodecaneso 35, 16146 Genova, Italy}
\email{enrico.daronche@edu.unige.it}
\address{Dipartimento di Matematica, Universit\`a di Padova, Via Trieste 63, 35121 Padova, Italy}
\email{matteo.longo@unipd.it}
\address{Dipartimento di Matematica, Universit\`a di Genova, Via Dodecaneso 35, 16146 Genova, Italy}
\email{stefano.vigni@unige.it}
\address{FernUni Schweiz, Schinerstrasse 18, 3900 Brig, Switzerland}
\email{francesco.zerman@unidistance.ch}



\begin{abstract}
Let $\mathbf{f}$ be a Hida family of cusp forms over a standard (power series) Iwasawa algebra. We determine the pseudo-isomorphism class of the big Shafarevich--Tate group attached to $\mathbf{f}$ over a suitable imaginary quadratic field by describing all its higher Fitting ideals; this description is given in terms of the universal Kolyvagin system of Howard's big Heegner points. As a second contribution, we study the variation of the structure of the Shafarevich--Tate groups of the classical specializations of $\mathbf{f}$.
\end{abstract}

\maketitle

\section{Introduction}

\subsection{Motivation} 

The goal of the present paper is to study the variation of the structure of Shafarevich--Tate groups in a Hida family of modular forms. To explain the context and give some motivation for our results, recall that if $E$ is an elliptic curve over $\Q$ and $K$ is an imaginary quadratic field, then the group structure of the Shafarevich--Tate group $\Sha(E/K)$ of $E$ over $K$ was first described, under suitable arithmetic assumptions, by Kolyvagin (\cite{KolStructure}; \emph{cf.}~also \cite{McCallum}); more precisely, one assumes that all the prime factors of the conductor of $E$ split in the imaginary quadratic field $K$ and that the basic $K$-rational Heegner point $y_K$ is non-torsion (here $y_K\in E(K)$ is defined by taking the Galois trace of the Heegner point of conductor $1$ on $E$, which is rational over the Hilbert class field of $K$). Under this assumption on $y_K$, Kolyvagin proves that $\Sha(E/K)$ is finite; his description of the isomorphism class of $\Sha(E/K)$ makes use of the collection of ``derived'' Heegner classes built out of Heegner points rational over ring class fields of $K$ and is obtained by studying divisibility properties of such a family. The case of higher (even) weight cusp forms was then considered by Masoero (\cite{masoero}), who extended Kolyvagin's techniques replacing Heegner points with Heegner cycles.  

Now let $p\geq5$ be a prime number and consider a $p$-adic family of cusp forms in the sense of Hida (\cite{Hida-Galois}, \cite{Hida-ENS}) or Coleman (\cite{coleman}). To each member of this family one can attach a Shafarevich--Tate group and the general question that we address in the present paper is to understand how these groups vary. We establish our main results in the case of Hida families and for Shafarevich--Tate groups over imaginary quadratic fields, further restricting (for technical reasons) to the case where the relevant Hida--Hecke algebra is a power series ring $\Lambda$ over a finite extension of $\Z_p$. Under suitable arithmetic conditions and assuming non-vanishing results for Howard's big Heegner points (\cite{Howard-Variation}), which are $\Lambda$-adic analogues of Heegner points, our result is a uniform description of the Fitting ideals of the Shafarevich--Tate group (viewed as $\Lambda$-module) of the $\Lambda$-adic Galois representation attached to our Hida family. In particular, the knowledge of Fitting ideals uniquely determines the pseudo-isomorphism class of this Shafarevich--Tate group. Specializing our main result to an arithmetic prime (\emph{i.e.}, a prime ideal of $\Lambda$ corresponding to a classical cusp form in the Hida family) recovers the group structure of the Shafarevich--Tate group of the corresponding eigenform, up to finitely many exceptions. 

\subsection{Main results} 

To explain our main results, fix as before a prime number $p\geq 5$ and an integer $N\geq 1$ such that $p\nmid N$; let 
\[ 
\mathbf{f}=\sum_{n\geq 1}a_n(\mathbf{f})q^n\in \Lambda[[q]]
\]
be a $p$-adic Hida family of tame level $N$, where $\Lambda\defeq\cO[[T]]$ and $\cO$ is the valuation ring of a finite extension of $\Q_p$. Set $\mathcal{X}\defeq\Hom_{\cont}(\Lambda,\overline{\Q}_p)$, where $\Hom_{\cont}$ denotes $\cO$-linear continuous homomorphisms, and let $\mathcal{X}^{\arith}$ be the subset of $\mathcal{X}$ consisting of arithmetic morphisms (these notions are reviewed in Section \ref{sec4}); for each $\eta\in\mathcal{X}^{\arith}$, the power series $\sum_{n\geq1}\eta\bigl(a_n(\mathbf{f})\bigr)q^n$ (called \emph{specialization of $\mathbf{f}$ at $\eta$}) is the $q$-expansion of a cusp form $\mathbf{f}_\eta$. 
Let us fix an imaginary quadratic field $K$ of discriminant $D_K$ such that
\begin{itemize}
\item $(D_K,Np)=1$;
\item all the prime factor of $N$ split in $K$.    
\end{itemize}
In other words, the field $K$ satisfies the \emph{Heegner hypothesis} relative to $N$ (\emph{cf.} Remark \ref{N^-} for comments on this assumption and possible extensions of our results to more general settings).
As a consequence, the root number (\emph{i.e.}, the sign of the functional equation of the $L$-function) of the base change of $\mathbf{f}_\eta$ to $K$ is generically $-1$, so the $L$-function of $\mathbf{f}_{\eta}$ over $K$ vanishes at the central critical point $k_\eta/2$, where $k_\eta$ is the weight of $\mathbf{f}_{\eta}$, for all but finitely many $\eta\in\mathcal{X}^{\arith}$. 

Let $\rho:G_\Q\rightarrow\GL_2(\Lambda)$ denote the $\Lambda$-adic Galois representation attached to $\mathbf{f}$ by Hida; then for arithmetic primes $\eta$, composing $\rho$ with localization at the kernel of $\eta$ gives a Galois representation isomorphic to the $p$-adic Galois representation attached to $\mathbf{f}_\eta$ by Deligne (\cite{Del-Bourbaki}). Consider the self-dual twist $\mathbf{T}^\dagger$ of $\mathbf{T}$, which should be understood as a $p$-adic family of self-dual Galois representations, and assume that its associated residual representation is absolutely irreducible. Denote by $\mathbf{A}^\dagger\defeq\Hom_{\cont}\bigl(\mathbf{T}^\dagger,\Bmu_{p^\infty}\bigr)$ the Kummer dual of $\mathbf{T}$, where $\Bmu_{p^\infty}$ is the (discrete) group of $p$-power roots of unity. Then we consider the (Greenberg) Selmer $\Lambda$-module $\Sel_{\Gr}(K,\mathbf{A}^\dagger)$ of $\mathbf{A}^\dagger$ over $K$, whose quotient by the maximal $\Lambda$-divisible subgroup is, by definition, the \emph{Shafarevich--Tate group $\Sha_{\Gr}\bigl(K,\mathbf{A}^\dagger\bigr)$} attached to $\mathbf{f}$. This is a $\Lambda$-module whose Pontryagin dual 
\[
X\defeq\Sha_{\Gr}\bigl(K,\mathbf{A}^\dagger\bigr)^\vee
\]
is finitely generated and torsion over $\Lambda$. Our main result describes, under Assumption \ref{Hidaass}, all the Fitting ideals of $X$ in terms of certain universal Kolyvagin systems.

Universal Kolyvagin systems have been introduced by Mazur and Rubin (\cite[Definition 3.1.6]{MR}); they are constructed by compiling Kolyvagin systems for \emph{artinian} quotients of $\mathbf{T}^\dagger$: see \S\ref{secUKS}, and especially Definition \ref{defUKS}, for details on this notion. In particular, a universal Kolyvagin system $\kappa$ admits a basic class $\kappa_1\in\Sel_{\Gr}\bigl(K,\mathbf{T}^\dagger\bigr)$ that one should understand in this abstract context as the $\Lambda$-adic analogue of the Heegner point $y_K$ considered before.  
More concretely, the universal Kolyvagin system that is relevant for this paper has been explicitly constructed by B\"{u}y\"{u}kboduk (\cite{Kazim}) by means of the big Heegner points introduced by Howard (\cite{Howard-Variation}).
Here we are going to work with the system $\kappa$ obtained by specializing B\"{u}y\"{u}kboduk's system with respect to the augmentation ideal. The basic class $\kappa_1$ for this universal Kolyagin system is just the big Heegner point $\mathfrak{z}_K\in H^1(K,\mathbf{T}^\dagger)$ defined in \cite{Howard-Variation}. As mentioned before, universal Kolyvagin systems are built from classes over artinian quotients of $\mathbf{T}^\dagger$. To explain this, we need some more notation. First, for every integer $n\geq1$ we denote by $K[n]$ the ring class field of $K$of conductor $n$ and set $G_n\defeq\Gal(K[n]/K[1])$. Let $\mathfrak{m}_\Lambda=(p,T)$ be the maximal ideal of $\Lambda$ and, for every integer $s\geq1$, 
set $T_s^\dagger\defeq\mathbf{T}^\dagger/\mathfrak{m}_\Lambda^s\mathbf{T}^\dagger$. 
Denote by $\mathcal{P}_s$ the set of Kolyvagin primes for $T_s^\dagger$; this is the
set of primes $\ell$ that do not divide $Np$, are inert in $K$ and satisfy $(\ell+1)T_s^\dagger=\{0\}$ (\emph{cf.}~Assumption \ref{ass:kolyvagin-primes}). 
Then $\kappa$ is determined by (compatible) classes $\kappa_n^{(s,t)}\in H^1\bigl(K,T_s^\dagger\bigr)\otimes G_n$ for integers $t\geq s\geq1$ and integers $n\geq1$ that are square-free products of primes in $\mathcal{P}_t$. Using these classes, one may introduce, for all $r,s\geq1$, the ideal $\mathfrak{C}_r(s)$ of $\Lambda/\mathfrak{m}^s$, which is defined as the intersection of ideals generated by the localizations $\loc_\ell(\kappa^{(s,t)}_{n})$
when $n$ is a product of at most $r$ (distinct) primes in $\mathcal{P}_{2s}$ and $\ell\in\mathcal{P}_{2s}$ is a prime that does not divide $n$, with $\kappa^{(s,t)}$ varying among some representatives of $\kappa$ modulo $\mathfrak{m}_\Lambda^s$. Then we can define the inverse limit 
\[
\mathfrak{C}_r\defeq\varprojlim_{s}\mathfrak{C}_r(s)\subset\Lambda
\]
(\emph{cf.}~Lemma \ref{defC_rlemma}). Let $I(\Lambda)$ denote the set of ideals of $\Lambda$ and put on $I(\Lambda)$ the following equivalence relation: $I\sim J$ if and only if $I \Lambda_\p= J\Lambda_\p$ for all height-$1$ prime ideals $\p$ of $\Lambda$. We write $[I]$ for the class of $I\in I(\Lambda)$ in $I(\Lambda)/\sim$. For an even $i\in\N$, consider the elements of $I(\Lambda)/\sim$ given by 
\[
[\mathfrak{D}_i^\epsilon] \defeq \prod_{r \geq i/2} \frac{[\mathfrak{C}_{2r+1}]}{[\mathfrak{C}_{2r+2}]} , \quad [\mathfrak{D}_i^{-\epsilon}] \defeq \prod_{r\geq i/2} \frac{[\mathfrak{C}_{2r}]}{[\mathfrak{C}_{2r+1}]}
\]
(by Proposition \ref{monoid}, these definitions are well posed). Here $\epsilon \in \{\pm 1\}$ is the negative of the generic sign of the functional equation of the complex $L$-functions of the specializations of $\mathbf{f}$.
Our main result (Theorem \ref{HidaTheorem}) describes the Fitting ideals $\Fitt_i(X^\pm)$ of the eigenspaces $X^\pm$ for the complex conjugation in terms of the ideals $\mathfrak{D}_i^\pm$; by Theorem \ref{evenoddFitt}, this uniquely determines the pseudo-isomorphism class of the $\Lambda$-modules $X^\pm$.  

\begin{itheorem} \label{thmA}
Assume that $\kappa_1$ is not torsion. For each $i\in\N$, we have 
\[
\Fitt_i(X^\pm)\sim \begin{cases}(\mathfrak{D}_i^\pm)^2 & \text{if $i$ is even},\\[2mm]
\mathfrak{D}^\pm_{i-1}\cdot\mathfrak{D}^\pm_{i+1} & \text{if $i$ is odd}.
\end{cases}
\]
Furthermore, the set $\bigl\{\Fitt_i(X^\pm)\bigr\}_{i\in\N}$ uniquely determines the pseudo-isomorphism class of the $\Lambda$-module $X^\pm$. 
\end{itheorem}

It is conjectured that the class $\kappa_1$ is indeed not torsion: this is the \emph{horizontal non-vanishing conjecture} of Howard, see \cite[Conjecture 3.4.1]{Howard-Variation}. As should be clear from the definition of the ideals $\mathfrak{C}_i$, the proof of Theorem \ref{thmA} builds on analogous results over artinian local rings, following a reconstructive strategy that was exploited previously by Mazur--Rubin (\cite{MR}) and Howard (\cite{HoHeeg}), and previously used in \cite{DRLV}. Under appropriate assumptions on the relevant Kolyvagin systems, and working in a more abstract setting which goes beyond the case of Hida families, one compiles structure theorems for Shafarevich--Tate groups of representations over artinian local rings (Theorem \ref{artSel}) to obtain structure theorems for Shafarevich--Tate groups of representations over complete discrete valuation rings (Theorem \ref{thmDVR}); one then passes, by an approximation process, to a structure theorem for $\Lambda$-adic representations (Theorem \ref{mainThmR}), of which Theorem \ref{thmA} is the incarnation in the case of Hida families. 

By specialization to arithmetic primes, Theorem \ref{thmA} recovers in a uniform way structure theorems for Shafarevich--Tate groups of modular forms varying in a Hida family, thus giving a precise meaning and an answer to the motivating question introduced before. To explain this result, let $\eta$ be an arithmetic prime and define 
\[
T_\eta\defeq\bigl(\mathbf{T}^\dagger/\p_\eta\mathbf{T}^\dagger\bigr)\otimes_{\Lambda/\p_\eta}S_\eta,
\] 
where $\p_\eta$ is the kernel of $\eta$ and $S_\eta$ is the integral closure of 
$\Lambda/\p_\eta$. For each $\eta$, the universal Kolyvagin system $\kappa$ gives rise 
to a universal Kolyvagin system $\kappa_\eta$ for $T_\eta$ that is made of classes $\kappa_{\eta,n}^{(s,t)}$ in $H^1(K,T_{\eta}/\pi_\eta^sT_\eta)\otimes G_n$ where $n,s,t$ are as before and $\pi_\eta$ is a uniformizer of $S_\eta$. This Kolyvagin system is also equipped with a basic class $\kappa_{\eta,1}$ in the (Greenberg) Selmer group $\Sel_{\Gr}(K,T_\eta)$. 
Let $F_\eta$ be the fraction field of $S_\eta$, define $A_\eta\defeq T_\eta\otimes_{S_\eta}(F_\eta/S_\eta)$ and let $\Sha_{\Gr}(K,A_\eta)$ be the Shafarevich--Tate group of the Galois representation $A_\eta$, which is defined to be the quotient of the Greenberg Selmer group of $A_\eta$ by its maximal divisible subgroup; as before, let $X_\eta$ be the Pontryagin dual of $\Sha_{\Gr}(K,A_\eta)$. For every $i\in\N$, in \S\ref{sectionDVR} we introduce certain generalized divisibility indices $\delta^{(i)}(\kappa_\eta)$ attached to $\kappa_\eta$; the definition of $\delta^{(i)}(\kappa_\eta)$, which is ultimately inspired by analogous definitions appearing in \cite{KolStructure}, is rather technical and makes crucial use of a \emph{rigidity theorem} (Theorem \ref{Rigidity}) that is an analogue in this setting of a result by Howard (\cite[Theorem 2.5.1]{howard2012bipartite}) and may be of independent interest. Then Theorem \ref{DVRtheorem} offers the following description of $X_\eta^\pm$. 

\begin{itheorem} 
Assume that $\kappa_{\eta,1}\in \Sel_\mathrm{Gr}(K,T_\eta)$ is not torsion. There are isomorphisms
\begin{equation}
X^{\pm \epsilon}_\eta\simeq\bigoplus_{i\in\N}\bigl(S_\eta/\pi_\eta^{d^{\pm \epsilon}_i}S_\eta\bigr)^2,
\end{equation}
where $d^\epsilon_i\defeq\delta^{(2i+1)}(\kappa_\eta)-\delta^{(2i+2)}(\kappa_\eta)$ and $d^{-\epsilon}_i\defeq\delta^{(2i)}(\kappa_\eta)-\delta^{(2i+1)}(\kappa_\eta)$ . 
\end{itheorem}

It is worth observing that, under the assumption that $\kappa_1$ is non-torsion, $\kappa_{\eta,1}$ is non-torsion as well for all but finitely many arithmetic primes $\eta$: it would be interesting to understand what happens at these finitely many exceptions. 

These results are open to several further generalizations. As explained in \cite[Section 4]{NP}, a Hida family has coefficients in $\Lambda$ if, \emph{e.g.}, there exists $\eta\in\mathcal X^{\arith}$ such that $\mathbf{f}_\eta$ is a $p$-isolated newform in $S_k\bigl(\Gamma_0(Np))$ for some even $k\geq2$. The general setting of a Hida family with coefficients in a finite extension of $\Lambda$ and the counterparts of the results in this paper in the anticyclotomic Iwasawa theory for Hida families, as well as the case where $N$ is divisible also by primes that are inert in $K$, will be investigated in a future project. 

\subsection{Structure of the paper}

In Section \ref{sec2}, we review the basics of the theory of Fitting ideals and see how they can be used to retrieve information on the structure of finitely generated torsion modules over discrete valuation rings and over the Iwasawa algebra $\Lambda$. In Section \ref{sec3}, we recall the definition of a universal Kolyvagin system and prove our main results on their use in the computation of the structure of Shafarevich--Tate groups for representations over principal artinian local rings, discrete valuation rings and $\Lambda$. Finally, Section \ref{sec4} is devoted to the application of our results to the case of big Galois representations attached to Hida families of modular forms.

\subsection{Notation and terminology} \label{sec:notation-and-terminology}

For every perfect field $F$, denote by $G_F$ the absolute Galois group of $F$. Fix once and for all algebraic closures $\overline\Q$ and $\overline\Q_\ell$ for every prime number $\ell$; moreover, choose embeddings $\overline{\Q}\hookrightarrow\C$ and $\overline{\Q}\hookrightarrow\overline{\Q}_\ell$. Attached to these choices, there are a uniquely determined complex conjugation $\tau_c\in G_\Q$ and an arithmetic Frobenius element $\Frob_\ell\in G_\Q$, which is defined up to elements of the absolute inertia subgroup at $\ell$.

\subsection*{Acknowledgements} 

We would like to thank Francesc Castella and Paul-Antoine Seitz for several helpful discussions on some of the topics of this paper.

\section{Background on Fitting ideals} \label{sec2}

In this section we review the definition and the main properties of Fitting ideals and we give results regarding their role in structure theorems for finitely generated modules over some rings. In doing it, we follow closely the exposition of \cite[\S2]{DRLV}.

\subsection{Fitting ideals of modules over a noetherian ring}

Let $R$ be a noetherian ring and let $M$ be a finitely generated $R$-module. Let $n\in\N$ such that $M$ can be generated over $R$ by $n$ elements. Choose a presentation
\[
\bigoplus_{j=1}^k R \overset\varphi\longrightarrow \bigoplus_{i=1}^n R \longrightarrow M \longrightarrow 0
\]
of $M$, for some $k\ge 1$. Denote by $A\in\M_{n\times k}(R)$ the matrix associated with $\varphi$ with respect to the canonical basis; finally, given an integer $r\geq0$, write $I_r(A)$ for the ideal of $R$ generated by the minors of order $r$ of $A$, with the convention that $I_r(A)\defeq(0)$ if $r>\min\{n,k\}$ (recall that a minor of order $0$ is defined to be $1$, so that $I_0(A)=R$).

In the definition that follows, the notation above is in force.

\begin{definition} \label{fitting-def}
Let $i\in\N$. The \emph{$i$-th Fitting ideal} of $M$ over $R$ is 
\[
\Fitt_i(M)\defeq\begin{cases} I_{n-i}(A) & \text{for $i\leq n$},\\[2mm]R& \text{for $i>n$}.\end{cases}
\]
\end{definition}

The ideal $\Fitt_0(M)$ is sometimes called \emph{the (initial) Fitting ideal} of $M$, in which case the ideals $\Fitt_i(M)$ for $i\geq1$ are referred to as the \emph{higher Fitting ideals} of $M$. A Fitting ideal $\Fitt_i(M)$ with $i$ even (respectively, odd) is called \emph{even} (respectively, \emph{odd}). When we need to specify the ring $R$, we write $\Fitt_{R,i}(M)$ in place of $\Fitt_i(M)$.

The next result asserts that Definition \ref{fitting-def} does indeed make sense.

\begin{proposition}
The Fitting ideals of $M$ are well defined, i.e., their definition is independent of the choice of a presentation of $M$ as an $R$-module.
\end{proposition}

\begin{proof} See, \emph{e.g.}, \cite[Corollary--Definition 20.4]{Eis} or \cite[Ch. 3, Theorem 1]{northcott}. \end{proof}

\subsection{Fitting ideals of modules over a DVR}

Let $R$ be a discrete valuation ring with uniformizer $\pi$. We recall the following well-known structure theorem.

\begin{theorem} \label{DVR-isom-thm}
Let $M$ be a finitely generated torsion $R$-module. If $M$ is non-trivial, then there is an isomorphism of $R$-modules
\[
M \simeq \bigoplus_{j=1}^n R\big/\bigl(\pi^{k_j}\bigr),
\]
where $k_j \geq 0$ are integers uniquely determined up to order.
\end{theorem}

\begin{proof} See, \emph{e.g.}, \cite[Ch. III, Theorem 7.5]{lang}, of which this result is a special case. \end{proof}

As stated in the following theorem, the structure of $M$ is determined (up to isomorphism) by the sequence of ideals $\bigl(\Fitt_i(M)\bigr)_{i\geq0}$.

\begin{theorem}\label{FittingDVR}
Let $M$ be a non-trivial finitely generated torsion $R$-module and suppose that
\begin{equation} \label{M-isom}
M \simeq \bigoplus_{j=1}^n R\big/\bigl(\pi^{k_j}\bigr)
\end{equation}
for integers $n\geq1$ and $1\leq k_1\leq\dots\leq k_n$. Then for each $i\in\{1,\dots,n\}$ there is an equality 
\[
\Fitt_i(M)=(\pi^{c_i})
\]
of ideals of $R$, where $c_i\defeq\sum_{j=1}^{n-i}k_j$.
\end{theorem}

\begin{proof}
    See \cite[Theorem 2.5]{DRLV}.
\end{proof}

\subsection{Fitting ideals of modules over $\Lambda$}\label{sec:fitting-over-Lambda}

Let $\mathcal{O}$ be a complete discrete valuation ring with finite residue field. The power series ring $\Lambda\defeq \mathcal{O} \llbracket T \rrbracket$ in one variable over $\mathcal{O}$ is a $2$-dimensional complete noetherian regular local ring. A finitely generated $\Lambda$-module $M$ is called $\emph{pseudo-null}$ if $M_\p=0$ for all prime ideals of $\Lambda$ of height at most one. A \emph{pseudo-isomorphism} between two finitely generated torsion $\Lambda$-modules $M$ and $N$ is a $\Lambda$-homomorphism $f: M \rightarrow N$ with pseudo-null kernel and cokernel, in this case $M$ and $N$ are said to be \emph{pseudo-isomorphic} and we write $M \sim N$ (it is possible to see that it is an equivalence relation). We recall the following well-known result about the structure (up to pseudo-isomorphism) of finitely generated torsion $\Lambda$-modules.

\begin{theorem} \label{str-pseudo}
 Let $M$ be a finitely generated torsion $\Lambda$-module. Then
 \[
 M \sim \bigoplus_{j=1}^n \bigoplus_{t=1}^\ell \Lambda/\p_j^{k_{j,t}},
 \]
where $\p_j$ are distinct height-$1$ prime ideals of $\Lambda$ and $k_{j,t} \geq 0$ are integers. All these ideals and integers are uniquely determined up to order.
\end{theorem}

\begin{proof}
   See, \emph{e.g.}, \cite[(5.1.10)]{NSW}. 
\end{proof}

Denote by $I(\Lambda)$ the set of ideals of $\Lambda$. Given $I,J \in I(\Lambda)$, we write $I \prec J$ if $I \Lambda_\p \subset J \Lambda_\p$ for every height-$1$ prime ideal $\p$ of $\Lambda$. We also write $I \sim J$ if $I \prec J$ and $J \prec I$ (it is an abuse of notation but the meaning of $\sim$ will be clear from the context) and we notice that it is an equivalence relation. We denote by $[I]$ the equivalence class of $I$ modulo $\sim$.

We define an operation $\cdot$ on the set $I(\Lambda)/\sim$ of ideals of $\Lambda$ modulo $\sim$ by setting
$$[I]\cdot[J] \defeq [I \cdot J]$$
for all $I,J \in I(\Lambda)$. This operation is well-defined since $I_1 \sim J_1$ and $I_2 \sim J_2$ imply $I_i \Lambda_\p=J_i\Lambda_\p$ for all height-one primes $\p$ and $i \in \{1,2\}$, then $I_1I_2 \Lambda_\p=J_1J_2\Lambda_\p$ and finally $I_1I_2 \sim J_1J_2$.

\begin{proposition} \label{monoid}
The pair $\bigl(I(\Lambda)/\!\!\sim,\cdot\bigr)$ is a cancellative, reduced and commutative monoid with identity $[\Lambda]$ such that, for all $I,J \in I(\Lambda)$, the following properties are satisfied:
\begin{enumerate}
\item $[I]^n=[J]^n$ implies $[I]=[J]$ for all integers $n \geq 1$,
\item $[I]$ divides $[J]$ if and only if $J \prec I$.
\end{enumerate}
\end{proposition}

\begin{proof} The fact that the given pair is a commutative monoid is trivial. To see that it is reduced, assume that $[I] \cdot[J]=[\Lambda]$; this implies that $IJ\Lambda_\p=\Lambda_\p$ for all the height-$1$ prime ideals $\p$ of $\Lambda$, so $I\Lambda_\p=J\Lambda_\p=\Lambda_\p$ and then $[I]=[J]=[\Lambda]$. To see that it is cancellative, consider ideals $I$, $J_1$ and $J_2$ of $R$ such that $I \cdot J_1 \sim I \cdot J_2$. Then
\[
I\Lambda_\p J_1 \Lambda_\p = I\Lambda_\p J_2 \Lambda_\p
\]
for all height-$1$ prime ideals $\p$ of $\Lambda$. We conclude that $J_1\Lambda_\p=J_2 \Lambda_\p$ for all $\p$, so $J_1\sim J_2$.

Property (1) follows from the equalities
\[
(I\Lambda_\p)^n=I^n \Lambda_\p=J^n \Lambda_\p=(J \Lambda_\p)^n,
\]
which hold for every height-$1$ prime ideal $\p$ of $\Lambda$. Since $\Lambda_\p$ is a DVR, it follows that $I \Lambda_\p=J \Lambda_\p$ for every height-$1$ prime $\p$ of $\Lambda$, whence $I \sim J$.

As for property (2), if there exists $I' \in I(\Lambda)$ such that $J\Lambda_\p=I'\Lambda_\p \cdot I\Lambda_\p$ for all height-$1$ prime ideals $\p$ of $\Lambda$, then $J \prec I$. The converse follows from the fact that $I\Lambda_\p=\Lambda_\p$ for almost all height-$1$ prime ideals $\p$ and any non-trivial $I \in I(\Lambda)$. \end{proof}

Thanks to the previous result, we can give the following definition.

\begin{definition} \label{pseudo-def}
An ideal $J$ of $\Lambda$ \emph{admits a pseudo-square root} if there is an ideal $I$ of $\Lambda$ such that $J\sim I^2$. In this case, the \emph{pseudo-square root} of $J$ is the (uniquely determined) equivalence class $\sqrt{[J]}\defeq[I]$.
\end{definition}

We have the following results.

\begin{theorem} \label{fitt-pseudo}
 Let $M$ be a finitely generated torsion $\Lambda$-module such that
\[
M \sim \bigoplus_{j=1}^n \bigoplus_{t=1}^\ell \Lambda/\p_j^{k_{j,t}},
\]
where the sequences $(k_{j,t})_{t \geq 0}$ are non-decreasing for every $j$. Then
\[
\Fitt_i(M) \sim \prod_{j=1}^n \prod_{t=1}^{\ell-i} \p_j^{k_{j,t}}
\]
for all integers $0 \leq i \leq \ell-1$. In particular, the sequence $\{[\Fitt_i(M)] \}_{i \geq 0}$ determines uniquely the pseudo-isomorphism class of $M$.
\end{theorem}

\begin{proof}See \cite[Theorem 2.8]{DRLV}.
\end{proof}

\begin{theorem} \label{evenoddFitt}
Let $N$ and $M$ be finitely generated torsion $\Lambda$-modules and assume that $M \sim N \oplus N$. There is an equality
\[
[\Fitt_{2i+1}(M)]= \sqrt{[\Fitt_{2i}(M) \cdot \Fitt_{2i+2}(M)]}
\]
for all integers $i \geq 0$. In particular, the pseudo-isomorphism class of $M$ is uniquely determined by the sequence $\{[\Fitt_{2i}(M)] \}_{i \geq 0}$.
\end{theorem}

\begin{proof} See \cite[Proposition 2.14]{DRLV}.\end{proof}

\section{Universal Kolyvagin systems} \label{sec3}

In this section, we give results on the structure of Selmer groups of certain representations in terms of associated universal Kolyvagin systems. Let us fix 
\begin{itemize}
\item an imaginary quadratic field $K$ of discriminant $D_K$ and class number $h_K$;
\item a prime number $p\nmid2D_Kh_K$.
\end{itemize}
For all integers $n\geq1$, denote by $K[n]$ the ring class field of $K$ of conductor $n$. The generator of $\Gal(K/\Q)$ lifts to the complex conjugation $\tau_c$ fixed in \S \ref{sec:notation-and-terminology} and, for a $\Z[\frac{1}{2}][\tau_c]$-module $M$, we write $M^\pm$ for the submodules of $M$ on which $\tau_c$ acts as $\pm1$.

\subsection{Universal Kolyvagin systems}

We review the notion of universal Kolyvagin systems; see, \emph{e.g.}, \cite[\S3]{BuDef}, \cite[\S2.7]{mastella2025anticyclotomic} and \cite[IV.2]{RubinES} for details. 

\subsubsection{Selmer structures} \label{selmer-structures-subsubsec}

Let $R$ be a complete noetherian local ring with maximal ideal $\mathfrak{m}$ and finite residue field of characteristic $p$; let $T$ be a free $R$-module of rank $2$ endowed with a continuous $R$-linear $G_\Q$-action. Assume that $T$ is unramified outside $Np$ for a certain integer $N\geq1$ such that $p\nmid N$. For a finite place $v$ of $K$, denote by $K_v$ the completion of $K$ at $v$, let $K_v^{\unr}$ be the maximal unramified extension of $K_v$ and set 
\[
H^1_\unr(K_v,T)\defeq H^1\bigl(\Gal(K_v^\unr/K_v),T^{I_{v}}\bigr)
\]
for the unramified local cohomology of $T$ at $v$, where $I_{v}$ is the inertia subgroup of $G_{K_v}$. We view $H^1_\unr(K_v,T)$ as an $R$-submodule of $H^1(K_v,T)$ by means of the inflation map.

\begin{definition}\label{def-SelmerStructure}
A \emph{Selmer structure} $\mathcal{F}$ on $T$ over $K$ is a family of $R$-submodules
\[
H^1_\mathcal{F}(K_v,T) \subset H^1(K_v,T)
\]
for any prime $v$ of $K$ such that $H^1_\mathcal{F}(K_v,T)=H^1_{\unr}(K_v,T)$ for any $v \nmid Np$. The Selmer group attached to $\mathcal{F}$ is
\[
\Sel_\mathcal{F}(K,T)\defeq\ker \Biggl(H^1(K,T) \rightarrow \prod_v \frac{H^1(K_v,T)}{H^1_\mathcal{F}(K_v,T)} \Biggr).
\]
\end{definition}
Let us fix a Selmer structure $\mathcal{F}$ on $T$ over $K$, which induces Selmer structures on all the subquotients of $T$ by propagation (see, \emph{e.g.}, \cite[Example 1.1.2]{MR}); these Selmer structures will be denoted with the same letter. 

Let $\mathcal{P}$ be a set of prime numbers $\ell$ such that $\ell\nmid Np$ and $\ell$ is inert in $K$; denote by $\mathcal{N}$ the set of square-free products of primes of $\mathcal{P}$ (with $1 \in \mathcal{N}$). For each $\ell\in\mathcal{P}$, write $\lambda$ for the unique prime of $K$ lying over $\ell$ and by $\F_\lambda$ the residue field of $K$ at $\lambda$, which has $\ell^2$ elements. As explained in \S \ref{sec:notation-and-terminology}, write $\Frob_\ell$ for the Frobenius element at $\ell$, and $\Frob_\lambda=\Frob_\ell^2$ for the Frobenius element at $\lambda$. These elements depend on our choice of the embedding $\overline{\Q}\hookrightarrow \overline{\Q}_\ell$ and are defined up to elements in the absolute inertia at $\ell$, so their image in $\Aut_R(T)$ is well defined as $T$ is unramified at $\ell$.
Finally, set $G_n\defeq\Gal(K[n]/K[1])$ for every $n \in \mathcal{N}$. In particular, $G_1$ is trivial, $G_\ell$ is cyclic of order $\ell+1$ and $G_n \simeq \prod_{\ell\mid n}G_\ell$, where the product varies among all primes $\ell$ dividing $n$.

\subsubsection{Local conditions} \label{local-subsubsec}

Let $a,b,c$ be pairwise coprime integers in $\mathcal{N}$. Starting from $\mathcal{F}$, define a modified Selmer structure on $T$ over $K$ as follows:
\[
H^1_{\mathcal{F}^a_b(c)}(K_v,T)\defeq\begin{cases}H^1_\mathcal{F}(K_v,T) & \text{if $v\nmid abc$};\\[2mm]H^1_{\tr}(K_v,T) & \text{if $v\,|\,c$};\\[2mm]
H^1_{\rel}(K_v,T) & \text{if $v\,|\,a$};\\[2mm]
H^1_{\str}(K_v,T) & \text{if $v\,|\,b$}.\end{cases}
\]
Here $H^1_{\rel}(K_v,T)\defeq H^1(K_v,T)$ is the \emph{relaxed} local condition,
$H^1_{\str}(K_v,T)\defeq\{0\}$ is the \emph{strict} local condition and $H^1_{\tr}(K_v,T)$ is the \emph{transverse} local condition, \emph{i.e.}, the kernel of restriction $H^1(K_v,T)\rightarrow H^1(K^{\tr}_v,T)$, where $K^{\tr}_v$ is the maximal totally tamely ramified abelian $p$-extension of $K_v$ contained in the completion of $K[\ell]$ at the prime above $v=(\ell)$ that is determined by the fixed embedding $\bar{\Q}\hookrightarrow\bar{\Q}_\ell$. Recall that $v$ totally splits in $K[1]$ and each prime of $K[1]$ above it is totally ramified in $K[\ell]$. Finally, if some of $a,b,c$ is $1$, then we omit the symbol from the notation 
(\emph{e.g.}, we simply write $H^1_{\mathcal{F}_b(c)}$ for $H^1_{\mathcal{F}^1_b(c)}$). 

\subsubsection{Kolyvagin systems}

In the following lines, let $R$ be an artinian local ring and denote by $K(T)$ the extension of $K$ fixed by the kernel of the $G_K$-representation attached to $T$.

\begin{assumption} \label{ass:kolyvagin-primes}
Every $\ell\in\mathcal{P}$ satisfies the following conditions:
\begin{itemize} 
\item $\Frob_\ell$ is conjugate to the complex conjugation $\tau_c$ in $\Gal\bigl(K(T)/\Q\bigr)$;
\item $(\ell+1)\cdot T=\{0\}$.
\end{itemize}
\end{assumption}

\begin{remark}
Let $p^s$ be the smallest power of $p$ such that $p^sT=\{0\}$. If the determinants of $\Frob_\ell$ and $\tau_c$ acting on $T$ are $\ell$ and $-1$, respectively, then the first condition in Assumption \ref{ass:kolyvagin-primes} implies $p^s\,|\,\ell+1$. Therefore, in this case, the second condition in Assumption \ref{ass:kolyvagin-primes} is automatically verified. This happens, \emph{e.g.}, when $T=E[p^s]$ is the $\Z/p^s\Z$-module of $p^s$-torsion points of an elliptic curve $E$ defined over $\Q$.
\end{remark}

Let $(n,\ell)\in\mathcal{N}\times\mathcal{P}$ with $\ell \nmid n$ and let $\lambda$ be the unique prime of $K$ above $\ell$. Thanks to the first condition in Assumption \ref{ass:kolyvagin-primes}, we know that $\Frob_\lambda$ acts trivially on $T$. Therefore, by \cite[Lemma 2.8]{mastella2025anticyclotomic}, there are canonical isomorphisms
\begin{equation} \label{eq:finite-isomorphism}
\alpha_\ell:H^1_{\f}(K_\lambda,T) \xlongrightarrow{\simeq} T,
\end{equation}
where $H^1_{\f}(K_\lambda,T)=H^1_{\mathrm{unr}}(K_\lambda,T)$ denotes unramified cohomology at unramified primes, and  
\[
\beta_\ell:H^1_{\s}(K_\lambda,T) \otimes G_\ell \xlongrightarrow{\simeq} T,
\]
where 
\[
H^1_{\s}(K_\lambda,T)\defeq H^1(K_\lambda,T)\big/H^1_{\f}(K_\lambda,T)\simeq H^1(I_\lambda,T)^{\Gal(K_\lambda^\mathrm{unr}/K_\lambda)}
\]
and $I_\lambda$ is the inertia subgroup of $G_{K_\lambda}$. 
Following \cite[\S2.7]{mastella2025anticyclotomic}, we also fix an $R$-automorphism $\chi_{n,\ell}:T \overset\simeq\longrightarrow T$ and use it to define modified finite-singular isomorphisms at $\ell$ as
\[
\phi_\ell^{\fs}(\chi_{n,\ell}):H^1_{\f}(K_\lambda,T_s)\overset{\alpha_\ell}\longrightarrow T \xrightarrow{\chi_{n,\ell}} T \xrightarrow{\beta_\ell^{-1}} H^1_{\s}(K_\lambda,T) \otimes G_\ell.
\]
We also introduce maps
\begin{itemize}
\item $\psi_{n \ell}^{\ell}:\Sel_{\mathcal{F}(n\ell)}(K,T) \otimes G_{n \ell} \xrightarrow{\loc_\ell} H^1(K_\lambda,T) \otimes G_{n \ell} \longrightarrow H^1_{\s}(K_\lambda,T) \otimes G_{n\ell}$;
\item $\psi_n^\ell(\chi_{n,\ell}):\Sel_{\mathcal{F}(n)}(K,T) \otimes G_n \xrightarrow{\loc_\ell} H^1_{\f}(K_\lambda,T) \otimes G_n \xrightarrow{\phi_\ell^{\fs}(\chi_{n, \ell})} H^1_{\s}(K_\lambda,T) \otimes G_{n\ell}$.
\end{itemize}

\begin{definition}\label{defKL}
A \textit{$\{\chi_{n,\ell} \}$-Kolyvagin system} for $(T,\mathcal{F},\mathcal{P})$ is a collection of cohomology classes ${\{\kappa_{n}\}}_{n\in\mathcal{N}}$ such that
\begin{itemize}
\item $\kappa_{n}\in\Sel_{\mathcal{F}(n)}(K,T)\otimes G_n$;
\item $\bigl(\psi_n^\ell(\chi_{n,\ell})\bigr)(\kappa_{n})=\psi_{n \ell}^\ell (\kappa_{n \ell})$
\end{itemize}
for all pairs $(n,\ell)\in\mathcal{N}\times\mathcal{P}$ with $\ell\nmid n$.
\end{definition}

Denote by $\ks\bigl(T,\mathcal{F},\mathcal{P},\{\chi_{n,\ell}\}\bigr)$ the $R$-module of $\{\chi_{n,\ell}\}$-Kolyvagin systems for $(T,\mathcal{F},\mathcal{P})$.

\subsubsection{Universal Kolyvagin systems} \label{secUKS}

Let $R$ be a complete noetherian local ring with maximal ideal $\mathfrak{m}$ and finite residue field of characteristic $p$. For any integer $s\geq1$, set $R_{s}\defeq R/\mathfrak{m}^s$ and $T_s\defeq T/\mathfrak{m}^sT$. Denote by $\mathcal{P}_s$ the maximal subset of $\mathcal{P}$ such that Assumption \ref{ass:kolyvagin-primes} is satisfied by the pair $(T_s,\mathcal{P}_s)$. Let also $\mathcal{N}_s$ be the subset of $\mathcal{N}$ consisting of 1 and of the square-free integers divisible only by primes contained in $\mathcal{P}_s$.

Now let $n \in \mathcal{N}_s$ and $\ell \in \mathcal{P}_s$ with $\ell\nmid n$. By \cite[Lemma 1.2.1]{MR}, isomorphisms $\alpha_\ell$ and $\beta_\ell$ for the representations $T_{s'}$ with $1\leq s'\leq s$ are functorial with respect to the canonical projection maps $T_s \twoheadrightarrow T_{s'}$. For each pair $(n,\ell)\in\mathcal{N}\times\mathcal{P}$ with $\ell\nmid n$, fix an $R$-linear automorphism $\chi_{n,\ell}:T \xrightarrow\simeq T$ and notice that it induces automorphisms $\chi_{n,\ell}:T_s \xrightarrow\simeq T_s$ for all integers $s\geq1$.

\begin{definition} \label{defUKS}
A \emph{universal $\{\chi_{n,\ell}\}$-Kolyvagin system} $\kappa$ for the triple $(T,\mathcal{F},\mathcal{P})$ is an element of the $R$-module
\[
\uks\bigl(T,\mathcal{F},\mathcal{P},\{ \chi_{n, \ell}\}\bigr) \defeq \varprojlim_s\Biggl(\varinjlim_{t\geq s}\ks\bigl(T_s,\mathcal{F},\mathcal{P}_t,\{ \chi_{n,\ell}\}\bigr)\!\Biggr).
\]
\end{definition}

\begin{remark} \label{remnot}
Let us specify the maps involved in Definition \ref{defUKS}. To begin with, observe that for $t'\geq t\geq s$ there is an inclusion $\mathcal{P}_{t'}\subset\mathcal{P}_{t}$. Given $\kappa\in\ks\bigl(T_{s},\mathcal{F},\mathcal{P}_{t},\{\chi_{n,\ell}\}\bigr)$, we may define an element $\kappa'\in\ks\bigl(T_{s},\mathcal{F},\mathcal{P}_{t'},\{\chi_{n,\ell}\}\bigr)$ by $\kappa'_n=\kappa_n$ for all $n\in\mathcal{N}_{t'}$: the association $\kappa\mapsto\kappa'$ then defines a map 
\[
F_{t,t'}^{(s)}:\ks\bigl(T_s,\mathcal{F},\mathcal{P}_{t},\{\chi_{n,\ell}\}\bigr)\longrightarrow \ks\bigl(T_s,\mathcal{F},\mathcal{P}_{t'},\{\chi_{n,\ell}\}\bigr),
\]
and the data 
\[
\Bigl(\Bigl\{\ks\bigl(T_s,\mathcal{F},\mathcal{P}_{t},\{\chi_{n,\ell}\}\bigr)\Bigr\}_{t\geq s},\Bigl\{F_{t,t'}^{(s)}\Bigr\}_{t'\geq t\geq s}\Bigr)
\]
give the direct system in $t\geq s$ in Definition \ref{defUKS}. On the other hand, if $s'\geq s$, then the canonical projection $T_{s'}\twoheadrightarrow T_s$ defines a map 
\[
P_{s',s}^{(t)}:\ks\bigl(T_{s'},\mathcal{F},\mathcal{P}_t,\{\chi_{n,\ell}\}\bigr)\longrightarrow\ks\bigl(T_{s},\mathcal{F},\mathcal{P}_t,\{\chi_{n,\ell}\}\bigr)
\]
for every $t\geq s'$, and then also a map
\[
P_{s',s}:\varinjlim_{t\geq s'}\ks\bigl(T_{s'},\mathcal{F},\mathcal{P}_t,\{\chi_{n,\ell}\}\bigr)\longrightarrow\varinjlim_{t\geq s}\ks\bigl(T_{s},\mathcal{F},\mathcal{P}_t,\{\chi_{n,\ell}\}\bigr);
\]
the data
\[
\biggl(\biggl\{\varinjlim_{t\geq s}\ks\bigl(T_s,\mathcal{F},\mathcal{P}_{t},\{\chi_{n,\ell}\}\bigr)\biggr\}_{s\geq1},\bigl\{P_{s',s}\bigr\}_{s'\geq s}\biggr)
\]
yield the inverse system in $s$ in Definition \ref{defUKS}. 
\end{remark}

Given $\kappa\in\uks\bigl(T,\mathcal{F},\mathcal{P},\{\chi_{n,\ell}\}\bigr)$, denote by 
$\kappa^{(s)}$ its natural projection to
\[
\varinjlim_{t\geq s}\ks\bigl(T_s,\mathcal{F},\mathcal{P}_t,\{\chi_{n,\ell}\}\bigr)
\]
and by $\mathfrak{K}_s$ its set of representatives in
$\bigsqcup_{t\geq s}\ks\bigl(T_s,\mathcal{F},\mathcal{P}_t,\{\chi_{n,\ell}\}\bigr)$. We choose $\kappa^{(s,t)}\in \mathfrak{K}_{s}$ for all integers $s \geq 1$ 
and we set
\[
\kappa_1\defeq\varprojlim_s\kappa^{(s,t)}_{1}\in\varprojlim_s H^1(K,T_s)\simeq H^1(K,T), 
\]
where $\kappa_1$ is well-defined since $\kappa^{(s,t)}_1$ does not depend on the choice of a representative in $\mathfrak{K}_s$. 
When $\chi_{n,\ell}$, $\mathcal{F}$ or $\mathcal{P}$ are understood, we omit the symbols from the notation; in particular, we 
write $\uks(T)$ or $\uks(T,\mathcal{P})$ 
for $\uks\bigl(T,\mathcal{F},\mathcal{P},\{\chi_{n,\ell}\}\bigr)$, 
and similarly we write $\ks(T_s)$ for $\varinjlim_{t\geq s}\ks\bigl(T_s,\mathcal{F},\mathcal{P}_t,\{\chi_{n,\ell}\}\bigr)$ and $\ks(T_s,\mathcal{P}_t)$ for $\ks\bigl(T_s,\mathcal{F},\mathcal{P}_t,\{\chi_{n,\ell}\}\bigr)$.

\subsection{Assumptions}

Let $v$ be a place of $K$ and see $T$ as an $R[[G_{K_v}]]$-module. The \textit{quotient category} $\mathrm{Quot}(T)$ of $T$ is the category whose objects are quotients $T/IT$ of $T$ by ideals $I$ of $R$ and the morphisms from $T/IT$ to $T/JT$ are the maps induced by scalar multiplications $r\in R$ with $rI\subset J$. As in \cite[Definition 1.1.2]{HoHeeg}, a local condition $\mathcal{G}$ on $T$ at $v$ is \emph{cartesian} if for any injective morphism $\alpha : T/IT\to T/JT$ in $\mathrm{Quot}(T)$ the local condition $\mathcal{G}$ on $T/IT$ is the same as the local condition obtained by propagating $\mathcal{G}$ from $T/JT$ to $T/IT$.

As before, denote by $\mathcal{F}$ a Selmer structure on $T$ over $K$ and by $\mathcal{P}$ a set of prime numbers satisfying Assumption \ref{ass:kolyvagin-primes}. When $R$ is a principal artinian local ring or a discrete valuation ring, we work under the following assumption on $(T,\mathcal{F},\mathcal{P})$, in which the first four conditions correspond to hypotheses H.1), H.2), H.3), H.4) of \cite[\S1.3]{HoHeeg}, whereas H.5) in \cite[\S1.3]{HoHeeg} is replaced by a stronger one, using the fact that $T$ is a representation of $G_\Q$. While conditions (1)--(5) are of a technical nature, the last one is fundamental for applications of Kolyvagin systems. Recall that we have fixed a complex conjugation $\tau_c\in G_\Q$. 

\begin{assumption}\label{ass:Howard-assumption} The $R[[G_\Q]]$-module $T$ satisfies the following conditions:
\begin{enumerate}
\item $T/\mathfrak{m}T$ is absolutely irreducible over $R/\mathfrak{m}$ as a representation of $G_K$; 
\item there is a finite Galois extension $F$ of $\Q$ containing $K$ such that $G_F$ acts trivially on $T$ and $H^1\bigl(F(\Bmu_{p^\infty})/K),T/\mathfrak{m}T\bigr)=0$; 
\item the local condition $\mathcal{F}$ is cartesian at $Np$ in the category $\mathrm{Quot}(T)$;
\item there is a perfect, skew-symmetric, $G_\Q$-equivariant and $R$-bilinear pairing
\[
(\cdot,\cdot):T\times T\longrightarrow R(1)
\]
for which $\mathcal{F}$ is self-dual;
\item[(5.a)] each of the eigenspaces $T^\pm$ of $\tau_c:T\to T$ is free of rank $1$ over $R$ and there is a splitting $T=T^+\oplus T^-$;
\item[(5.b)] the local condition $\mathcal{F}$ at any prime $v$ of $K$ is stable under the action of $\tau_c$; 
\item[(6)] for every $c^\pm\in H^1(K,T/\mathfrak{m}T)^\pm\smallsetminus\{0\}$ and all integers $s\geq1$, there are infinitely many $\ell\in\mathcal{P}_s$ such that $\loc_\ell(c^\pm)\ne 0$.
\end{enumerate}
\end{assumption}

\subsection{Kolyvagin systems over principal artinian local rings}\label{sec:artinian-rings}

In this subsection, let $R$ be a principal artinian local ring of length $k \geq 1$ and fix a generator $\pi$ of its maximal ideal $\mathfrak{m}$. Fix a set $\mathcal{P}$ of prime numbers that do not divide $Np$, are inert in $K$ and satisfy Assumption \ref{ass:kolyvagin-primes}. Fix also a Selmer structure $\mathcal{F}$ on $T$ over $K$ and suppose that the triple $(T,\mathcal{F},\mathcal{P})$ satisfies Assumption \ref{ass:Howard-assumption}. Thanks to \cite[Lemma 1.5.1]{HoHeeg}, the triple $\bigl(T,\mathcal{F}^a_b(n),\mathcal{P}\bigr)$ satisfies Assumption \ref{ass:Howard-assumption} for all pairwise coprime integers $a,b,n\in\mathcal{N}$. 

\subsubsection{Cyclic summands of Selmer groups}

In the following, for any group of type $\Sel_\bullet(K,\bullet)$, we denote by $\Sel_\bullet^\pm(K,\bullet)$ its submodules on which $\tau_c$ acts as $\pm 1$.

\begin{lemma} \label{InfPrimes} 
Let $n\in\mathcal{N}$ and let $C^{\pm}$ be a cyclic summand of maximal length of $\Sel_{\mathcal{F}(n)}^\pm(K,T)$. If $\Sel_{\mathcal{F}(n)}(K,T)\neq 0$, then there are infinitely many 
    $\ell \in \mathcal{P}$ such that the restriction of
$
\loc_\ell:\Sel_{\mathcal{F}(n)}(K,T) \rightarrow H^1_{\f}(K_\lambda,T)
$
to $C^\pm$ is injective. 
\end{lemma}

\begin{proof} We assume that $\Sel_{\mathcal{F}(n)}^\pm(K,T)$ are both non-zero, since the other cases are simpler. Let $C^\pm\subset \Sel_{\mathcal{F}(n)}^\pm(K,T)$ be two cyclic components of maximal length, say $l^\pm$, and choose $R$-generators $c^\pm$. Then $d^\pm=\pi^{l^\pm-1}c^\pm$ are non-zero elements of
$\Sel_{\mathcal{F}(n)}^\pm(K,T)[\mathfrak{m}]$ and it suffices to prove that $\loc_\ell(d^\pm)\neq 0$. Since $\Sel_{\mathcal{F}(n)}^\pm(K,T)[\mathfrak{m}]\simeq \Sel_{\mathcal{F}(n)}^\pm(K,T/\mathfrak{m}T)$ by \cite[Lemma 1.3.3]{HoHeeg} (the isomorphism in \cite{HoHeeg} is $\tau_c$-equivariant), the images $\bar{d}^\pm$ of $d^\pm$ in $\Sel_{\mathcal{F}(n)}(K,T/\mathfrak{m}T)$ are both non-zero. By condition (6) in Assumption \ref{ass:Howard-assumption}, there are infinitely many $\ell\in\mathcal{P}$ such that $\loc_\ell\bigl(\bar{d}^\pm\bigr)\neq0$, and the conclusion follows. \end{proof}

If $\lambda$ is a prime of $K$ whose residual characteristic belongs to $\mathcal{P}$, then $H^1_{\tr}(K_\lambda,T)$ projects isomorphically onto $H^1_\mathrm{s}(K_\lambda,T)$ 
and there is a splitting $H^1(K_\lambda,T)\simeq H^1_{\f}(K_\lambda,T)\oplus H^1_\mathrm{s}(K_\lambda,T)$ (see, \emph{e.g.}, \cite[Proposition 1.1.9]{HoHeeg}). 

\subsubsection{Results on Selmer groups} \label{preliminary-subsubsec}

We collect some preliminary results on Selmer groups.

\begin{proposition}\label{HoStrArt}
For all $n\in\mathcal{N}$, there are an $R$-module $M_n$ and an isomorphism
\[
\Sel_{\mathcal{F}(n)}(K,T)\simeq R^{e(n)}\oplus M _n\oplus M_n
\]
for some $e(n)\in\{0,1\}$.
\end{proposition}

\begin{proof} See, \emph{e.g.}, \cite[Theorem 1.4.2]{HoHeeg}. \end{proof}

For all $n\in\mathcal{N}$, set
\begin{itemize}
\item $\lambda_n\defeq\length_R(M_n)$;
\item $\rho(n)^\pm\defeq\dim_{R/\mathfrak{m}}\bigl(\Sel_{\mathcal{F}(n)}^\pm(K,T/\mathfrak{m}T)\bigr)=\dim_{R/\mathfrak{m}}\bigl(\Sel_{\mathcal{F}(n)}^\pm(K,T)[\mathfrak{m}]\bigr)$;
\item $\rho(n)\defeq\rho(n)^+ +\rho(n)^-=\dim_{R/\mathfrak{m}}\bigl(\Sel_{\mathcal{F}(n)}(K,T/\mathfrak{m}T)\bigr)=\dim_{R/\mathfrak{m}}\bigl(\Sel_{\mathcal{F}(n)}(K,T)[\mathfrak{m}]\bigr)$,
\end{itemize}
where the rightmost equalities in the last two points stem from \cite[Lemma 1.3.3]{HoHeeg}.

\begin{proposition} \label{SelIncl}
Let $n \in \mathcal{N}$ and $\ell \in \mathcal{P}$ with $\ell \nmid n$. There are integers $a,b,\delta \geq 0$ such that in the diagram of inclusions
\begin{equation} \label{diagram-inclusions-eq}
\begin{tikzcd}
 & \Sel_{\mathcal{F}^\ell(n)}(K,T) & \\
 \Sel_{\mathcal{F}(n)}(K,T) \arrow[ur,hook,"k-a{,}\:k-b"] & & \Sel_{\mathcal{F}(n\ell)}(K,T) \arrow[ul,hook,"a+\delta{,}\:b+\delta"] \\
 & \Sel_{\mathcal{F}_\ell(n)}(K,T) \arrow[ul,hook,"a{,}\:b"] \arrow[ur,hook,"k-a-\delta{,}\:k-b-\delta"] &
\end{tikzcd}
\end{equation}
the cokernel of each map is a direct sum of two cyclic $R$-modules of the indicated lengths.
\end{proposition}

\begin{proof} See, \emph{e.g.}, \cite[Lemma 1.5.8]{HoHeeg}. \end{proof}

The next two lemmas deal with eigenspaces for complex conjugation.

\begin{lemma} \label{Ho1.5.3} 
Suppose that $n\ell\in \mathcal{N}$ with $\ell$ a prime number. 
\begin{enumerate}
\item If $\loc_\ell\bigl(\Sel_{\mathcal{F}(n)}^\pm(K,T/\mathfrak{m}T)\bigr)\neq0$, then $\rho(n\ell)^\pm=\rho(n)^\pm-1$ and 
\[
\loc_\ell\bigl(\Sel_{\mathcal{F}(n\ell)}^\pm(K,T/\mathfrak{m}T)\bigr)=0.
\]
\item If $\loc_\ell\bigl(\Sel_{\mathcal{F}(n)}^\pm(K,T/\mathfrak{m}T)\bigr)=0$, then $\rho(n\ell)^\pm=\rho(n)^\pm+1$. 
\item The integer $\rho(n)^\pm$ is even if and only if $\rho(n\ell)^\pm$ is odd.  
\end{enumerate}
In particular, $e(n)$ in Proposition \ref{HoStrArt} is independent of $n \in \mathcal{N}$.
\end{lemma}

\begin{proof} Parts (1) and (2) are \cite[Lemma 1.5.3 and Proposition 1.5.5]{HoHeeg}, 
while part (3) is an immediate consequence of them. \end{proof}

\begin{lemma} \label{lemmaH1f}
For each prime $\lambda$ of $K$ whose residual characteristic lies in $\mathcal{P}$, the $R$-module $H^1_{\f}(K_\lambda,T)$ is free of rank $2$ and each of the two eigenspaces $H^1_{\f}(K_\lambda,T)^\pm$ for the action of $\tau_c$ is free of rank $1$ over $R$.
\end{lemma}

\begin{proof} As is shown in \cite[Lemma 2.11]{mastella2025anticyclotomic}, isomorphism $\alpha_\ell$ in \eqref{eq:finite-isomorphism} is $\tau_c$-equivariant. Therefore, the conclusion follows from condition (5.a) in Assumption \ref{ass:Howard-assumption}. \end{proof}

\subsubsection{Free Kolyvagin systems} \label{free-subsubsec}

Given $\epsilon\in\{\pm1\}$, for every $n\in\mathcal{N}$ set 
\[
\epsilon_n\defeq\epsilon\cdot(-1)^{\nu(n)}\in\{\pm1\},
\]
where $\nu(n)$ is the number of distinct primes dividing $n$.

The following notion is inspired by \cite[Definition 2.3.6]{howard2012bipartite}.

\begin{definition}
A Kolyvagin system $\kappa$ for $\bigl(T,\mathcal{F},\mathcal{P},\{\chi_{n,\ell}\}\bigr)$ is \emph{free} if for every $n\in\mathcal{N}$ there is a rank-$1$ free $R$-submodule $C_n$ of $\Sel_{\mathcal{F}(n)}(K,T)$ such that $\kappa_n\in C_n$. Moreover, $\kappa$ has \emph{sign $\epsilon$} if $\kappa_n\in \Sel^{\epsilon_n}_{\mathcal{F}(n)}(K,T)$.
\end{definition}

As an immediate consequence of this definition, one shows that if $\kappa\in\ks\bigl(T,\mathcal{F},\mathcal{P},\{ \chi_{n,\ell}\}\bigr)$ is a free Kolyvagin system of sign $\epsilon$ and $\kappa_n\ne0$ for some $n\in\mathcal{N}$, then one can find a rank-$1$ free $R$-submodule $C_n$ of $\Sel_{\mathcal{F}(n)}^{\epsilon_n}(K,T)$ such that $\kappa_n\in C_n$.

The existence of a non-trivial free Kolyvagin system of sign $\epsilon\in\{\pm1\}$ leads to the following improvement on Proposition \ref{HoStrArt}.

\begin{proposition}\label{prop-structureart}
Suppose that there exists a non-trivial, free $\kappa\in\ks\bigl(T,\mathcal{F},\mathcal{P},\{ \chi_{n,\ell}\}\bigr)$ of sign $\epsilon$. Then there is a $\tau_c$-equivariant decomposition
\[
\Sel_{\mathcal{F}(n)}(K,T)\simeq R\oplus M_n\oplus M_n
\]
with $R\subset \Sel_{\mathcal{F}(n)}^{\epsilon_n}(K,T)$. Moreover, $\rho(n)^{\epsilon_n}$ is odd and $\rho(n)^{-\epsilon_n}$ is even. 
\end{proposition}

\begin{proof} For every $n\in\mathcal{N}$, let $\Sel_{\mathcal{F}(n)}(K,T)\simeq R^{e(n)}\oplus M_n\oplus M_n$ be any decomposition coming from Proposition \ref{HoStrArt}. We divide the proof into four steps. 

\texttt{Step 1.} We claim that $e(n)=1$ for all $n\in\mathcal N$; note that, by Lemma \ref{Ho1.5.3}, it suffices to prove this equality for a single $n\in\mathcal N$. Fix an $n\in\mathcal N$ such that $\kappa_n\not=0$, which exists because $\kappa$ is not trivial, and recall that $\lambda_n=\length(M_n)$.

If $\lambda_n=0$, then $\Sel_{\mathcal{F}(n)}(K,T) \simeq R^{e(n)}$ and the non-triviality of $\kappa_n$ implies $e(n)=1$. Thus, if $\lambda_n>0$, then to prove the claim it suffices to show that there exists $n'\in\mathcal{N}$ with $\kappa_{n'}\neq0$ and $\lambda_{n'}<\lambda_n$. Since $\kappa$ is free of sign $\epsilon$ and $\kappa_n\ne 0$, we can fix a free $R$-submodule $C_n\subset\Sel^{\epsilon_n}_{\mathcal{F}(n)}(K,T)$ of rank $1$ such that $\kappa_n \in C_n$. 
We also fix $D_n$ to be a cyclic direct summand of $\Sel^{-\epsilon_n}_{\mathcal{F}(n)}(K,T)$ of maximal length, say $d$, and we set $E_n \defeq C_n \oplus D_n$. By Lemma \ref{InfPrimes}, we choose a prime $\ell \in \mathcal{P}$ such that the restriction of $\loc_\ell$ to $E_n$ is injective. Combining the exact sequences
\[
0 \longrightarrow \Sel^{\pm \epsilon_n}_{\mathcal{F}_\ell(n)}(K,T) \longrightarrow \Sel^{\pm \epsilon_n}_{\mathcal{F}(n)}(K,T) \longrightarrow H^1_{\f}(K_\lambda,T)^{\pm \epsilon_n}
\]
with Lemma \ref{lemmaH1f} shows that  $a=k$, $b=d$, $\delta=0$ in diagram \eqref{SelIncl}. Furthermore, by comparing explicitly the lengths of $\Sel_{\mathcal{F}(n)}(K,T)$ and $\Sel_{\mathcal{F}(n\ell)}(K,T)$ using again Proposition \ref{SelIncl}, we obtain the equality 
\begin{equation}\label{eq:relation-lambda-n-nl}
    \lambda_{n\ell}=\lambda_n-d.
\end{equation}
Since $\loc_\ell(\kappa_n) \neq0$, we have 
\begin{equation}\label{reclaw}
\psi_{n \ell}^\ell(\kappa_{n\ell})=\psi_n^\ell(\chi_{n,\ell})(\kappa_n)\neq 0\end{equation}
and, in particular, $\kappa_{n\ell}\neq0$. We distinguish two cases.

\texttt{Case 1: $d>0$.} We immediately deduce from equality \eqref{eq:relation-lambda-n-nl} that $\lambda_{n\ell}<\lambda_n$, so we can take $n'\defeq n\ell$.

\texttt{Case 2: $d=0$.} In this case, equality \eqref{eq:relation-lambda-n-nl} gives $\lambda_{n \ell}=\lambda_n$ and, by Lemma \ref{Ho1.5.3}, we have $\rho(n\ell)^{\epsilon_n}=\rho(n)^{\epsilon_n}-1>0$, the inequality being a consequence of $\lambda_n>0$. Since $\epsilon_{n\ell}=-\epsilon_n$, we can fix a rank-$1$ free $R$-module $C_{n\ell}\subset\Sel^{-\epsilon_{n}}_{\mathcal{F}(n \ell)}(K,T)$ such that $\kappa_{n \ell} \in C_{n \ell}$ and a cyclic direct summand $D_{n \ell}$ of maximal length $d'$ of $\Sel^{\epsilon_n}_{\mathcal{F}(n\ell)}(K,T)$. Set $E_{n\ell} \defeq C_{n\ell} \oplus D_{n \ell}$ and choose a prime $\ell' \in \mathcal{P}$ such that the restriction of $\loc_{\ell'}$ to $E_{n \ell}$ is injective. In this case $d'>0$, so we can take $n'\defeq n\ell\ell'$.

\texttt{Step 2.} Keep $n\in\mathcal{N}$ such that $\kappa_n\ne 0$. We prove that $\rho(n)^{\epsilon_n}$ is odd and $\rho(n)^{-\epsilon_n}$ is even; by Lemma \ref{Ho1.5.3}, the same will be true for every $n\in\mathcal{N}$. \texttt{Step 1} implies that there is an isomorphism
\[
\Sel_{\mathcal{F}(n)}(K,T) \simeq R \oplus M_n \oplus M_n,
\]
from which it follows that $\rho(n)=\rho(n)^{\epsilon_n}+\rho(n)^{-\epsilon_n}$ is odd. Assume, by contradiction, that $\rho(n)^{\epsilon_n}$ is even (and, necessarily, larger than $0$). Combining the freeness of $\kappa$ with Lemma \ref{InfPrimes}, we can choose distinct primes $\ell,\ell'\in\mathcal{P}$ and rank-$1$ free $R$-modules $C_n\subset \Sel_{\mathcal{F}(n)}^{\epsilon_n}(K,T)$, $C_{n\ell}\subset \Sel_{\mathcal{F}(n\ell)}^{-\epsilon_{n}}(K,T)$ such that
\begin{itemize}
\item $\kappa_\star\in C_\star$ for $\star\in\{n,n\ell\}$;
\item $\loc_\ell$ (respectively, $\loc_{\ell'}$) is injective on $C_n$ (respectively, $C_{n\ell}$).
\end{itemize}
In addition, we can also fix a cyclic direct summand $D_{n\ell}$ of $\Sel_{\mathcal{F}(n\ell)}^{\epsilon_{n}}(K,T)$ of maximal length and assume that $\loc_{\ell'}$ is injective on $D_{n\ell}$. Repeating the same argument of \eqref{reclaw} two times, we obtain $\kappa_{n\ell\ell'}\ne 0$. By Lemma \ref{Ho1.5.3}, $\rho(n\ell)^{\epsilon_n}=\rho(n)^{\epsilon_n}-1$, and this quantity is positive since $\rho(n)^{\epsilon_n}\ge 2$. Therefore, $D_{n\ell}$ is non-trivial and, again by Lemma \ref{Ho1.5.3}, we obtain that $\rho(n\ell\ell')^{\epsilon_n}=\rho(n)^{\epsilon_n}-2$. Therefore, as outputs we obtain
\begin{itemize}
\item[(a)] $\kappa_{n\ell\ell'}\ne 0$;
\item[(b)] $\rho(n\ell\ell')^{\epsilon_n}=\rho(n)^{\epsilon_n}-2$.
\end{itemize}
An inductive argument shows that we can iterate the procedure above (adding a new couple of primes at each step) until the quantity in (b) becomes $0$. This, however, contradicts the non-triviality of the Kolyvagin class appearing in (a), yielding the claim.

\texttt{Step 3.} Setting $\Sel_n\defeq\Sel_{\mathcal{F}(n)}(K,T)$ and $\Sel_n^\pm\defeq\Sel_{\mathcal{F}(n)}^\pm(K,T)$ to simplify the notation, for each $s\in\N$ define, as in the proof of \cite[Theorem 1.4.2]{HoHeeg}, the $R/\mathfrak{m}$-vector spaces 
\[
V_s\defeq\Sel_n[\mathfrak{m}^s]\big/\mathfrak{m}\Sel_n[\mathfrak{m}^{s+1}],\quad V_s^\pm\defeq\Sel_n^\pm[\mathfrak{m}^s]\big/\mathfrak{m}\Sel_n^\pm[\mathfrak{m}^{s+1}].
\]
We claim that for each $s\in\{0,\dots,k-1\}$ the dimension of the $R/\mathfrak{m}$-vector space $V_s^\pm$ is even, equal to the number of cyclic components of length smaller or equal to $s$ in $\Sel_n^\pm$. As in the proof of \cite[Theorem 1.4.2]{HoHeeg} for $V_s$, this follows from the existence of a non-degenerate, alternating pairing 
\begin{equation} \label{FH-eq}
\langle\cdot,\cdot\rangle:V_s^\pm\times V_s^\pm\longrightarrow R[\mathfrak{m}]
\end{equation}
whose right kernel is $V_{s+1}^\pm$. The pairing in \eqref{FH-eq} is the restriction of the Flach--Howard pairing 
\begin{equation} \label{FH-eq2}
\langle\cdot,\cdot\rangle:V_s\times V_s\longrightarrow R[\mathfrak{m}]
\end{equation}
that, following the recipe in \cite{Flach}, was built in \cite[\S1.4]{HoHeeg}; the right kernel of \eqref{FH-eq2} is $V_{s+1}$. We need to show that \eqref{FH-eq2} induces a non-degenerate pairing on $V_s^\pm$. Using the fact that the local invariant map is equivariant for the action of $\Gal(K_v/\Q_\ell)$ for all $v\,|\,\ell$, and comparing with the definition of \eqref{FH-eq2} in \cite{HoHeeg}, we get an equality 
\[
\big\langle\tau_v(a),\tau_v(b)\big\rangle=\langle a,b\rangle
\]
for all $a,b\in V_s$, where $\tau_v$ is the generator of $\Gal(K_v/\Q_\ell)$ (in other words, either $\tau_v$ is trivial or $\tau_v$ coincides with $\tau_c$). For $a\in V_s^\pm$ and $b\in V_s^\mp$, we have 
\[
\langle\tau_c(a),\tau_c(b)\rangle = -\langle a,b\rangle,
\]
whence $\langle a,b\rangle=0$. Therefore, the pairing in \eqref{FH-eq2} induces a perfect, alternating pairing as in \eqref{FH-eq} whose right kernel is $V_{s+1}\cap V_s^\pm=V_{s+1}^\pm$. 

\texttt{Step 4.} By the previous steps, we know that, for all $n\in\mathcal{N}$, there is an isomorphism
\[
\Sel^\pm_{\mathcal{F}(n)}(K,T)\simeq R^{r_n^\pm}\oplus N_n^\pm\oplus N_n^\pm
\]  
for an $R$-module $N_n$ with direct factors of length at most $k-1$ and positive integers $r_n^\pm$ with $r_n^{\epsilon_n}$ odd and $r_n^{-\epsilon_n}$ even. By taking 
\[
M_n\defeq\Bigl(R^{\frac{r^{\epsilon_n}_n-1}{2}}\oplus N_n^{\epsilon_n}\Bigr)\oplus\Bigl(R^{\frac{r^{-\epsilon_n}_n}{2}}\oplus N_n^{-\epsilon_n}\Bigr),
\]
we obtain a splitting of $\Sel_{\mathcal{F}(n)}(K,T)$ of the desired form. \end{proof}

Before proving other properties of free Kolyvagin systems, we need two auxiliary results in commutative algebra.

\begin{lemma} \label{expSub}
Let $A\simeq\bigoplus_{i=0}^n R/\mathfrak{m}^{a_i}$ and $B\simeq\bigoplus_{i=0}^n R/\mathfrak{m}^{b_i}$ be finitely generated $R$-modules where the sequences ${(a_i)}_{0 \leq i \leq n}$ and ${(b_i)}_{0 \leq i \leq n}$ are non-increasing. Then there is an injection $A \hookrightarrow B$ if and only if $a_i \leq b_i$ for all $i\in\{0,\dots,n\}$.
\end{lemma}

\begin{proof} This is a consequence of the structure theorem for finitely generated modules over a PID (see, \emph{e.g.}, \cite[Ch. III, Theorem 7.5]{lang}). \end{proof}

For a finitely generated $R$-module $M$, let $\mu(M)$ be the minimal number of generators of $M$.

\begin{lemma} \label{injKer}
Let $A$, $B$ be finitely generated $R$-modules. Suppose that all the cyclic direct summands of $A$ have length that is greater than or equal to the length of any cyclic direct summand of $B$. If $f:A \oplus B \rightarrow A$ is an $R$-module homomorphism, then there is an injection $B\hookrightarrow\ker(f)$. Furthermore, if $f$ is surjective,  then there is an isomorphism $B\simeq\ker(f)$.
\end{lemma}

\begin{proof} We divide our strategy into three steps. 

\texttt{Step 1.} Let $M$ and $N$ be finitely generated $R$-modules such that $N \subset M$ and
\begin{equation} \label{eq:decomposition-M-N}
M \simeq \bigoplus_{i=1}^k R/(\pi^{m_i}),\quad N\simeq\bigoplus_{i=1}^t R/(\pi^{n_i}).
\end{equation}
Set
\[
\mathcal{A}(M,N)\defeq\bigl\{S\subset\{1,\dots,t\}\mid\text{$\exists\;\sigma:S\hookrightarrow\{1,\dots,k\}$ such that $n_i\geq m_{\sigma(i)}$ for all $i\in S$}\bigr\}
\]
and $i(M,N)\defeq\max\bigl\{\#S\mid S\in\mathcal{A}(M,N)\bigr\}$. We claim that
\[
\mu(M/N)\geq\mu(M)-i(M,N).
\]
The projection $\mathrm{pr}:M\twoheadrightarrow M/\pi M$ induces a surjection $M/N\twoheadrightarrow(M/\pi M)\big/\mathrm{pr}(N)$, so there is an inequality
\[
\mu(M/N) \geq \mu(M) - \dim_{R/\pi R} \mathrm{pr}(N).
\]
On the other hand, the inequality
\[
\dim_{R/\pi R}\mathrm{pr}(N)\leq i(M,N)
\]
follows by observing that the image under $\mathrm{pr}$ of a generator of a cyclic component of $N$ of length $\ell$ is $0$ in all the components that are associated with a cyclic component of $M$ of length greater than $\ell$.
 More precisely, let $C$ be a cyclic component of $N$ of length $\alpha$, let $D$ be the maximal cyclic component of $M$ containing $C$ and let $\beta=\length_R(D)$. If $\beta>\alpha$, then $C=\pi^{\beta-\alpha}D$, hence $\mathrm{pr}(C)=0$. Thus, if $\mathrm{pr}(C)\ne 0$, then $\alpha=\beta$, \emph{i.e.}, $C=D$. If $C=R/(\pi^{n_i})$ is one of the summands of $N$ in \eqref{eq:decomposition-M-N} with $\mathrm{pr}(C)\ne 0$, then the argument above implies that $C$ is also a direct summand of $M$, so there is an $m_j$ such that $n_i=m_j$. The set $S$ of all indices $i\in\{1,\dots,t\}$ such that $\mathrm{pr}(R/\pi^{n_i})\ne 0$ has cardinality $\dim_{R/\pi R} \mathrm{pr}(N)$ and, thanks to the argument above, $S\subset \mathcal{A}(M,N)$. It follows that $\dim_{R/\pi R} \mathrm{pr}(N) \leq i(M,N)$.

\texttt{Step 2.} Let $A$ and $B$ be as in the statement and let $f:A\oplus B\rightarrow A$ be an $R$-module homomorphism. From \texttt{Step 1}, there is an inequality
\[
\mu\bigl((A\oplus B)/\ker(f)\bigr)\geq\mu(A)+\mu(B)-i\bigl(A\oplus B,\ker(f)\bigr).
\]
Since $f$ induces an injection $(A\oplus B)/\ker(f)\hookrightarrow A$, there is an inequality
\[
\mu(A)\geq\mu\bigl((A\oplus B)/\ker(f)\bigr),
\]
whence
\[
\mu(B)\leq i\bigl(A\oplus B,\ker(f)\bigr).
\]
This means that there are $\mu(B)$ cyclic components of $\ker(f)$ of length greater than or equal to the length of $\mu(B)$ fixed cyclic components of $A \oplus B$. The first claim of the lemma follows from Lemma \ref{expSub}, as all the cyclic components of $A$ have length greater than or equal to the length of each of the cyclic components of $B$.

\texttt{Step 3.} If $f$ is surjective, then $\length(B)=\length\bigl(\ker(f)\bigr)$ and the second statement follows from the first one. \end{proof}

By Proposition \ref{prop-structureart}, the existence of a non-trivial free $\kappa\in\ks\bigl(T,\mathcal{F},\mathcal{P},\{\chi_{n,\ell}\}\bigr)$ of sign $\epsilon$ yields a $\tau_c$-decomposition
\[
\Sel_{\mathcal{F}(n)}(K,T)=R\oplus M_n\oplus M_n=R\oplus (M_n^{\epsilon_n})^2\oplus (M_n^{-\epsilon_n})^2
\]
for every $n\in\mathcal{N}$. From here on, write $\length(M)$ for the length of an $R$-module $M$.

\begin{proposition} \label{SelComp}
Suppose that there exists a non-trivial, free $\kappa\in\ks\bigl(T,\mathcal{F},\mathcal{P},\{\chi_{n,\ell}\}\bigr)$ of sign $\epsilon$, let $n\in\mathcal{N}$ and write isomorphisms
\[
M_n^{\epsilon_n}\simeq\bigoplus_{i\geq0}R/\mathfrak{m}^{d_i^{\epsilon_n}}, \quad M_n^{-\epsilon_n}\simeq\bigoplus_{i\geq0} R/\mathfrak{m}^{d_i^{-\epsilon_n}},
\]
where ${(d_i^{\pm \epsilon_n})}_{i\geq0}$ are non-increasing sequences of non-negative integers with $d_i^{\pm \epsilon_n}=0$ for $i\gg 0$. Furthermore, let $\ell\in \mathcal{P}$ be a prime number not dividing $n$ such that the restriction of the localization $\loc_\ell:\Sel_{\mathcal{F}(n)}(K,T) \rightarrow H^1_{\f}(K_\lambda,T)$ to the direct sum of an $\epsilon_n$-component isomorphic to $R$ and a $-\epsilon_n$-component isomorphic to $R/\mathfrak{m}^{d_0^{-\epsilon_n}}$ is injective. Then there are isomorphisms of $R$-modules
\[
M_{n \ell}^{\epsilon_n}\simeq M_n^{\epsilon_n} \simeq \bigoplus_{i \geq 0}R/\mathfrak{m}^{d_i^{\epsilon_n}}, \quad M_{n \ell}^{-\epsilon_n} \simeq \bigoplus_{i \geq 1}R/\mathfrak{m}^{d_i^{-\epsilon_n}}.
\]
\end{proposition}

\begin{proof} In this proof, set $\epsilon\defeq\epsilon_n$ to simplify the notation. The exact sequences
\[
0 \longrightarrow \Sel^{\pm \epsilon}_{\mathcal{F}_\ell(n)}(K,T) \longrightarrow \Sel^{\pm \epsilon}_{\mathcal{F}(n)}(K,T) \longrightarrow H^1_{\f}(K_\lambda,T)^{\pm \epsilon}
\]
combined with Lemma \ref{lemmaH1f} and the injectivity of $\loc_\ell$ imply that $a=k$, $b=d_0^{-\epsilon}$, $\delta=0$ in diagram \eqref{diagram-inclusions-eq}.
Indeed, one has
\[
\length\Biggl(\frac{\Sel^\epsilon_{\mathcal{F}(n)}(K,T)}{\Sel^\epsilon_{\mathcal{F}_\ell(n)}(K,T)} \Biggr)=k, \quad \length \Biggl(\frac{\Sel^{-\epsilon}_{\mathcal{F}(n)}(K,T)}{\Sel^{-\epsilon}_{\mathcal{F}_\ell(n)}(K,T)} \Biggr)=d_0^{-\epsilon}.
\]
In particular, by Lemma \ref{injKer} there are isomorphisms
\begin{equation}\label{eq-sel1}
\Sel^\epsilon_{\mathcal{F}_\ell(n)}(K,T)\simeq\bigoplus_{i\geq0}\bigl(R/\mathfrak{m}^{d_i^\epsilon}\bigr)^2,\quad\Sel^{-\epsilon}_{\mathcal{F}_\ell(n)}(K,T)\simeq R/\mathfrak{m}^{d_0^{-\epsilon}} \oplus \bigoplus_{i \geq 1}\bigl(R/\mathfrak{m}^{d_i^{-\epsilon}}\bigr)^2.
\end{equation}
By Proposition \ref{prop-structureart}, $\rho(n\ell)^{-\epsilon}$ is odd, so by Proposition \ref{SelIncl} there are also equalities 
\[
\length\Biggl(\frac{\Sel^\epsilon_{\mathcal{F}(n \ell)}(K,T)}{\Sel^\epsilon_{\mathcal{F}_\ell(n)}(K,T)} \Biggr)=0, \quad \length \Biggl( \frac{\Sel^{-\epsilon}_{\mathcal{F}(n \ell)}(K,T)}{\Sel^{-\epsilon}_{\mathcal{F}_\ell(n)}(K,T)} \Biggr)=k-d_0^{-\epsilon}, 
\]
which imply the equalities
\[
\length(M^\epsilon_{n \ell})=\frac{\length\bigl(\Sel^\epsilon_{\mathcal{F}(n\ell)}(K,T)\bigr)}{2}=\frac{\length\bigl(\Sel^\epsilon_{\mathcal{F}_\ell(n)}(K,T)\bigr)}{2}=\sum_{i \geq 0} d_i^\epsilon,
\]
and
\[
\begin{split}
\length(M^{-\epsilon}_{n \ell})&=\frac{\length\bigl(\Sel^{-\epsilon}_{\mathcal{F}(n\ell)}(K,T)\bigr)-k}{2}\\&=\frac{\length\bigl(\Sel^{-\epsilon}_{\mathcal{F}_\ell(n)}(K,T)\bigr)+k-d_0^{-\epsilon}-k}{2}=\sum_{i \geq 1}d_i^{-\epsilon}.
\end{split}
\]
We assume isomorphisms $M^\epsilon_{n\ell}\simeq\bigoplus_{i\geq0}R/\mathfrak{m}^{c_i^\epsilon}$ and $M^{-\epsilon}_{n\ell}\simeq\bigoplus_{i\geq1}R/\mathfrak{m}^{c_i^{-\epsilon}}$, where ${(c_i^\epsilon)}_{i\geq0}$ and ${(c_i^{-\epsilon})}_{i\geq1}$ are non-increasing sequences with $c_i^{\pm\epsilon}=0$ for $i\gg0$. Since $\bigoplus_{i\geq0}\bigl(R/\mathfrak{m}^{d_i^\epsilon}\bigr)^2$ is a submodule of $\Sel^\epsilon_{\mathcal{F}(n\ell)}(K,T)$ and $\bigoplus_{i\geq1}\bigl(R/\mathfrak{m}^{d_i^{-\epsilon}}\bigr)^2$ is a submodule of $\Sel^{-\epsilon}_{\mathcal{F}(n \ell)}(K,T)$ by \eqref{eq-sel1}, $c_i^{\pm\epsilon} \geq d_i^{\pm\epsilon}$ for all $i$, by Lemma \ref{expSub}. Finally, there are equalities
\[
\sum_{i\geq0}c_i^\epsilon=\length(M^\epsilon_{n\ell})=\sum_{i\geq0}d_i^{\epsilon}, \quad \sum_{i\geq1}c_i^{-\epsilon}=\length(M^{-\epsilon}_{n\ell})=\sum_{i\geq1}d_i^{-\epsilon},
\]
whence $c_i^{\pm\epsilon}=d_i^{\pm \epsilon}$ for all $i$. The proof of the proposition is complete.\end{proof}

Given an $R$-module $M$ and an element $m\in M$, set 
\[
\ind(m,M)\defeq\max\bigl\{a\in\Z, \ a\le k\mid m\in \mathfrak{m}^aM\bigr\}. 
\]
When the module $M$ is clear from the context, we will often write $\ind(m)$ for $\ind(m,M)$. Moreover, for every $n\in\mathcal{N}$, define the $R$-module
\[
\Stub_n\defeq\mathfrak{m}^{\lambda_n}\Sel_{\mathcal{F}(n)}(K,T)\simeq\mathfrak{m}^{\lambda_n} R^{e(n)}.
\]

\begin{proposition} \label{classInStub}
Let $\kappa\in\ks\bigl(T,\mathcal{F},\mathcal{P},\{\chi_{n,\ell}\}\bigr)$ be a non-trivial free Kolyvagin system. Then
\[
\lambda_n\leq\ind\bigl(\kappa_n,\Sel_{\mathcal{F}(n)}(K,T)\bigr)
\]
for all $n\in\mathcal{N}$. In particular, $\kappa_n\in\Stub_n$ for all $n\in\mathcal{N}$.
\end{proposition}

\begin{proof} We proceed by induction on $\rho(n)$. If $\rho(n)=1$, then $M_n=0$ and the claim is trivial. Let $k\geq3$ be an odd integer; we assume the statement for all $n\in\mathcal{N}$ with $\rho(n)=k-2$. Let $n\in\mathcal{N}$ with $\rho(n)=k$. Since $\kappa$ is free, we can fix a free $R$-module $C_n \subset \Sel^{\epsilon_n}_{\mathcal{F}(n)}(K,T)$ of rank $1$ such that $\kappa_n \in C_n$. Let $D_n$ be a cyclic direct summand of $\Sel^{-\epsilon_n}_{\mathcal{F}(n)}(K,T)$ of maximal length $d$ and set $E_n \defeq C_n \oplus D_n$. By Lemma \ref{InfPrimes}, there exists $\ell \in \mathcal{P}$ such that the restriction of $\loc_\ell$ to $E_n$ is injective. We distinguish two cases.
    
\texttt{Case 1: $d>0$.} By Proposition \ref{SelComp}, we have $\lambda_{n\ell}=\lambda_{n}-d$ and $\rho(n \ell)=\rho(n)-2$; by induction, we get
\[
\lambda_{n \ell}\leq\ind\bigl(\kappa_{n\ell},\Sel_{\mathcal{F}(n\ell)}(K,T)\bigr).
\]
It follows that
\[
\begin{aligned}
\lambda_n&=\lambda_{n \ell}+d\leq\ind\bigl(\kappa_{n\ell},\Sel_{\mathcal{F}(n\ell)}(K,T)\bigr)+d\\&\leq\ind\Bigl(\loc_\ell(\kappa_{n\ell}),\loc_\ell\bigl(\Sel_{\mathcal{F}(n\ell)}(K,T)\bigr)\!\Bigr)+d\\
&=\ind\bigl(\loc_\ell(\kappa_{n\ell}),H^1_{\s}(K_\lambda,T) \otimes G_{n\ell}\bigr)\\
&=\ind\bigl(\loc_\ell(\kappa_n),H^1_{\f}(K_\lambda,T)\otimes G_n\bigr)=\ind\bigl(\kappa_n,\Sel_{\mathcal{F}(n)}(K,T)\bigr);
\end{aligned}
\]
here we use the fact that $\loc_\ell\bigl(\Sel_{\mathcal{F}(n\ell)}^{-\epsilon_n}(K,T)\bigr)$ is cyclic of colength $k-d$ in $H^1_{\tr}(K_\lambda,T)^{-\epsilon_n}\simeq R$ (see the proof of Proposition \ref{SelComp}) for the equality in the third line and the fact that $\loc_\ell$ is injective on $C_n$ for the last equality.
    
\texttt{Case 2: $d=0$.} By Proposition \ref{SelComp}, we have $\lambda_n=\lambda_{n\ell}$, $\rho(n\ell)^{\epsilon_n}=\rho(n)^{\epsilon_n}-1$ and $\rho(n\ell)^{-\epsilon_n}=1$. We can fix a free $R$-module $C_{n\ell}\subset \Sel_{\mathcal{F}(n\ell)}^{-\epsilon_n}(K,T)$ of rank $1$ such that $\kappa_{n\ell} \in C_{n\ell}$ and a direct summand $D_{n\ell}$ of $\Sel_{\mathcal{F}(n\ell)}^{\epsilon_n}(K,T)$ of maximal length $d'$. Set $E_{n\ell}\defeq C_{n\ell} \oplus D_{n\ell}$ and choose $\ell'\in \mathcal{P}$ such that the restriction of $\loc_{\ell'}$ to $E_{n \ell}$ is injective. In this case, $d'>0$, so we can argue as in \texttt{case 1} and get the inequality
\[
\lambda_{n\ell}\leq\ind\bigl(\kappa_{n\ell},\Sel_{\mathcal{F}(n\ell)}(K,T)\bigr).
\]
Finally, we apply the arguments in \texttt{Case 1} again and get the inequality
\[
\lambda_{n}\leq\ind\bigl(\kappa_{n},\Sel_{\mathcal{F}(n)}(K,T)\bigr).
\]
The proof is complete. \end{proof}

\subsubsection{Sheaves on graphs}\label{sec:sheaves-graphs}

Let us assume until the end of \S \ref{sec:sheaves-graphs} that there is a non-trivial free $\kappa\in\ks\bigl(T,\mathcal{F},\mathcal{P},\{\chi_{n,\ell}\}\bigr)$ of sign $\epsilon$.

\begin{definition}
    A set $\{\chi_{n,\ell}\}$ of $R$-automorphisms of $T$ is called an \emph{admissible modification} for $(T,\mathcal{F},\mathcal{P})$ if
    $$\psi_n^\ell(\chi_{n,\ell})(\Stub_n)=\psi_{n\ell}^\ell(\Stub_{n\ell})$$
    for every $n \in \mathcal{N}$ and $\ell \in \mathcal{P}$ with $\ell \nmid n$.
\end{definition}

\begin{proposition} \label{modification-prop}
There exists an admissible modification for $(T,\mathcal{F},\mathcal{P})$.
\end{proposition}

\begin{proof} We need to prove that the equality
\[
\length\bigl(\loc_\ell(\Stub_n)\bigr)=\length\bigl(\loc_\ell(\Stub_{n\ell})\bigr)
\]
holds for all $n\in\mathcal{N}$ and all $\ell\in\mathcal{P}$ with $\ell\nmid n$. Diagram \eqref{diagram-inclusions-eq} shows that 
\[
\length\bigl(\loc_\ell(\Stub_n)\bigr)=\max\bigl\{0,a-\lambda_n,b-\lambda_n\bigr\}
\]
and
\[
\length\bigl(\loc_\ell(\Stub_{n\ell})\bigr)=\max\bigl\{0,k-a-\delta-\lambda_{n\ell},k-b-\delta-\lambda_{n\ell}\bigr\}
\]
for some $a,b,\delta\in\N$. Again by Proposition \ref{SelIncl}, we have $\lambda_{n\ell}=\lambda_n+k-a-b-\delta$, whence
\[
\length\bigl(\loc_\ell(\Stub_{n\ell})\bigr)=\max\bigl\{0,a-\lambda_n,b-\lambda_n\bigl\}=\length\bigl(\loc_\ell(\Stub_n)\bigr).
\]
The proof is complete. \end{proof}

\begin{proposition} \label{modification-prop2}
There exists an admissible modification $\{ \chi'_{n,\ell}\}$ for $(T,\mathcal{F},\mathcal{P})$ such that $\kappa\in\ks\bigl(T,\mathcal{F},\mathcal{P},\{\chi'_{n,\ell}\}\bigr)$.
\end{proposition}

\begin{proof}Set
\[
c_{n,\ell} \defeq \psi_n^\ell(\chi_{n,\ell})(\kappa_n)=\psi_{n \ell}^\ell(\kappa_{n \ell}).
\]
By the arguments in the proof of Proposition \ref{modification-prop}, $\psi_n^\ell(\chi_{n,\ell})(\Stub_n)$ and $\psi_{n\ell}^\ell(\Stub_{n\ell})$ are $R$-cyclic submodules of $H^1_{\s}(K_\lambda,T) \otimes G_{n\ell}$ of the same length. If $c_{n,\ell}\neq0$, then $\psi_n^\ell(\chi_{n,\ell})(\Stub_n)$ and $\psi_{n\ell}^\ell(\Stub_{n\ell})$ have non-trivial intersection by Proposition \ref{classInStub}, so they are equal and we can take $\chi'_{n,\ell} \defeq \chi_{n,\ell}$. If $c_{n,\ell} = 0$, then there exists an $R$-automorphism $\chi'_{n,\ell}$ such that $\psi_n^\ell(\chi'_{n,\ell})(\Stub_n)=\psi_{n\ell}^\ell(\Stub_{n\ell})$. It follows that $\{\chi_{n,\ell}'\}$ is an admissible modification and $\kappa\in\ks\bigl(T,\mathcal{F},\mathcal{P},\{\chi_{n,\ell}'\}\bigr)$.
\end{proof}

In light of Proposition \ref{modification-prop2}, from now on we always assume that $\{\chi_{n,\ell}\} $ is an admissible modification for $(T,\mathcal{F},\mathcal{P})$. We denote by $\mathcal{X}$ the graph whose vertices $v(n)$ are indexed by $n \in \mathcal{N}$ and an edge $e(n,n \ell)$ connects $v(n)$ to $v(n \ell)$ whenever $n \in \mathcal{N}$ and $\ell \in \mathcal{P}$ with $\ell \nmid n$. We define the \emph{Kolyvagin system sheaf} $\ks (\mathcal{X})$ on $\mathcal{X}$ in the following way. For all vertices $v=v(n)$ of $\mathcal{X}$ we set
\[
\ks(v) \defeq \Sel_{\mathcal{F}(n)}(K,T) \otimes G_n
\]
and for all edges $e=e(n,n \ell)$ we set
\[
\ks(e) \defeq H^1_{\s}(K_\lambda,T) \otimes G_{n \ell}.
\]
For every edge $e(n,n\ell)$ with endpoint $v$ we define the vertex-to-edge map
\[
\psi_v^e:\ks(v)\longrightarrow\ks(e),\quad\psi_v^e\defeq\begin{cases}\psi_n^\ell(\chi_{n,\ell}) & \text{if $v=v(n)$},\\[2mm]\psi_{n\ell}^\ell & \text{if $v=v(n\ell)$}. \end{cases}
\]
We define a locally cyclic (see, \emph{e.g.}, \cite[Definition 3.4.2]{MR}) subsheaf $\Stub(\mathcal X)$ of $\ks(\mathcal{X})$, called the \emph{stub sheaf} of $\mathcal X$, as follows:
\begin{itemize}
\item $\Stub\bigl(v(n)\bigr)\defeq\Stub_n$;
\item $\Stub\bigl(e(n,n\ell)\bigr)\defeq\psi_n^\ell(\chi_{n,\ell})(\Stub_n)$;
\item the vertex-to-edge map $\psi_v^e$ is induced by the one of $\ks(\mathcal{X}).$
\end{itemize}

\begin{definition}
A \emph{core vertex} of $\mathcal{X}$ is a vertex $v(n)$ of $\mathcal X$ such that $\Stub_n \simeq R$.
\end{definition}

\begin{remark} \label{corerem} 
A vertex $v(n)$ of $\mathcal X$ is a core vertex if and only if $\rho(n)=1$. Therefore, there is a sign $\epsilon\in\{\pm\}$ such that $\rho(n)^\epsilon=1$ and $\rho(n)^{-\epsilon}=0$. Thus, by Lemma \ref{Ho1.5.3}, for all $\ell\nmid n$ and $\ell\in\mathcal{P}$ there are equalities $\rho(n\ell)^{-\epsilon}=\rho(n)^{-\epsilon}+1=1$; moreover, again by Lemma \ref{Ho1.5.3}, if $\loc_\ell\bigl(\Sel_{\mathcal{F}(n)}^\epsilon(K,T/\mathfrak{m}T)\bigr)=0$, then $\rho(n\ell)^\epsilon=\rho(n)^\epsilon+1=2$ and $\rho(n\ell)=\rho(n)+2=3$, whereas if $\loc_\ell\bigl(\Sel_{\mathcal{F}(n)}^\epsilon(K,T/\mathfrak{m}T)\bigr)\neq0$, then
$\rho(n\ell)^\epsilon=\rho(n)^\epsilon-1=0$ and $\rho(n\ell)=\rho(n)=1$.\end{remark}

From here on, we denote by $\mathcal{X}_0$ the \emph{core subgraph} of $\mathcal X$, \emph{i.e.}, the subgraph of $\mathcal{X}$ whose vertices are the core vertices. Moreover, for every $a\in\mathcal{N}$, we write $\mathcal{X}_{0,a}$ for the subgraph of $\mathcal{X}_0$ consisting of the core vertices $v(n)$ with $a \,|\,n$. Finally, $\Stub(\mathcal{X}_0)$ and $\Stub(\mathcal{X}_{0,a})$ are the sheaves induced by $\Stub(\mathcal{X})$.

\begin{definition}
Let $G \in \{\mathcal{X},\mathcal{X}_0,\mathcal{X}_{0,a}\}$. Given $v$ and $v'$ vertices of $G$, a \emph{path} from $v$ to $v'$ is a finite sequence of vertices $v=v_0,v_1 ,\dots v_t=v'$ such that $v_i$ is connected to $v_{i+1}$ by an edge $e_i$. A path is \emph{surjective} for $\Stub(G)$ if the vertex-to-edge map
$$\psi_{v_{i+1}}^{e_i}: \Stub(v_{i+1}) \longrightarrow \Stub(e_i)$$
is an isomorphism for every $i$.
\end{definition}

\begin{remark} \label{remsurjpath}
A surjective path in $\Stub(G)$ for $G\in\{\mathcal{X},\mathcal{X}_0,\mathcal{X}_{0,a}\}$ from a vertex $v$ to a vertex $v'$ defines, as explained in \cite[\S3.4]{MR}, a surjection $\Stub(v)\twoheadrightarrow\Stub(v')$.
\end{remark} 

We need three lemmas on (core) vertices of $\mathcal X$.

\begin{lemma} \label{existSurjPath}
For any vertex $v$ of $\mathcal{X}$, there are a core vertex $v_0$ and a surjective path from $v_0$ to $v$ in $\mathcal{X}$. For any $a\in\mathcal{N}$, there exists $n\in\mathcal{N}$ with $a\,|\,n$ such that $v(n)$ is a core vertex of $\mathcal X$.
\end{lemma}

\begin{proof} Let $v$ be a vertex of $\mathcal X$. If $v$ is a core vertex, then there is nothing to prove. Otherwise, let $n\in\mathcal{N}$ be such that $v=v(n)$ and set $w_0\defeq v(n)$. Thanks to Proposition \ref{prop-structureart}, we can fix a rank-$1$ free $R$-submodule $C_n$ of $\Sel^{\epsilon_n}_{\mathcal{F}(n)}(K,T)$ and a direct summand $D_n$ of $\Sel^{-\epsilon_n}_{\mathcal{F}(n)}(K,T)$ of maximal length $d$. Set $E_n \defeq C_n \oplus D_n$ and choose $\ell\in\mathcal{P}$ such that the restriction of $\loc_\ell$ to $E_n$ is injective (use Lemma \ref{InfPrimes}). Then, by Proposition \ref{SelComp}, $\lambda_{n\ell}=\lambda_n-d$.

\texttt{Case 1: $d>0$.} Set $w_1\defeq v(n \ell)$ and notice that $\length\bigl(\Stub(w_0)\bigr)<\length\bigl(\Stub(w_1)\bigr)$. Furthermore, the map
\[
\psi_{w_0}^{e(n,n\ell)}:\Stub(w_0)\longrightarrow\Stub\bigl(e(n,n\ell)\bigr)
\]
is an isomorphism, as $\psi_n^\ell(\chi_{n,\ell})$ is injective on $E_n$ by the choice of $\ell$.

\texttt{Case 2: $d=0$.} Fix a rank-$1$ free $R$-module $C_{n\ell}\subset\Sel^{-\epsilon_n}_{\mathcal{F}(n\ell)}(K,T)$ and a direct summand $D_{n\ell}$ of $\Sel^{\epsilon_n}_{\mathcal{F}(n\ell)}(K,T)$ of maximal length $d'>0$. Set $E_{n\ell} \defeq C_{n\ell} \oplus D_{n \ell}$ and choose $\ell'\in\mathcal{P}$ such that the restriction of $\loc_{\ell'}$ to $C_{n \ell}$ is injective (use again Lemma \ref{InfPrimes}). Now set $w_1 \defeq v(n\ell\ell')$ and observe that $\length\bigl(\Stub(w_0)\bigr)<\length\bigl(\Stub(w_1)\bigr)$. Moreover, $\psi_{v(n)}^{e(n,n\ell)}$ and $\psi_{v(n\ell)}^{e(n\ell,n\ell\ell')}$ are isomorphisms by the choices of $\ell$ and $\ell'$.

Iterating this argument, we find a vertex $w_t$ such that $\length(\Stub(w_t))=k$ (it is a core vertex) and a surjective path $w_t, \dots, w_1,w_0=v$. The second part of the statement follows directly from the construction above.
\end{proof}

\begin{lemma} \label{lemma-connected}
If $v(a)$ is a core vertex, then $\mathcal{X}_{0,a}$ is path connected.
\end{lemma}

\begin{proof} Let $n=ab \in \mathcal{N}$ be such that $v(n)$ is a core vertex. We prove by induction on the number of prime divisors of $b$ that there exists a path from $v(n)$ to $v(a)$ in $\mathcal{X}_{0,a}$. Combining the fact that $v(n)$ is a core vertex with Proposition \ref{prop-structureart}, we have $\rho^{\epsilon_n}(n)=1$ and $\rho^{-\epsilon_n}(n)=0$.

\texttt{Case 1.} Assume that there is a prime number $\ell\,|\,b$ such that $\loc_\ell\bigl(\Sel_{\mathcal{F}(n)}(K,T/\mathfrak{m}T)\bigr)\neq0$. By part (1) of Lemma \ref{Ho1.5.3}, we must have $\rho^{\epsilon_n}(n/\ell)=0$ and $\rho^{-\epsilon_n}(n/\ell)=1$. In particular, $v(n/\ell)$ is a vertex in $\mathcal{X}_{0,a}$ and we can conclude by induction hypothesis.

\texttt{Case 2}. Assume $\loc_\ell\bigl(\Sel_{\mathcal{F}(n)}(K,T/\mathfrak{m}T)\bigr)=0$ for all $\ell\,|\,b$ and choose a prime number $q\,|\,b$. Since $\loc_q\bigl(\Sel_{\mathcal{F}(n)}(K,T/\mathfrak{m}T)\bigr)=0$, there is an inclusion
\[
\Sel_{\mathcal{F}(n)}(K,T/\mathfrak{m}T) \subset \Sel_{\mathcal{F}(n/q)}(K,T/\mathfrak{m}T).
\]
Then $\rho^{\epsilon_n}(n/q) \geq 1$ and, by Lemma \ref{Ho1.5.3}, $\rho^{\epsilon_n}(n/q)=2$ and $\rho^{-\epsilon_n}(n/q)=1$. Thanks to Lemma \ref{InfPrimes}, there is a prime $\ell\in \mathcal{P}$ such that $\loc_\ell$ is injective when restricted to a $2$-dimensional $R/\mathfrak{m}$-vector subspace of $\Sel_{\mathcal{F}(n/q)}(K,T/\mathfrak{m}T)$ that contains $\Sel_{\mathcal{F}(n)}(K,T/\mathfrak{m}T)$ and $\Sel^{-\epsilon_n}_{\mathcal{F}(n/q)}(K,T/\mathfrak{m}T)$. Then, by Lemma \ref{Ho1.5.3} we have $\rho(n\ell/q)=1$ and $\rho(n \ell)=1$. 
Now assume by contradiction that 
\begin{equation} \label{loct}
\loc_t\bigl(\Sel_{\mathcal{F}(n\ell/q)}(K,T/\mathfrak{m}T)\bigr)=0
\end{equation} for every prime number $t\,|\,(b\ell/q)$. Then all the inclusions
\[
\begin{split}
\Sel_{\mathcal{F}(n\ell/q)}(K,T/\mathfrak{m}T)&\subset\Sel_{\mathcal{F}_{b\ell/q}(a)}(K,T/\mathfrak{m}T)\\&\subset\Sel_{\mathcal{F}_{b/q}(a)}(K,T/\mathfrak{m}T)\subset\Sel_{\mathcal{F}(n)}(K,T/\mathfrak{m}T)
\end{split}
\]
are equalities, as $\rho(n)=\rho(n\ell/q)=1$. Combining this fact with the vanishing assumption in \eqref{loct} contradicts the condition $\loc_\ell\bigl(\Sel_{\mathcal{F}(n)}(K,T/\mathfrak{m}T)\bigr)\neq0$; it follows that there exists a prime $t\,|\,(b\ell/q)$ such that $\loc_t\bigl(\Sel_{\mathcal{F}(n\ell/q)}(K,T/\mathfrak{m}T)\bigr)\neq0$. Applying the argument in \texttt{Case 1}, we get $\rho(n\ell/qt)=1$. Finally, $v(n)$, $v(n\ell)$, $v(n \ell/q)$, $v(n\ell/qt)$ is a path in $\mathcal{X}_{0,a}$, which concludes the proof by the inductive hypothesis. \end{proof}

\begin{lemma} \label{surjPathX0}
For any vertex $v$ of $\mathcal{X}$ and any core vertex $v_0$ of $\mathcal{X}$, there is a surjective path from $v_0$ to $v$.
\end{lemma}

\begin{proof} Thanks to Lemma \ref{existSurjPath}, we just need to prove the statement when $v$ is a core vertex too. Let $v_0=v(a)$ and $v=v(b)$ with $a, b \in \mathcal{N}$. By Lemma \ref{existSurjPath} there exists $n \in \mathcal{N}$ with $ab |n$ such that $v(n)$ is a core vertex. By Lemma \ref{lemma-connected}, there are paths from $v(a)$ to $v(n)$ in $\mathcal{X}_{0,a}$ and from $v(b)$ to $v(n)$ in $\mathcal{X}_{0,b}$. In particular, there is a path from $v(a)$ to $v(b)$ in $\mathcal{X}_0$. We claim that every path in $\mathcal{X}_0$ is surjective, \emph{i.e.}, we claim that $\loc_\ell$ is injective on $\Sel_{\mathcal{F}(n)}(K,T)$ and $\Sel_{\mathcal{F}(n\ell)}(K,T)$ for every $n \in\mathcal{N}$ and $\ell \in \mathcal{P}$ with $\ell \nmid n$ and $v(n)$ and $v(n \ell)$ core vertices; this claim clearly concludes the proof. To show the claim, since $\rho(n)=\rho(n \ell)=1$, by Lemma \ref{Ho1.5.3} there exists $\epsilon \in \{\pm\}$ such that $\rho^\epsilon(n\ell)=\rho^\epsilon(n)-1$, so we have
\begin{equation} \label{locnonzero}
\loc_\ell\bigl(\Sel_{\mathcal{F}(n)}(K,T/\mathfrak{m}T)\bigr)\neq0,
\end{equation}
by Lemma \ref{Ho1.5.3}.  Since $v(n)$ is a core vertex, $\rho(n)=1$, so the inclusion $\Sel_{\mathcal{F}_\ell(n)}(K,T/\mathfrak{m}T)\subset\Sel_{\mathcal{F}(n)}(K,T/\mathfrak{m}T)$ combined with \eqref{locnonzero} implies $\Sel_{\mathcal{F}_\ell(n)}(K,T/\mathfrak{m}T)=0$. Then, by \cite[Lemma 1.3.3]{HoHeeg}, we also have $\Sel_{\mathcal{F}_\ell(n)}(K,T)=0$ and the claim follows since it is the kernel of both the localization maps we are considering.
    \end{proof}

Now we introduce the notion of global section of the sheaf $\Stub(\mathcal{X})$.

\begin{definition}
A \emph{global section} of $\Stub(\mathcal{X})$ is a function $s$ on the vertices and the edges of $\mathcal{X}$ such that
\[
\bigl(\psi_n^\ell(\chi_{n,\ell})\bigr)\Bigl(s\bigl(v(n)\bigr)\!\Bigr)=s(e)=\psi_{n \ell}^\ell \Bigl(s\bigl(v(n\ell)\bigr)\!\Bigr)
\]
for every edge $e=e(n,n\ell)$ of $\mathcal{X}$.
\end{definition}

From here on, set 
\begin{itemize}
\item $\partial^{(j)}(\kappa)\defeq\min\bigl\{\ind(\kappa_n,\Sel_{\mathcal{F}(n)}(K,T))\mid\nu(n)=j\bigr\}$ for $j\geq0$;
\item $\partial(\kappa)\defeq\min\bigl\{\partial^{(j)}(\kappa)\mid j\geq0\bigr\}$.
\end{itemize}
The following is a rigidity result for free Kolyvagin systems.

\begin{theorem}[Rigidity theorem] \label{Rigidity}
The class $\kappa_n$ generates $\mathfrak{m}^{\partial(\kappa)} \Stub_n$ for all $n \in \mathcal{N}$.
\end{theorem}

\begin{proof} For every $n\in\mathcal{N}$ and $\ell\in\mathcal{P}$ with $\ell\nmid n$, set
\[
s\bigl(v(n)\bigr)\defeq\kappa_n,\quad s\bigl(e(n,n\ell)\bigr)\defeq\psi_n^\ell(\chi_{n,\ell})(\kappa_n).
\]
By definition of a Kolyvagin system and Proposition \ref{classInStub}, $s$ is a global section of $\Stub(\mathcal{X})$. Let $v_0$ be a core vertex of $\mathcal{X}$ (which exists by Lemma \ref{existSurjPath}) and let $\delta \geq 0$ be an integer such that $s(v_0)$ generates $\mathfrak{m}^\delta \Stub(v_0)$. By Lemma \ref{surjPathX0}, there is a surjective path from $v_0$ to any vertex $v$ of $\mathcal{X}$. It follows (\emph{cf.} Remark \ref{remsurjpath}) that there is a surjection $\Stub(v_0)\twoheadrightarrow\Stub(v)$ mapping $s(v_0)$ to $s(v)$. In particular, $s(v)$ generates $\mathfrak{m}^\delta\Stub(v)$ for every vertex $v$ of $\mathcal{X}$. On the other hand, we have
\[
\delta=\ind\bigl(s(v),\Stub(v)\bigr)\leq\ind\bigl(s(v),\ks(v)\bigr) 
\]
for all vertices $v$ of $\mathcal{X}$ with $s(v) \neq 0$, with equality if $v$ is a core vertex. In particular, there are equalities
\[
\begin{split}
\delta&=\min\bigl\{\ind\bigl(s(v),\ks(v)\bigr)\mid\text{$v$ is a vertex of $\mathcal{X}$}\bigr\}\\
&=\min\bigl\{\ind\bigl(\kappa_n,\Sel_{\mathcal{F}(n)}(K,T)\bigr)\mid n\in\mathcal{N}\bigr\}=\partial(\kappa),
\end{split}
\]
and the theorem is proved. \end{proof}

\subsubsection{Structure theorem}

We are now ready to prove the main theorem of this section, that gives an explicit link between the structure of the Selmer group with the $\partial$-invariants attached to a free Kolyvagin system.

\begin{theorem}[Structure theorem] \label{artSel}
Let $\kappa$ be a free Kolyvagin system of sign $\epsilon \in \{\pm 1\}$ in $\ks\bigl(T,\mathcal{F},\mathcal{P},\{\chi_{n,\ell}\}\bigr)$ with $\kappa_1\neq0$ and assume that there are isomorphisms of $R$-modules
\[
\Sel^\epsilon_\mathcal{F}(K,T)\simeq R\oplus\bigoplus_{i\in\N}\bigl(R/\mathfrak{m}^{d_i^\epsilon}\bigr)^2,\quad \Sel^{-\epsilon}_\mathcal{F}(K,T) \simeq \bigoplus_{i\in\N}\bigl(R/\mathfrak m^{d_i^{-\epsilon}}\bigr)^2
\]
with ${(d_i^{\pm \epsilon})}_{i\in\N}$ non-increasing sequences in $\N$ such that $d_i^{\pm \epsilon}=0$ for $i \gg 0$. Then
\[
\partial^{(j)}(\kappa)=\begin{cases} \min\bigl\{k,\partial(\kappa)+\sum_{i\geq j/2}(d_i^\epsilon+d_i^{-\epsilon})\bigr\} & \text{if $j$ is even}\\[2mm]
\min\bigl\{k,\partial(\kappa)+\sum_{i\geq(j+1)/2}(d_i^\epsilon+d_i^{-\epsilon})+d_{(j-1)/2}^{\epsilon}\bigr\} & \text{if $j$ is odd}
\end{cases}
\]
for all $j\in\N$.
\end{theorem}
\begin{proof} 
For simplicity, we set $\ind(\kappa_n)\defeq\ind(\kappa_n,\Sel_{\mathcal{F}(n)}(K,T))$. By Theorem \ref{Rigidity}, for all $n\in\mathcal{N}$ there is the equality
\[
\ind(\kappa_n)=\min\bigl\{k,\partial(\kappa)+\lambda_n\bigr\},
\]
which yields the equality
\[
\partial^{(j)}(\kappa)=\min\bigl\{k,\partial(\kappa)+\lambda_n\mid n\in\mathcal{N},\nu(n)=j \bigr\}.
\]
We prove the statement by induction on $j$. As
\[
\partial^{(0)}(\kappa)=\min\bigl\{k,\partial(\kappa)+\length(M_0)\bigr\}=\min\biggl\{k,\partial (\kappa)+ \sum_{i\in\N} (d_i^\epsilon+d_i^{-\epsilon})\biggr\},
\]
the basis of the induction follows. We split the proof of the inductive step into two parts, depending on the parity of $j$.

\texttt{Step 1.} Assume the statement holds true for $j-1$ and for any free Kolyvagin system in $\ks\bigl(T,\mathcal{F},\mathcal{P},\{\chi_{n,\ell}\}\bigr)$; we prove it for $j$ when $j$ is odd. Let $n \in \mathcal{N}$ with $\nu(n)=j$ and consider the maps
\begin{equation}\label{eq:product-of-localizations}
\Sel^{\pm \epsilon}_{\mathcal{F}}(K,T) \longrightarrow\prod_{\ell|n} H^1_{\f}(K_\lambda,T)^{\pm \epsilon}
\end{equation}
induced by the product of localizations, where the right-hand sides are free of rank $j$.
The images of these maps inject into $R$-modules that are respectively isomorphic to 
\[
R \oplus \bigoplus_{i=0}^{\frac{j-1}{2}-1}\bigl(R/\mathfrak{m}^{d_i^\epsilon}\bigr)^2,\quad \bigoplus_{i=0}^{\frac{j-1}{2}-1}(R/\mathfrak{m}^{d_i^{-\epsilon}})^2 \oplus R\big/\mathfrak{m}^{d^{-\epsilon}_{(j-1)/2}}, 
\]
so, by Lemma \ref{injKer}, the kernels contain $R$-submodules that are isomorphic to 
\[
\bigoplus_{i\geq\frac{j-1}{2}}\bigl(R/\mathfrak{m}^{d_i^{\epsilon}}\bigr)^2,\quad R/\mathfrak{m}^{d_{(j-1)/2}^{-\epsilon}}\oplus\bigoplus_{i\geq\frac{j-1}{2}+1} \bigl(R/\mathfrak{m}^{d_i^{-\epsilon}}\bigr)^2,
\]
respectively. Since the kernels are contained in $\Sel^{\pm\epsilon}_{\mathcal{F}(n)}(K,T)$ 
and $\rho(n)^{-\epsilon}$ is odd, there are inequalities
\[
\length(M^\epsilon_n) \geq \sum_{i \geq \frac{j-1}{2}}d_i^\epsilon, \quad \length(M_n^{-\epsilon}) \geq \sum_{i \geq \frac{j+1}{2}} d_i^{-\epsilon}
\]
by Lemma \ref{expSub}
and it follows that
\[
\begin{aligned}
\partial^{(j)}(\kappa)&=\min\bigl\{k,\partial(\kappa)+\length(M_n)\mid\text{$n\in \mathcal{N}$, $\nu(n)=j$}\bigr\}  \\
&\geq \min\Biggl\{k, \partial(\kappa)+ \sum_{i \geq \frac{j+1}{2}}(d_i^\epsilon+d_i^{-\epsilon})+d_{(j-1)/2}^\epsilon\Biggr\}.
\end{aligned}
\]
Now we prove the opposite inequality. Let $C\subset\Sel^{\epsilon}_{\mathcal{F}}(K,T)$ be a rank-$1$ free $R$-module such that $\kappa_1 \in C$ and fix $D$ to be a cyclic direct summand of maximal length of $\Sel^{-\epsilon}_{\mathcal{F}}(K,T)$. Set $E \defeq C \oplus D$. Using Lemma \ref{InfPrimes}, choose $\ell \in \mathcal{P}$ such that the restriction of
\[
\loc_\ell:\Sel_\mathcal{F}(K,T) \rightarrow H^1_{\f}(K_\lambda,T)
\]
to $E$ is injective. 
By Lemma \ref{SelComp}, there are isomorphisms
\[
M_\ell^\epsilon \simeq \bigoplus_{i \geq 0}R/\mathfrak{m}^{d_i^\epsilon}, \quad M_\ell^{-\epsilon} \simeq \bigoplus_{i \geq 1}R/\mathfrak{m}^{d_i^{-\epsilon}}.
\]
Now set $\mathcal{P}'\defeq\mathcal{P}\smallsetminus\{\ell\}$ and consider the Kolyvagin system $\kappa'$ for $(T,\mathcal{F}(\ell),\mathcal{P}')$ obtained by setting $\kappa_n' \defeq \kappa_{n \ell}$ for all $n \in \mathcal{N}$ with $\ell \nmid n$, which is free of sign $-\epsilon$. Since $\kappa_\ell \neq 0$ is contained both in $\kappa$ and in $\kappa'$, it generates both $\mathfrak{m}^{\partial(\kappa)}\Stub_\ell$ and $\mathfrak{m}^{\partial(\kappa')}\Stub_\ell$, by Theorem \ref{Rigidity}. Thus, there is an equality $\partial(\kappa)=\partial(\kappa')$. Finally, we obtain
\[
\begin{split}
\partial^{(j)}(\kappa) \leq \partial^{(j-1)}(\kappa')&=\min\Biggl\{k,\partial(\kappa')+\sum_{i \geq\frac{j-1}{2}}(d_i^\epsilon+d_{i+1}^{-\epsilon})\Biggr\}\\&=\min\Biggl\{k, \partial(\kappa)+\sum_{i\geq\frac{j+1}{2}}(d_i^\epsilon+d_i^{-\epsilon})+d_{(j-1)/2}^\epsilon\Biggr\},
\end{split}
\]
where the first equality follows from our inductive hypothesis.

\texttt{Step 2.} The proof when $j$ is even follows the exact same lines of Step 1, the only difference being that the images of the two maps in \eqref{eq:product-of-localizations} inject into $R$-modules that are respectively isomorphic to 
\[
R \oplus \bigoplus_{i=0}^{\frac{j-2}{2}-1}\bigl(R/\mathfrak{m}^{d_i^\epsilon}\bigr)^2\oplus R/\mathfrak{m}^{d^\epsilon_{(j-2)/2}},\quad\bigoplus_{i=0}^{\frac{j-2}{2}}\bigl(R/\mathfrak{m}^{d_i^{-\epsilon}}\bigr)^2, 
\]
and so the kernels contain $R$-submodules that are respectively isomorphic to 
\[
R/\mathfrak{m}^{d_{(j-2)/2}^\epsilon}\oplus\bigoplus_{i\geq\frac{j}{2}} \bigl(R/\mathfrak{m}^{d_i^{\epsilon}}\bigr)^2,\quad\bigoplus_{i\geq\frac{j}{2}} \bigl(R/\mathfrak{m}^{d_i^{-\epsilon}}\bigr)^2
\]
by Lemma \ref{injKer}. Arguing as in \texttt{Step 1}, one gets the first inequality, while the opposite one is obtained, by the same argument, using the same auxiliary Kolyvagin system $\kappa'$. \end{proof}

\subsection{Universal Kolyvagin systems over a DVR} \label{sectionDVR}

Let $R$ be a complete discrete valuation ring with maximal ideal $\mathfrak{m}=(\pi)$ and quotient field $F$; we assume that $F$ is a finite extension of $\Q_p$ inside $\overline{\Q}_p$. Suppose throughout this section that Assumption \ref{ass:Howard-assumption} is satisfied. Set 
\[
V\defeq T\otimes_RF,\quad A\defeq V/T=T\otimes_R(F/R).
\]
Fix a Selmer structure $\mathcal{F}$ on $V$ over $K$, with the property that $H^1_\mathcal{F}(K_v,V)=H^1_{\f}(K_v,V)$ for any place $v\nmid Np$ of $K$, where $H^1_{\f}(K_v,V)=H^1_{\mathrm{unr}}(K_v,V)$ denotes unramified cohomology at unramified primes.
Then, $\mathcal{F}$ induces a Selmer structure on $T$ (to be denoted with the same symbol) by \emph{propagation}, \emph{i.e.} by defining $H^1_{\mathcal{F}}(K_v,T)$ to be the inverse image of $H^1_{\mathcal{F}}(K_v,V)$ under the map induced by the inclusion $T \hookrightarrow V$ for all primes $v$ of $K$. Thanks to \cite[Lemma 1.1.9]{MR}, we still have that $H^1_{\mathcal{F}}(K_v,T)=H^1_{\unr}(K_v,T)$ for all $v \nmid Np$.
Analogously, let $H^1_\mathcal{F}(K_v,A)$ be the Selmer structure on $A$ obtained by propagation from $H^1_\mathcal{F}(K_v,V)$ via the projection map $V\twoheadrightarrow A$ and define
\[
\Sel_\mathcal{F}(K,A)\defeq\ker\Biggl(H^1(K,A)\longrightarrow \prod_v \frac{H^1(K_v,A)}{H^1_\mathcal{F}(K_v,A)}\Biggr),
\] 
where $v$ varies over all the primes of $K$. Similarly, define Selmer structures on $T_s=T/\mathfrak{m}^sT$ by propagation via the canonical map $T\twoheadrightarrow T_s$ and denote by $\Sel_\mathcal{F}(K,T_s)$ the corresponding Selmer group. We can see that the triple $(T_s,\mathcal{F},\mathcal{P}_s)$ satisfies Assumption \ref{ass:Howard-assumption}. As before, we sometimes omit $\mathcal{F}$, $\mathcal{P}$ and $\{\chi_{n,\ell}\}$ (which we assume fixed throughout) in the notation of 
Kolyvagin systems.

Consider $\kappa \in \uks(T)=\uks(T,\mathcal{F},\mathcal{P},\{ \chi_{n,\ell}\})$. As explained in Section \ref{secUKS}, for any integer $s\geq 1$ we write $\kappa^{(s)}\in\varinjlim_{t\geq s}\ks(T_s,\mathcal{F},\mathcal{P}_t,\{ \chi_{n,\ell}\})$ for the canonical projection of $\kappa$ and we call $\mathfrak{K}_s$ its set of representatives in $\bigsqcup_{t \geq s}\ks(T_s,\mathcal{F},\mathcal{P}_t,\{ \chi_{n,\ell}\})$. We will often denote by $\kappa^{(s,t)}$ an element of $\mathfrak{K}_s$ that lies in $\ks(T_s,\mathcal{P}_t)=\ks(T_s,\mathcal{F},\mathcal{P}_t,\{ \chi_{n,\ell}\})$, for $t\ge s$.

\begin{assumption} \label{asskappaDVR}
    There are $\epsilon \in \{ \pm 1\}$ and $\kappa \in \uks(T)$ such that $\kappa_1 \neq 0$ and  $\kappa_n^{(s,t)} \in \Sel_{\mathcal{F}(n)}(K,T_s)^{\epsilon_n}$ for all $s \geq 1$, for all $\kappa^{(s,t)} \in \mathfrak{K}_s$ and for all $n \in \mathcal{N}_t$, where $\epsilon_n\defeq\epsilon\cdot(-1)^{\nu(n)}$.
\end{assumption}

Recall the integer $e(n)\in\{0,1\}$ in Proposition \ref{HoStrArt}. We need some auxiliary results.

\begin{lemma} \label{freeKS} 
Let $s,t$ be positive integers with such that $t \geq 2s-1$. Then any element of $\mathfrak{K}_s\cap \ks(T_s,\mathcal{P}_t)$ is a free Kolyvagin system of sign $\epsilon$.
\end{lemma}


\begin{proof} 
Let $t\ge 2s-1$ and $n\in\mathcal{N}_t$. Let $\Stub_n^{(s)}$ denote the Stub submodule of $\Sel_{\mathcal{F}(n)}(K,T_{s})$ and let $\kappa^{(s,t)}\in \mathfrak{K}_s\cap\ks(T_s,\mathcal{P}_t)$ for some $t\ge 2s-1$. Choose generators of $G_\ell$ for all $\ell\mid n$ and identify $\Sel_{\mathcal{F}(n)}(K,T_s)\otimes G_n$ with $\Sel_{\mathcal{F}(n)}(K,T_s)$. 
Under this identification, by \cite[Lemma 1.6.4]{HoHeeg}, $\kappa^{(s,t)}_n$ belongs to $\Stub_n^{(s)}$. By Assumption \ref{asskappaDVR}, we know that $\kappa_1\ne 0$, hence $\kappa_1^{(s)}\ne 0$ for $s\gg 0$. This implies that, for such an $s$, $\Stub_1^{(s)}\ne 0$, hence $e(1)=1$ (with $T_s$ in place of $T$ in Proposition \ref{HoStrArt}). By Lemma \ref{Ho1.5.3}, $e(n)$ does not depend on $s$ or $n$, so there is an isomorphism
\[
\Sel_{\mathcal{F}(n)}(K,T_{s})\simeq R_s\oplus M_n^{(s)}\oplus M_n^{(s)}
\]
for all $s\ge 1$ and some $R_s$-module $M_n^{(s)}$. Then $\kappa_n^{(s,t)}\in\Stub_n^{(s)}$ is contained in a free $R_s$-component of $\Sel_{\mathcal{F}(n)}(K,T_s)$ and, by Assumption \ref{asskappaDVR}, we also have $\kappa_n\in \Sel_{\mathcal{F}(n)}^{\epsilon_n}(K,T_s)$, which yields the claim. \end{proof}

The existence of a free Kolyvagin system $\kappa^{(s,t)}\in\mathfrak{K}_s$ for every $s\in\N$ allows us to apply all the results of \S \ref{sec:artinian-rings}. In particular, Theorem \ref{artSel} holds for $\kappa^{(s,t)}$ (as, by Proposition \ref{modification-prop2}, one can always find an admissible modification for $\kappa^{(s,t)}$).

In what follows, if $\kappa^{(s,t)}\in\mathfrak{K}_s$, then $\ind\bigl(\kappa_n^{(s,t)}\bigr)$ stands for $\ind\bigl(\kappa_n^{(s,t)},\Sel_{\mathcal{F}(n)}(K,T_s)\bigr)$.

\begin{lemma} \label{lemma3.34}
Let $s_1,s_2,t_1,t_2\geq1$ be integers such that $t_1\ge s_1$ and $t_2\ge s_2$. For any $\kappa^{(s_1,t_1)}\in\mathfrak{K}_{s_1}\cap \ks(T_{s_1},\mathcal{P}_{t_1})$ and  $\kappa^{(s_2,t_2)}\in\mathfrak{K}_{s_2}\cap\ks(T_{s_2},\mathcal{P}_{t_2})$ with $\kappa_1^{(s_1,t_1)}\neq0\neq\kappa_1^{(s_2,t_2)}$, there is an equality 
\[
\ind\bigl(\kappa^{(s_1,t_1)}_{1}\bigr)=\ind\bigl(\kappa^{(s_2,t_2)}_{1}\bigr).
\]
\end{lemma}

\begin{proof} Without loss of generality, suppose $s_2 \geq s_1$. Since $\kappa_1^{(s_1,t_1)}$ and $\kappa_1^{(s_2,t_2)}$ are independent of $t_1$ and $t_2$, we can also assume that $t_2\ge t_1$ and that $t_i \geq 2s_i-1$ for $i=1,2$. Set $\kappa^{(s_1,t_2)}\defeq P_{s_2,s_1}^{(t_2)}\bigl(\kappa^{(s_2,t_2)}\bigr)$ for the image of $\kappa^{(s_2,t_2)}$ under the map
\[
P_{s_2,s_1}^{(t_2)}:\ks(T_{s_2},\mathcal{P}_{t_2}) \longrightarrow \ks(T_{s_1},\mathcal{P}_{t_2})
\]
induced by the canonical projection $T_{s_2}\twoheadrightarrow T_{s_1}$ (\emph{cf.} Remark \ref{remnot}); observe that $\kappa^{(s_1,t_2)}\in\mathfrak{K}_{s_1}$. Then there exist $t\geq\max\{t_1,t_2\}$ and $\kappa^{(s_1,t)}\in\ks(T_{s_1},\mathcal{P}_t)$ such that
\[
F_{t_1,t}^{(s_1)}\bigl(\kappa^{(s_1,t_1)}\bigr)=\kappa^{(s_1,t)}=F_{t_2,t}^{(s_1)}\bigl(\kappa^{(s_1,t_2)}\bigr),
\]
where $F_{t_i,t}^{(s_1)}:\ks(T_{s_1},\mathcal{P}_{t_i})\rightarrow \ks(T_{s_1},\mathcal{P}_t)$ are the canonical maps (\emph{cf.} Remark \ref{remnot}). It follows that there are equalities
\[
\ind\bigl(\kappa^{(s_1,t_1)}_{1}\bigr)=\ind\bigl(\kappa^{(s_1,t)}_{1}\bigr)=\ind\bigl(\kappa^{(s_1,t_2)}_{1}\bigr).
\]
In particular, $\kappa^{(s_1,t_2)}_{1} \neq 0$ and, by Lemma \ref{freeKS}, both $\kappa^{(s_2,t_2)}_{1}$ and $\kappa^{(s_1,t_2)}_{1}$ lie in a direct summand of the respective Selmer groups that is free of rank $1$. 
The commutative triangle
\[ 
\xymatrix@C=45pt@R=40pt{\Sel_{\mathcal{F}}(K,T_{s_2})\ar[r]^-{\pi^{s_2-s_1}\cdot} \ar[d]& \Sel_{\mathcal{F}}(K,T_{s_2})[\mathfrak{m}^{s_1}]\\ \Sel_{\mathcal{F}}(K,T_{s_1})\ar[ru]^-\simeq} 
\] 
coming from \cite[Lemma 1.3.3]{HoHeeg} shows that
\[
\ind\bigl(\kappa^{(s_2,t_2)}_{1}\bigr)=\ind\bigl(\kappa^{(s_1,t_2)}_{1}\bigr),
\]
and the proof is complete.
\end{proof}

Thanks to Lemma \ref{lemma3.34}, we can define the index $\ind(\kappa_1)$ of $\kappa_1=\varprojlim_s\kappa_1^{(s,t)}$ to be $\ind\bigl(\kappa^{(s,t)}_{1}\bigr)$ for any integer $s\geq1$ and any $\kappa^{(s,t)}\in\mathfrak{K}_s$ such that $\kappa^{(s,t)}_{1}\neq0$. Moreover, set
\begin{equation} \label{delta-def-eq}
\delta^{(j)}\bigl(\kappa^{(s)}\bigr)\defeq\max_{\kappa^{(s,t)}\in \mathfrak{K}_{s}}\bigl\{\partial^{(j)}(\kappa^{(s,t)})\bigr\}\in\N, \quad\delta\bigl(\kappa^{(s)}\bigr)\defeq\max_{\kappa^{(s,t)}\in \mathfrak{K}_{s}}\bigl\{\partial(\kappa^{(s,t)})\bigr\}\in\N
\end{equation}
for all integers $s>\ind(\kappa_1)$ and all $j\in\N$.

\begin{remark} 
By Theorem \ref{artSel} and Lemma \ref{freeKS}, there are inequalities 
\[
\partial^{(j+1)}\bigl(\kappa^{(s,t)}\bigr)\leq\partial^{(j)}\bigl(\kappa^{(s,t)}\bigr)\leq\ind(\kappa_1)
\]
for all $j\in\N$ and all $t\geq 2s-1$. Moreover, by definition, for $t\geq t'$ there are inequalities 
\[
\partial^{(j)}\bigl(\kappa^{(s,t)}\bigr)\geq\partial^{(j)}\bigl(\kappa^{(s,t')}\bigr),\quad\partial\bigl(\kappa^{(s,t)}\bigr)\geq\partial\bigl(\kappa^{(s,t')}\bigr),
\]
so the definitions of $\delta^{(j)}\bigl(\kappa^{(s)}\bigr)$ and $\delta\bigl(\kappa^{(s)}\bigr)$ in \eqref{delta-def-eq} are meaningful and give natural numbers.  
\end{remark} 

\begin{lemma} \label{existenceKol}
Let $s>\ind(\kappa_1)$. There is a free $\kappa^{(s,t)}\in\mathfrak{K}_s$ such that
\[
\delta\bigl(\kappa^{(s)}\bigr)=\partial\bigl(\kappa^{(s,t)}\bigr),\quad\delta^{(j)}\bigl(\kappa^{(s)}\bigr)=\partial^{(j)}\bigl(\kappa^{(s,t)}\bigr)
\]
for all $j\in\N$.
\end{lemma}

\begin{proof} Let $\kappa^{(s,t')}\in\mathfrak{K}_s$ be such that $\delta\bigl(\kappa^{(s)}\bigr)=\partial\bigl(\kappa^{(s,t')}\bigr)$ with $t'\geq 2s-1$
and set 
\[
J\defeq\ind(\kappa_1)-\partial\bigl(\kappa^{(s,t')}\bigr)\in\N.
\]
Moreover, choose $\kappa^{(s,t_j)}\in \mathfrak{K}_s$ for $j\in\{0,\dots,J\}$ such that $\delta^{(j)}\bigl(\kappa^{(s)}\bigr)=\partial^{(j)}\bigl(\kappa^{(s,t_j)}\bigr)$. Then there is $\kappa^{(s,t)}\in\mathfrak{K}_s$ with $t\geq\max\{t',t_0,\dots,t_J\}$ such that
\[
F^{(s)}_{t',t}\bigl(\kappa^{(s,t')}\bigr)=\kappa^{(s,t)}=F_{t_j,t}^{(s)}\bigl(\kappa^{(s,t_j)}\bigr)
\]
for all $j\in\{0,\dots,J\}$ and it satisfies the stated properties for $\delta$ and $\delta^{(j)}$ with $0\le j\le J$. 

By Theorem \ref{artSel}, $\partial(\kappa^{(s,t)})\le \partial^{(j)}(\kappa^{(s,t)})\le \ind(\kappa_1)$ for all $j\in\N$ and $\partial^{(j)}(\kappa^{(s,t)})$ is strictly decreasing in $j$ until it stabilizes at $\partial(\kappa^{(s,t)})$ for all $j\ge J$. Therefore, there are equalities 
\[
\partial^{(J)}(\kappa^{(s,t)})=\partial(\kappa^{(s,t)})=\delta(\kappa^{(s)})=\delta^{(J)}(\kappa^{(s)}).
\]
Since $\partial(\kappa^{(s,t)})\le \partial^{(j)}(\kappa^{(s,t)})\le \partial^{(J)}(\kappa^{(s,t)})$ and $\delta(\kappa^{(s,t)})\le \delta^{(j)}(\kappa^{(s,t)})\le \delta^{(J)}(\kappa^{(s,t)})$ for all $j>J$, we obtain the claimed relations also for the remaining $j$'s.
\end{proof}

From the proof of the previous lemma, it is also clear that if $\kappa^{(s,t)}\in\mathfrak{K}_s\cap\ks(T_s,\mathcal{P}_t)$ satisfies the statement of the lemma, then $F^{(s)}_{t,t'}\bigl(\kappa^{(s,t)}\bigr)\in \mathfrak{K}_s\cap\ks(T_s,\mathcal{P}_{t'})$ satisfies the same relations for every $t'\ge t$.

\begin{lemma} \label{boundIncrSeq}
The sequences $\bigl(\delta(\kappa^{(s)})\bigr)_{s>\ind(\kappa_1)}$ and $\bigl(\delta^{(j)}(\kappa^{(s)})\bigr)_{s>\ind(\kappa_1)}$ are constant.
\end{lemma}

\begin{proof} Let $s_2\geq s_1> \ind(\kappa_1)$ be integers and let $i \in \{1,2\}$. By Lemma \ref{existenceKol}, there exist $\kappa^{(s_i,t_i)} \in \mathfrak{K}_{s_i}$ such that
\[
\delta\bigl(\kappa^{(s_i)}\bigr)=\partial\bigl(\kappa^{(s_i,t_i)}\bigr),\quad\delta^{(j)}\bigl(\kappa^{(s_i)}\bigr)=\partial^{(j)}\bigl(\kappa^{(s_i,t_i)}\bigr)
\]
for all $j\ge 0$. Without loss of generality, we can assume $t_i>2s_i-1$ and $t_2>t_1$. In particular, by Lemma \ref{freeKS}, the systems $\kappa^{(s_i,t_i)}$ are free. Setting $\kappa^{(s_1,t_2)} \defeq P^{(t_2)}_{s_2,s_1}\bigl(\kappa^{(s_2,t_2)}\bigr)$, the same argument that occurs at the end of the proof of Lemma \ref{lemma3.34} shows that
\[
\partial\bigl(\kappa^{(s_1,t_2)}\bigr)=\partial\bigl(\kappa^{(s_2,t_2)}\bigr),\quad\partial^{(j)}\bigl(\kappa^{(s_1,t_2)}\bigr)=\partial^{(j)}\bigl(\kappa^{(s_2,t_2)}\bigr).
\]
Then there exist $t \geq \max\{t_1,t_2\}$ and $\kappa^{(s_1,t)} \in \ks(T_{s_1},\mathcal{P}_t)$ such that
\[
F^{(s_1)}_{t_1,t}\bigl(\kappa^{(s_1,t_1)}\bigr)=\kappa^{(s_1,t)}=F^{(s_2)}_{t_2,t}\bigl(\kappa^{(s_1,t_2)}\bigr)
\]
and
\[
\partial\bigl(\kappa^{(s_1,t_i)}\bigr)=\partial\bigl(\kappa^{(s_1,t)}\bigr),\quad\partial^{(j)}\bigl(\kappa^{(s_1,t_i)}\bigr)=\partial^{(j)}\bigl(\kappa^{(s_1,t)}\bigr)
\]
for all $i\in\{0,1\}$ and $j\ge 0$. The proof is complete. 
\end{proof}

Thanks to Lemma \ref{boundIncrSeq}, we can set
\[
\delta(\kappa)\defeq\lim_{s\rightarrow\infty}\delta\bigl(\kappa^{(s)}\bigr),\quad\delta^{(j)}(\kappa)\defeq\lim_{s\rightarrow\infty}\delta^{(j)}\bigl(\kappa^{(s)}\bigr).
\]

\begin{theorem}[Structure theorem] \label{thmDVR}
Let $\kappa\in\uks\bigl(T,\mathcal{F},\mathcal{P},\{\chi_{n,\ell}\}\bigr)$ satisfying Assumption 3.32 for a suitable $\epsilon\in\{\pm 1\}$. For all $i\in\N$, set 
\[
d^\epsilon_i\defeq\delta^{(2i+1)}(\kappa)-\delta^{(2i+2)}(\kappa),\quad d^{-\epsilon}_i\defeq\delta^{(2i)}(\kappa)-\delta^{(2i+1)}(\kappa).
\]
Then there are isomorphisms
\[
\Sel^\epsilon_\mathcal{F}(K,A)\simeq(F/R)\oplus\bigoplus_{i\geq0}\bigl(R/\mathfrak{m}^{d^\epsilon_i}\bigr)^2,\quad\Sel^{-\epsilon}_\mathcal{F}(K,A)\simeq\bigoplus_{i\geq0}\bigl(R/\mathfrak{m}^{d^{-\epsilon}_i}\bigr)^2.
\]
Furthermore, the sequences $\bigl(d^{\pm \epsilon}_i\bigr)_{i\in\N}$ are non-increasing.
\end{theorem}
\begin{proof} By \cite[Proposition B.2.7]{RubinES}, the $R$-module $\Sel_{\mathcal{F}}(K,A)$ is cofinitely generated, so there are isomorphisms
\[
\Sel^\epsilon_{\mathcal{F}}(K,A) \simeq \left( F/R \right) ^{r^\epsilon} \oplus \bigoplus_{i \geq 0 }  R/ \mathfrak{m}^{c^\epsilon_i}, \quad \Sel^{-\epsilon}_{\mathcal{F}}(K,A) \simeq \left( F/R \right) ^{r^{-\epsilon}} \oplus \bigoplus_{i \geq 0 }  R/ \mathfrak{m}^{c^{-\epsilon}_i}
\]
for some $r^{\pm\epsilon}\in\N$ and ${(c^{\pm \epsilon}_i)}_{i\in\N}$ non-increasing sequences of non-negative integers such that $c^{\pm \epsilon}_i=0$ for $i\gg0$. For all integers $s\geq1$, there are isomorphisms
\begin{equation} \label{eqCT}
\Sel^{\pm \epsilon}_\mathcal{F}(K,T_s) \simeq \Sel^{\pm \epsilon}_{\mathcal{F}}(K,A)[\mathfrak{m}^s] \simeq\bigl(R/\mathfrak{m}^s \bigr)^{r^{\pm \epsilon}} \oplus \bigoplus_{i \geq 0} R/\mathfrak{m}^{\min \{s,c^{\pm \epsilon}_i \}}.
\end{equation} 
If we choose an integer $s\geq1$ such that $s > c^{\pm \epsilon}_i$ for all $i$, then we can rewrite \eqref{eqCT} as 
\begin{equation} \label{eqCT1}
\Sel^{\pm \epsilon}_\mathcal{F}(K,T_s) \simeq\bigl(R/\mathfrak{m}^s\bigr)^{r^{\pm \epsilon}} \oplus  
\bigoplus_{i \geq 0} R/\mathfrak{m}^{c^{\pm \epsilon}_i}.
\end{equation}
We can also assume $s>\ind(\kappa_1)$ and recall that, by Lemmas \ref{freeKS} and \ref{existenceKol}, there are $t\gg 0$ and a free $\kappa^{(s,t)} \in \mathfrak{K}_s\cap\ks(T_s,\mathcal{P}_t)$ of sign $\epsilon$ such that
\begin{equation}\label{eq:structureDVR}
    \delta(\kappa^{(s)})=\partial(\kappa^{(s,t)}), \:\:\:\:\delta^{(j)}(\kappa^{(s)})=\partial^{(j)}(\kappa^{(s,t)})
\end{equation}
for all $j\in\N$; observe that $\kappa^{(s,t)}_{1} \neq 0$ in $\Sel^\epsilon_\mathcal{F}(K,T_s)$. Choose integers $d^{\pm \epsilon}_i\geq 0$ as in Theorem \ref{artSel}, so that 
\begin{equation} \label{eqAR}
\Sel^\epsilon_\mathcal{F}(K,T_s) \simeq R/\mathfrak{m}^s\oplus \bigoplus_{i\geq 0} \Bigl(R/\mathfrak{m}^{d^\epsilon_i}\Bigr)^2, \quad \Sel^{-\epsilon}_\mathcal{F}(K,T_s) \simeq  \bigoplus_{i\geq 0} \Bigl(R/\mathfrak{m}^{d^{-\epsilon}_i}\Bigr)^2
\end{equation} 
with 
\begin{equation}  \label{equality2}
\partial^{(j)}\bigl(\kappa^{(s,t)}\bigr)=\begin{cases} \min\bigl\{s,\partial(\kappa^{(s,t)})+\sum_{i\geq j/2}(d_i^\epsilon + d_i^{-\epsilon})\bigr\} & \text{if $j$ is even}\\[2mm]
\min\bigl\{s,\partial(\kappa^{(s,t)})+\sum_{i\geq(j+1)/2}(d_i^\epsilon+d_i^{-\epsilon})+d_{(j-1)/2}^{\epsilon}\bigr\} & \text{if $j$ is odd}
\end{cases}
\end{equation}
for all $j\in\N$. Then $r^\epsilon$ is odd, $r^{-\epsilon}$ is even and there are equalities
\[
d^\epsilon_i=\begin{cases}s&\text{for $0\leq i\leq\displaystyle{\frac{r^\epsilon-1}{2}}-1$}\\[3mm]c^\epsilon_{2i-r^\epsilon+1}=c^\epsilon_{2i-r^\epsilon+2}&\text{for $i\geq\displaystyle{\frac{r^\epsilon-1}{2}}$}\end{cases}
\]
and
\[
d^{-\epsilon}_i=\begin{cases}s&\text{for $0\leq i\leq\displaystyle{\frac{r^{-\epsilon}}{2}}-1$}\\[3mm]c^{-\epsilon}_{2i-r^{-\epsilon}}=c^{-\epsilon}_{2i-r^{-\epsilon}+1}&\text{for $i\geq\displaystyle{\frac{r^{-\epsilon}}{2}}$}\end{cases}.
\]
Therefore, there are isomorphisms
\[
\Sel^\epsilon_\mathcal{F}(K,T_s)\simeq R/\mathfrak{m}^s \oplus\Bigl(\bigl(R/\mathfrak{m}^s\bigr)^2 \Bigr)^{\frac{r^\epsilon-1}{2}}\oplus\bigoplus_{i \geq \frac{r^\epsilon-1}{2}}\bigl(R/\mathfrak{m}^{d^\epsilon_i}\bigr)^2
\]
and
\[
\Sel^{-\epsilon}_\mathcal{F}(K,T_s)\simeq \Bigl(\bigl(R/\mathfrak{m}^s\bigr)^2 \Bigr)^{\frac{r^{-\epsilon}}{2}}\oplus\bigoplus_{i \geq \frac{r^{-\epsilon}}{2}}\bigl(R/\mathfrak{m}^{d^{-\epsilon}_i}\bigr)^2.
\]
Since $\bigl(\delta^{(0)}(\kappa^{(s)})\bigr)_{s>\ind(\kappa_1)}$ is constant by Lemma \ref{boundIncrSeq}, from \eqref{equality2} with $j=0$ we find that $d^{\pm \epsilon}_i < s$ for all $i\in\N$ (and $s\gg0$), whence $r^\epsilon=1$ and $r^{-\epsilon}=0$. Moreover, since $\bigl(\delta^{(j)}(\kappa^{(s)})\bigr)_{s>\ind(\kappa_1)}$ is constant (again by Lemma \ref{boundIncrSeq}), equalities \eqref{eq:structureDVR} and \eqref{equality2} show that for $s\gg0$ the equality
\begin{equation}\label{eqpart}
\delta^{(j)}(\kappa^{(s)})=\begin{cases}\delta(\kappa^{(s)})+\sum_{i\geq j/2} (d^\epsilon_i+d^{-\epsilon}_i) & \text{if $j$ is even}\\[2mm]
\delta(\kappa^{(s)})+\sum_{i\geq(j+1)/2}(d_i^\epsilon+d_i^{-\epsilon})+d^\epsilon_{(j-1)/2} & \text{if $j$ is odd}
\end{cases}
\end{equation}
holds for all $j\in\N$. Letting $s\rightarrow+\infty$ in \eqref{eqpart} yields the equality
\begin{equation} \label{eqpart1}
\delta^{(j)}(\kappa)=\begin{cases}\delta(\kappa)+\sum_{i\geq j/2} (d^\epsilon_i+d^{-\epsilon}_i) & \text{if $j$ is even}\\[2mm]
\delta(\kappa)+ \sum_{i \geq (j+1)/2}(d_i^\epsilon+d_i^{-\epsilon})+d^\epsilon_{(j-1)/2} & \text{if $j$ is odd}
\end{cases},
\end{equation}
so for $i=j$ we get 
\[
d^\epsilon_i=\delta^{(2i+1)}(\kappa)-\delta^{(2i+2)}(\kappa),\quad d^{-\epsilon}_i=\delta^{(2i)}(\kappa)-\delta^{(2i+1)}(\kappa).
\]
The proof of the theorem is complete. \end{proof}

Denote by
\[
\Sha_\mathcal{F}(K,A)\defeq\Sel_\mathcal{F}(K,A)\big/\divv\bigl(\Sel_\mathcal{F}(K,A)\bigr)
\]
the Shafarevich--Tate group of $A$ over $K$ with respect to $\mathcal{F}$, where
$\divv\bigl(\Sel_\mathcal{F}(K,A)\bigr)$ is the maximal divisible subgroup of $\Sel_\mathcal{F}(K,A)$. Moreover, write $X(T)$ for the Pontryagin dual of $\Sha_\mathcal{F}(K,A)$, \emph{i.e.}, set
\[
X(T)\defeq\Hom\Bigl(\Sha_\mathcal{F}(K,A),\Q_p/\Z_p\Bigr).
\]
Theorem \ref{thmDVR} implies that there are isomorphisms
\begin{equation}\label{ShaX(T)}
X^{\pm \epsilon}(T) \simeq \bigoplus_{i\geq0}\bigl(R/\mathfrak{m}^{d^{\pm \epsilon}_i}\bigr)^2
\end{equation}
with $d^\epsilon_i=\delta^{(2i+1)}(\kappa)-\delta^{(2i+2)}(\kappa)$ and $d^{-\epsilon}_i=\delta^{(2i)}(\kappa)-\delta^{(2i+1)}(\kappa)$. 
In particular, by Theorem \ref{FittingDVR}, the equalities
\begin{equation}\label{eq:fitting-DVR+}
\Fitt_i\bigl(X^\epsilon(T)\bigr)=\mathfrak{m}^{2\sum_{s \geq i/2}(\delta^{(2s+1)}(\kappa)-\delta^{(2s+2)}(\kappa))}
\end{equation}
and
\begin{equation}\label{eq:fitting-DVR-}
\Fitt_i\bigl(X^{-\epsilon}(T)\bigr)=\mathfrak{m}^{2\sum_{s \geq i/2}(\delta^{(2s)}(\kappa)-\delta^{(2s+1)}(\kappa))}
\end{equation}
hold for all even $i\in\N$.

\subsection{Universal Kolyvagin systems over $\Lambda$}

We turn our attention to universal Kolyvagin systems over $\Lambda$.

\subsubsection{Assumptions} 

Let $\Lambda\defeq\cO\llbracket T\rrbracket$ be the completed power series ring in the variable $T$ with coefficients in a discrete valuation ring $\cO$ that is finite over $\Z_p$; let $\pi$ be a uniformizer of $\cO$, let $\mathfrak{m}=(\pi)$ be its maximal ideal and let $\F\defeq\cO/\mathfrak{m}$ its residue field. Denote by $\mathfrak{m}_\Lambda=(\pi,T)$ the maximal ideal of $\Lambda$. 

Now let $\mathbf{T}$ be a free $\Lambda$-module of rank $2$ endowed with a continuous action of $G_\Q$. Set $\mathbf{A}\defeq\Hom_{\cont}(\mathbf{T},\Bmu_{p^\infty})$, where $\Bmu_{p^\infty}\subset\overline\Q$ is the group of roots of unity of $p$-power order, and, by an abuse of notation, denote by $\mathcal{F}$ also the Selmer structure for $\mathbf{A}$ defined by letting $H^1_\mathcal{F}(K_v,\mathbf{A})$ be the orthogonal complement of $H^1_\mathcal{F}(K_v,\mathbf{T})$ with respect to the local Tate pairing for all primes $v$ of $K$. The Selmer group of $\mathbf{A}$ over $K$ with respect to $\mathcal{F}$ is
\[
\Sel_\mathcal{F}(K,\mathbf{A}) \defeq \ker \Biggl(H^1(K,\mathbf{A})\longrightarrow \prod_v \frac{H^1(K_v,\mathbf{A})}{H^1_\mathcal{F}(K_v,\mathbf{A})}\Biggr),
\]
where $v$ varies over all the primes of $K$. As in \cite[Proposition 2.1.5]{MR}, one can show that $\Sel_\mathcal{F}(K,\mathbf{T})$ is a finitely generated $\Lambda$-module and, by \cite[Lemma 2.3.3]{MR}, $\Sel_\mathcal{F}(K,\mathbf{A})$ is a cofinitely generated $\Lambda$-module. The \emph{big Shafarevich--Tate group} of $\mathbf{A}$ over $K$ with respect to $\mathcal{F}$ is
\[
\Sha_\mathcal{F}(K,\mathbf{A})\defeq \Sel_\mathcal{F}(K,\mathbf{A})\big/\divv_{\mathfrak{m}_\Lambda} \bigl(\Sel_\mathcal{F}(K,\mathbf{A})\bigr),
\]
where $\divv_{\mathfrak{m}_\Lambda}\bigl(\Sel_\mathcal{F}(K,\mathbf{A})\bigr)$ is the maximal $\mathfrak{m}_\Lambda$-divisible subgroup of $\Sel_\mathcal{F}(K,\mathbf{A})$; consider its Pontryagin dual
\[ 
X \defeq\Sha_\mathcal{F}(K,\mathbf{A})^\vee.
\]
Let $\mathfrak{P}$ be a height $1$ prime ideal of $\Lambda$. Let $\mathcal{O}_{\mathfrak{P}}$ be the integral closure of $\Lambda/\mathfrak{P}$, which is a discrete valuation ring such that $[\mathcal{O}_{\mathfrak{P}}:\Lambda/\mathfrak{P}]<\infty$; let $\pi_\mathfrak{P}$ be a uniformizer of $\mathcal{O}_{\mathfrak P}$. The composition of the canonical projection $\Lambda \twoheadrightarrow \Lambda/\mathfrak{P}$ and the canonical inclusion $\Lambda/\mathfrak{P}\hookrightarrow \mathcal{O}_{\mathfrak P}$ induces a map $s_\mathfrak{P}:\Lambda\rightarrow \mathcal{O}_\mathfrak{P}$. Let $F_{\mathfrak P}$ be the fraction field of $\mathcal O_{\mathfrak P}$ and set 
\begin{equation} \label{T-V-A-eq}
T_{\mathfrak{P}}\defeq\mathbf{T} \otimes_\Lambda \mathcal{O}_{\mathfrak{P}},\quad V_\mathfrak{P}\defeq T_\mathfrak{P}\otimes_{\mathcal{O}_\mathfrak{P}}F_\mathfrak{P},\quad A_\mathfrak{P}\defeq V_\mathfrak{P}/T_\mathfrak{P},
\end{equation}
where the tensor product in the definition of $T_{\mathfrak{P}}$ is taken with respect to $s_{\mathfrak P}$. It follows that $T_{\mathfrak{P}}$ is a free $\mathcal{O}_{\mathfrak{P}}$-module of rank $2$ equipped with a continuous action of $G_\Q$. By an abuse of notation, the obvious map $\mathbf{T}\rightarrow T_\mathfrak{P}$ and the map $H^1(K,\mathbf{T})\rightarrow H^1(K,T_\mathfrak{P})$ induced in cohomology will be still denoted by $s_\mathfrak{P}$.

We assume that for each height-$1$ prime ideal $\mathfrak{P}$ of $\Lambda$ we are given a Selmer structure $\mathcal{F}_\mathfrak{P}$ on $V_\mathfrak{P}$; then, for all integers $n\geq1$, we define by propagation Selmer structures on $A_\mathfrak{P}$ and $T_\mathfrak{P}$, still denoted by $\mathcal{F}_\mathfrak{P}$. Moreover, we assume that $T_\mathfrak{P}$ is endowed with a perfect, symmetric, $\mathcal{O}_\mathfrak{P}$-bilinear pairing
\[
{(\cdot,\cdot)}_\mathfrak{P}:T_\mathfrak{P} \times T_\mathfrak{P}\longrightarrow\cO_\mathfrak{P}(1)
\]
such that $(x^\sigma,y^{\tau_c \sigma \tau_c})=(x,y)^\sigma$ for all $\sigma \in G_K$, where $\tau_c \in G_\Q$ is a fixed complex conjugation. In particular, this pairing induces a perfect, $G_K$-equivariant, $\cO_\mathfrak{P}$-bilinear pairing
\[
{(\cdot,\cdot)}_\mathfrak{P}:T_{\mathfrak{P}}\times A_\mathfrak{P}\longrightarrow \Bmu_{p^\infty}.
\]
In addition, we assume that $\mathcal{F}_\mathfrak{P}$ is self-dual. Observe that the dual of the map $s_{\mathfrak{P}}:\mathbf{T}\rightarrow T_{\mathfrak{P}}$ induces a canonical map $A_\mathfrak{P}\rightarrow \mathbf{A}$. We will work under the following two assumptions.

\begin{assumption} \label{asstripleR}
For every height-$1$ prime ideal $\mathfrak{P}$ of $\Lambda$, the properties below hold true:
\begin{enumerate}
\item the triple $(T_\mathfrak{P},\mathcal{F}_\mathfrak{P},\mathcal{P})$ satisfies Assumption \ref{ass:Howard-assumption};
\item the Selmer structure $\mathcal{F}_\mathfrak{P}$ is compatible 
with the Selmer structure $\mathcal{F}$ on $\mathbf{T}$ in the sense that 
the specialization map $s_\mathfrak{P}:H^1(K,\mathbf{T})\rightarrow H^1(K,T_\mathfrak{P})$ induces a map 
\begin{equation} \label{smap}
s_\mathfrak{P}:\Sel_{\mathcal{F}}(K,\mathbf{T})/\mathfrak{P}\longrightarrow \Sel_{\mathcal{F}_\mathfrak{P}}(K,T_\mathfrak{P})
\end{equation}
and the specialization map $i_\mathfrak{P}:H^1(K,A_\mathfrak{P}) \rightarrow H^1(K,\mathbf{A})$ induces a map
\begin{equation} \label{imap}
i_\mathfrak{P}:\Sel_{\mathcal{F}_\mathfrak{P}}(K,A_\mathfrak{P})\longrightarrow \Sel_\mathcal{F}(K,\mathbf{A})[\mathfrak{P}];
\end{equation}
\item (\emph{control theorem}): there is a finite set $\Sigma_\mathrm{CT}$ of height-$1$ prime ideals of $\Lambda$ such that for all $\mathfrak{P}\notin\Sigma_\mathrm{CT}$ the kernel and the cokernel of the maps $s_\mathfrak{P}$ and $i_\mathfrak{P}$ are finite, of order bounded by a constant that depends on $[\mathcal{O}_\mathfrak{P}:\Lambda/\mathfrak{P}]$ but not on $\mathfrak{P}$ itself. 
\end{enumerate}
\end{assumption}

Let $\kappa\in\uks(\mathbf{T},\mathcal F)=\uks\bigl(\mathbf{T},\mathcal{F},\mathcal{P},\{\chi_{n,\ell}\}\bigr)$, let $\mathfrak{P}$ be a height-$1$ prime ideal of $\Lambda$ and write $\kappa_\mathfrak{P}$ for the universal Kolyvagin system for $(T_\mathfrak{P},\mathcal{F}_\mathfrak{P},\mathcal{P})$ that is the image of $\kappa$ under the natural map $\uks(\mathbf{T},\mathcal{F})\longrightarrow \uks(T_\mathfrak{P},\mathcal{F}_\mathfrak{P})$. 

\begin{assumption} \label{asskolyR}
There are $\kappa\in\uks\bigl(\mathbf{T},\mathcal F)$ and a finite set $\Sigma_{\EX}\supset\Sigma_{\mathrm{CT}}$ of height-$1$ prime ideals of $\Lambda$ such that $\kappa_{\mathfrak{P},1}\neq0$ for all $\mathfrak{P}\notin\Sigma_{\EX}$. Furthermore, there is $\epsilon\in\{\pm1\}$ such that $\kappa_n^{(s,t)} \in \Sel_{\mathcal{F}(n)}(K,T_s)^{\epsilon_n}$ for all $s\geq1$, all $\kappa^{(s,t)}\in \mathfrak{K}_s$ and all $n\in\mathcal{N}_t$, where $\epsilon_n\defeq\epsilon\cdot(-1)^{\nu(n)}$.
\end{assumption}

In particular, $\kappa_{\mathfrak{P}}$ satisfies Assumption \ref{asskappaDVR} for all $\mathfrak{P}\notin\Sigma_{\EX}$.

\begin{proposition} \label{corank-prop}
The $\Lambda$-corank of $\Sel_{\mathcal{F}}(K,\mathbf{A})$ is $1$.
\end{proposition}

\begin{proof} Combining part (3) of Assumption \ref{asstripleR}, the non-triviality condition in Assumption \ref{asskolyR} and Theorem \ref{thmDVR}, the proposition follows from an argument analogous to that in \cite[Corollary 3.25]{DRLV}. \end{proof}

The map $i_\mathfrak{P}$ in \eqref{imap} takes the maximal divisible subgroup of $\Sel_{\mathcal{F}_\mathfrak{P}}(K,A_\mathfrak{P})$ to the maximal $\mathfrak{m}_\Lambda$-divisible subgroup $\mathrm{div}_{\mathfrak{m}_\Lambda}\bigl(\Sel_\mathcal{F}(K,\mathbf{A})\bigr)$ of $\Sel_\mathcal{F}(K,\mathbf{A})$; thus, $i_{\mathfrak P}$ induces a map
\[
i_\mathfrak{P}:\Sha_{\mathcal{F}_\mathfrak{P}}(K,A_\mathfrak{P})\longrightarrow \Sha_\mathcal{F}(K,\mathbf{A})[\mathfrak{P}],
\]
still denoted by the same symbol. Therefore, Pontryagin duality gives a map 
\begin{equation} \label{cospShadual}
i^\vee_\mathfrak{P}:X /\mathfrak{P} X\longrightarrow X({T}_{\mathfrak{P}}).
\end{equation}
Taking $\mathfrak{P}$-torsion in the short exact sequence of $\Lambda$-modules 
\[
0\longrightarrow\mathrm{div}_{\mathfrak{m}_\Lambda}\bigl(\Sel_\mathcal{F}(K,\mathbf{A})\bigr)\longrightarrow\Sel_\mathcal{F}(K,\mathbf{A})\longrightarrow\Sha_\mathcal{F}(K,\mathbf{A})\longrightarrow0
\]
defining $\Sha_\mathcal{F}(K,\mathbf{A})$, we obtain an exact sequence of $\Lambda$-modules
\begin{equation} \label{shaP}
0\longrightarrow\mathrm{div}_{\mathfrak{m}_\Lambda}\bigl(\Sel_\mathcal{F}(K,\mathbf{A})\bigr)[\mathfrak{P}]\longrightarrow\Sel_\mathcal{F}(K,\mathbf{A})[\mathfrak{P}]\longrightarrow\Sha_\mathcal{F}(K,\mathbf{A})[\mathfrak{P}].
\end{equation}

\begin{lemma} \label{right-onto-lemma}
The rightmost map in \eqref{shaP} is surjective. 
\end{lemma}

\begin{proof} Proceed as in the proof of \cite[Lemma 3.26]{DRLV}.\end{proof}

\begin{proposition} \label{CTSha1}
For all height-$1$ prime ideals $\mathfrak{P}$ of $\Lambda$ not in $\Sigma_{\EX}$, the kernel and the cokernel of the map $i^\vee_\mathfrak{P}$ in $\eqref{cospShadual}$ are finite of order bounded by a constant depending on $[\mathcal{O}_{\mathfrak{P}}:\Lambda/\mathfrak{P}]$. 
\end{proposition}

\begin{proof} Using Proposition \ref{corank-prop} and Lemma \ref{right-onto-lemma}, one can mimic the arguments in the proof of \cite[Proposition 3.27]{DRLV}. \end{proof}

The map $i_\mathfrak{P}^\vee$ in \eqref{cospShadual} induces a map
\begin{equation} \label{ipvee}
i_\mathfrak{P}^\vee: X/\mathfrak{P}X \otimes_{\Lambda/\mathfrak{P}} \mathcal{O}_\mathfrak{P} \longrightarrow X(T_\mathfrak{P}) \otimes_{\Lambda/\mathfrak{P}} \mathcal{O}_\mathfrak{P},
\end{equation}
denoted by the same symbol. Composing with the map $X(T_\mathfrak{P}) \otimes_{\Lambda/\mathfrak{P}} \mathcal{O}_\mathfrak{P}\rightarrow 
X(T_\mathfrak{P})$ defined on pure tensors by $x\otimes a\mapsto ax$, we obtain a further map
\begin{equation} \label{ipvee2}
i_\mathfrak{P}^\vee: X/\mathfrak{P}X \otimes_{\Lambda/\mathfrak{P}} \mathcal{O}_\mathfrak{P} \longrightarrow X(T_\mathfrak{P}).
\end{equation}
Note that $i_\mathfrak{P}^\vee$ is $G_\Q$-equivariant, hence it induces maps
\[
(i_\mathfrak{P}^\vee)^\pm:X^\pm \otimes_\Lambda \mathcal{O}_\mathfrak{P} \rightarrow X^\pm(T_\mathfrak{P}).
\]
For a height-$1$ prime ideal $\mathfrak{P}$ of $\Lambda$ and $j\in\N$, we introduce the following notation:  
\begin{itemize}
\item if $\mathfrak{P}=(f)$ with $f\neq\pi$, then set $\mathfrak{P}_j\defeq(f+\pi^j)$ for all $j\ge 0$ and write $e_\mathfrak{P}$ for the ramification index of $\Frac(\cO_{\mathfrak{P}})$ over $\Frac(\cO)$;
\item if $\mathfrak{P} = (\pi)$, then set $\mathfrak{P}_j\defeq(\pi + T^j)$ and $e_\mathfrak{P}\defeq1$.
\end{itemize}
There is an integer $N(\mathfrak{P})\geq1$ such that for all $j\geq N(\mathfrak{P})$ the following properties hold:
\begin{itemize}
\item $\mathfrak{P}_j$ is a height-$1$ prime ideal of $\Lambda$;
\item $\mathfrak{P}_j\notin\Sigma_{\EX}$;
\item if $\mathfrak{P}\neq(\pi)$, then there is an isomorphism of rings $\Lambda/\mathfrak{P}_j\simeq\Lambda/\mathfrak{P}$.
\end{itemize}
For all integers $j \geq N(\mathfrak{P})$, set
\[
\mathcal{O}_j\defeq\mathcal{O}_{\mathfrak{P}_j},\quad T_j\defeq {T}_{\mathfrak{P}_j},\quad \kappa_j\defeq\kappa_{\mathfrak{P}_j},\quad\kappa_{j,n}\defeq(\kappa_{\mathfrak{P}_j}{)}_n,\quad i_j^\vee\defeq i_{\mathfrak{P}_j}^\vee.
\]
Moreover, denote by $\pi_j$ a uniformizer of $\cO$ and by $e_j$ the ramification index of $\Frac(\cO_j)$ over $\Frac(\cO)$. Finally, let
\begin{equation} \label{s-j-eq}
s_j:\Lambda\longrightarrow\Lambda/\mathfrak{P}_j\longmono\mathcal O_j
\end{equation}
be the specialization map.

\begin{lemma} \label{extScalars}
Let $\mathfrak{P}$ be a height-$1$ prime ideal of $\Lambda$ not contained in $\Sigma_\mathrm{EX}$. The kernel and the cokernel of $(i_j^\vee)^\pm$ are finite and bounded by a constant depending on $\mathfrak{P}$ but not on $j$ for $j\geq N(\mathfrak{P})$.
\end{lemma}

\begin{proof}
   It follows directly from \cite[Lemma 3.29]{DRLV}.
\end{proof}

We give a first result on the structure of $X^\pm$.

\begin{proposition} \label{propX=M+M}
There exist torsion $\Lambda$-modules $M^{\pm }$ such that $X^{\pm } \sim M^{\pm } \oplus M^{\pm }$.
\end{proposition}

\begin{proof} Let $\epsilon \in \{\pm 1\}$. Fix a prime ideal $\mathfrak{P}$ of $\Lambda$ such that $\mathfrak{P}\,|\,\charr(X^\epsilon)$, where $\charr(X^\epsilon)$ denotes the characteristic ideal of $X^\epsilon$. We assume
\[
X^\epsilon \sim N=N' \oplus N_{\mathfrak{P}},
\]
where $\charr (N')$ is coprime with $\mathfrak{P}$ and $N_{\mathfrak{P}}\simeq\bigoplus_{i=1}^n \Lambda/\mathfrak{P}^{e_i}$. We need to show that, for all integers $e\geq1$, the number of indices $i$ such that $e_i=e$ is even. Considering the composition 
\[
N_\mathfrak{P}\otimes_\Lambda\cO_j\longrightarrow N\otimes_\Lambda\cO_j\longrightarrow X^\epsilon\otimes_\Lambda\cO_j\longrightarrow X^\epsilon(T_j)=M_j\oplus M_j,
\]
the claim follows from arguments similar to those in the proof of \cite[Proposition 3.32]{DRLV}. \end{proof}

\subsubsection{Construction of the ideals $\mathfrak{D}^{\pm}_i$} \label{constrideals}

Let $r \geq 0$ and $s \geq 1$ be integers. For all $\kappa^{(s,t)}\in\mathfrak{K}_s$ with $t \geq 2s$, define
\[
\mathfrak{C}_r(\kappa^{(s,t)})\defeq\Bigl(\loc_\ell(\kappa^{(s,t)}_{n})\;\Big|\;n\in\mathcal{N}_{t},\; \ell\in\mathcal{P}_{t},\;\ell\nmid n,\;\nu(n)\leq r\Bigr)\subset \Lambda/\mathfrak{m}_\Lambda^s,
\]
where the inclusion is obtained by considering localizations as maps
\[
\Sel^\pm_{\mathcal{F}(n)}(K,\mathbf{T}_s)\longrightarrow  H^1_{\f}(K_\lambda,\mathbf{T}_s)^\pm\simeq\Lambda/\mathfrak{m}_\Lambda^s,
\]
the isomorphism being induced by a fixed isomorphism $\mathbf{T}^\pm\simeq\Lambda$. It follows that there are isomorphisms $H^1_{\f}(K_\lambda,\mathbf{T}_s)^\pm\simeq\mathbf{T}_s^\pm\simeq \Lambda/\mathfrak{m}_\Lambda^s$. Now consider the ideals
\begin{equation} \label{defideals}
\mathfrak{C}_r(s)\defeq\bigcap_{\kappa^{(s,t)}\in\mathfrak{K}_{s},t\geq 2s}\mathfrak{C}_r\bigl(\kappa^{(s,t)}\bigr)
\end{equation}
of $\Lambda/\mathfrak{m}_\Lambda^s$.

\begin{lemma}\label{defC_rlemma}
The ideals $\mathfrak{C}_r(s)$ form an inverse system varying $s \geq 1$.
\end{lemma}

\begin{proof}
This is clear from the compatibility of the localization maps. 
\end{proof}

Thanks to the previous lemma we can define
\[
\mathfrak{C}_r \defeq \varprojlim_{s} \mathfrak{C}_r(s) \subset \Lambda.
\]
Now let $i\in\N$ be even. Since $(\mathfrak{C}_r)_{r \geq 0}$ is a non-decreasing sequence of ideals, by Proposition \ref{monoid} we can define
\[
[\mathfrak{D}_i^\epsilon] \defeq \prod_{r \geq i/2} \frac{[\mathfrak{C}_{2r+1}]}{[\mathfrak{C}_{2r+2}]},\quad [\mathfrak{D}_i^{-\epsilon}] \defeq \prod_{r\geq i/2} \frac{[\mathfrak{C}_{2r}]}{[\mathfrak{C}_{2r+1}]},
\]
where $\mathfrak{D}_i^\pm$ are some chosen representatives of the corresponding classes in $I(\Lambda)/\sim$ (see Section \ref{sec:fitting-over-Lambda}). Since $\Lambda$ is noetherian, $[\mathfrak{C}_r]=[\mathfrak{C}_{r+1}]$ for all $r \gg 0$, so the infinite products in the previous definitions are actually finitely supported (recall that, by Proposition \ref{monoid}, the monoid $I(\Lambda)/\sim$ is cancellative).

\subsubsection{Main result}

Given two sequences of real numbers ${(x_j)}_{j\geq1}$ and ${(y_j)}_{j\geq1}$, we write $x_j \prec y_j$ if $\liminf(y_j-x_j)\neq-\infty$; moreover, we write $x_j \sim y_j$ if $x_j \prec y_j$ and $y_j \prec x_j$.

\begin{proposition} \label{fittingspec}
Let $M$ be a finitely generated torsion $\Lambda$-module and let $\mathfrak{P}$ be a height $1$ prime ideal of $\Lambda$. If there are equalities
\[
\Fitt_i(M)\Lambda_\mathfrak{P}=\mathfrak{P}^m \Lambda_\mathfrak{P},\quad\Fitt_{\mathcal{O}_j,i}(M \otimes_\Lambda \mathcal{O}_j)=\pi_j^{m_j}\mathcal{O}_j,
\]
then $m_j \sim m e_\mathfrak{P} j$.
\end{proposition}

\begin{proof} See, \emph{e.g.}, \cite[Proposition 4.2]{DRLV}. \end{proof}

Now fix $i\in\N$, a height-$1$ prime ideal $\mathfrak{P}$ of $\Lambda$ and set
\begin{equation} \label{alphabeta}
\mathfrak{P}^{\alpha^{\pm }_i} \Lambda_\mathfrak{P}\defeq\Fitt_i(X^{\pm}) \Lambda_\mathfrak{P},\quad \mathfrak{P}^{\beta^{\pm }_i} \Lambda_\mathfrak{P}\defeq(\mathfrak{D}_i^{\pm })^2 \Lambda_\mathfrak{P}=\Fitt_0\bigl(\Lambda/(\mathfrak{D}_i^{\pm })^2\bigr)\Lambda_{\mathfrak{P}}.
\end{equation}
Moreover, for all integers $j\geq N(\mathfrak{P})$ put
\begin{equation} \label{ajbj}
\Bigl(\pi_j^{a^{\pm}_j}\Bigr)\defeq\Fitt_{\cO_j,i}(X^{\pm}\otimes_\Lambda\cO_j),\quad
\Bigl(\pi_j^{b^{\pm}_j}\Bigr)\defeq s_j\bigl((\mathfrak{D}_i^{\pm})^2\bigr)\cO_j=\Fitt_{\cO_j,0}\bigl((\Lambda / (\mathfrak{D}_i^{\pm})^2)\otimes_\Lambda\cO_j\bigr).
\end{equation}
By Proposition \ref{fittingspec}, we have
\begin{equation} \label{a_jb_jsim}
a^{\pm}_j\sim\alpha^{\pm}_i e_\mathfrak{P} j,\quad b^{\pm}_j\sim\beta^{\pm}_i e_\mathfrak{P} j.
\end{equation}
On the other hand, by Lemma \ref{extScalars}, for each $j$ there are exact sequences of $\mathcal{O}_j$-modules 
\begin{equation} \label{KXTC} 
0 \longrightarrow K^{\pm }_j \longrightarrow X^{\pm } \otimes_\Lambda \mathcal{O}_j \longrightarrow X^{\pm }(T_j) \longrightarrow C^{\pm }_j \longrightarrow 0,
\end{equation}
where $K^{\pm }_j$ and $C^{\pm }_j$ are finite, of order bounded independently of $j$. If we set
\[
\Fitt_0(K^{\pm }_j)\defeq\Bigl(\pi_j^{k^{\pm }_j}\Bigr),\quad\Fitt_i\bigl(X^{\pm }(T_j)\bigr)\defeq\Big(\pi_j^{t^{\pm}_j}\Bigr),\quad\Fitt_0(C^{\pm }_j)\defeq\Bigl(\pi_j^{c^{\pm }_j}\Bigr),
\]
then there is $c\in\N$ independent of $j$ such that $k^\pm_j,c^\pm_j \leq c$. Combining the exact sequence \eqref{KXTC} with \cite[Proposition 2.3.(8)]{DRLV}, we get $a^\pm_j + c^\pm_j \geq t^\pm_j$, hence $t^\pm_j\prec a^\pm_j$. Taking Pontryagin duals of the sequence \eqref{KXTC}, we obtain the exact sequence 
\[
X^\pm(T_j)^\vee\longrightarrow (X^\pm\otimes_\Lambda \mathcal{O}_j)^\vee\longrightarrow (K^\pm_j)^\vee\longrightarrow0.
\]
The $\mathcal{O}_j$-modules $K^\pm_j$, $X^\pm\otimes_\Lambda\cO_j$, $X^\pm(T_j)$ are
finite (the second and the third being torsion modules over a DVR with finite residue field), so they are isomorphic to their Pontryagin duals as $\cO_j$-modules; it follows that
\[
\Fitt_0\bigl(K^\pm_j\bigr)=\Fitt_0\bigl((K^\pm_j)^\vee\bigr),\quad\Fitt_i(Y)=\Fitt_i\bigl(Y^\vee\bigr)
\]
for $Y\in\bigl\{X^\pm\otimes_\Lambda\cO_j, X^\pm(T_j)\bigr\}$. Therefore, again by \cite[Proposition 2.3]{DRLV}, we have $t^\pm_j + k^\pm_j \geq a^\pm_j$, hence $a^\pm_j\prec t^\pm_j$; we conclude that 
\begin{equation}\label{aj=tj}
a^\pm_j \sim t^\pm_j. 
\end{equation} 
As a consequence of Theorem \ref{thmDVR}, equations \eqref{eq:fitting-DVR+} and \eqref{eq:fitting-DVR-} yield also the equalities
\[
t_j^\epsilon=2\cdot\sum_{s\geq i/2}\bigl(\delta^{(2s+1)}(\kappa_j)-\delta^{(2s+2)}(\kappa_j)\bigr),\quad t_j^{-\epsilon}=2\cdot\sum_{s\geq i/2}\bigl(\delta^{(2s)}(\kappa_j)- \delta^{(2s+1)}(\kappa_j)\bigr).
\]
For all integers $r\geq0$ and $k\gg0$, denote by $I_r(k)\subset\Lambda$ the inverse image of $\mathfrak{C}_r(k)$ via the projection $\Lambda\to\Lambda/\mathfrak{m}_\Lambda^k$. Observe that the sequence $\bigl(I_r(k)\bigr)_k$ is non-increasing under inclusion and $\mathfrak{C}_r=\bigcap_{k \ge 0}I_r(k)$. Recall the specialization map $s_j:\Lambda\rightarrow\cO_j$ from \eqref{s-j-eq}.

\begin{lemma} \label{kj0}
For every $j\geq N(\mathfrak P)$, there is an integer $k_{j,0} >0$ such that
\[
s_j\bigl(I_r(k_{j,0})\bigr)\cO_j=s_j(\mathfrak{C}_r)\cO_j.
\]
\end{lemma}

\begin{proof} See the proof of \cite[Lemma 4.3]{DRLV}. \end{proof}

\begin{lemma} \label{eqlengths}
Let $I_1$, $I_2$, $I_3$ be ideals of $\Lambda$ such that $I_1 I_2 \sim I_3$. Then
\[
\length\bigl(\cO_j/s_j(I_3)\cO_j\bigr)\sim\length\bigl(\cO_j/s_j(I_1)\cO_j\bigr) +\length\bigl(\cO_j/s_j(I_2)\cO_j\bigr).
\]
\end{lemma}

\begin{proof} For all $s\in\{1,2,3\}$, we have
\[
\length\bigl(\cO_j/s_j(I_s)\cO_j\bigr)=\length_{\cO_j}\bigl((\Lambda/I_s) \otimes_\Lambda\cO_j\bigr)\sim m_se_\mathfrak{P}j,
\]
where $I_s \Lambda_{\mathfrak{P}}=\mathfrak{P}^{m_s} \Lambda_\mathfrak{P}$ and the equivalence is a consequence of Proposition \ref{fittingspec}. On the other hand, $I_1 I_2 \sim I_3$ by assumption, hence $m_3=m_1 + m_2$ and the lemma follows. \end{proof}

Now we can prove

\begin{theorem} \label{mainThmR}
Under Assumptions \ref{asstripleR} and \ref{asskolyR}, $\Fitt_i(X^{\pm}) \sim  (\mathfrak{D}^{\pm}_i)^2$ for all even $i\in\N$.
\end{theorem}

\begin{proof} Keep the same notation as above. By \eqref{a_jb_jsim}, we have
\[
\beta_i^{\pm}e_\mathfrak{P}j\sim b_j^\pm\defeq\length_{\cO_j}\Bigl(\bigl(\Lambda/(\mathfrak{D}^\pm_i)^2\bigr)\otimes_\Lambda\cO_j\Bigr)= 2\cdot\length_{\cO_j}\bigl(\cO_j/s_j(\mathfrak{D}^\pm_i)\cO_j\bigr).
\]
Then, by definition of $\mathfrak{D}_i^\pm$ and Lemma \ref{eqlengths}, there are equivalences
\begin{equation} \label{bjepsilon}
b_j^\epsilon\sim 2\cdot\sum_{r\geq i/2}\Bigl(\length_{\cO_j}\bigl(\cO_j/s_j(\mathfrak{C}_{2r+1})\cO_j\bigr)-\length_{\cO_j}\bigl(\cO_j/s_j(\mathfrak{C}_{2r+2})\cO_j\bigr)\!\Bigr)
\end{equation}
and
\begin{equation} \label{bj-epsilon}
b_j^{-\epsilon}\sim 2\cdot\sum_{r\geq i/2}\Bigl(\length_{\cO_j}\bigl(\cO_j/s_j(\mathfrak{C}_{2r})\cO_j\bigr)-\length_{\cO_j}\bigl(\cO_j/s_j(\mathfrak{C}_{2r+1})\cO_j\bigr)\!\Bigr).
\end{equation}
We claim that
\[
\length_{\cO_j}\bigl(\cO_j/s_j(\mathfrak{C}_r)\cO_j\bigr)=\delta^{(r)}(\kappa_j)
\]
for all $r\in\N$ and all $j\gg0$. Fix an integer $k_j>0$ such that $k_j>k_{j,0}$, with $k_{j,0}$ as in Lemma \ref{kj0}. The set $\mathcal S$ of ideals involved in the intersection in \eqref{defideals} satisfies the following property: for all $I_1,I_2\in\mathcal S$, there exists $I_3\in\mathcal S$ such that $I_3\subset I_1\cap I_2$. Since $\Lambda/\mathfrak{m}_\Lambda^{k_j}$ is finite, there exists $\kappa^{(k_j,t)} \in \mathfrak{K}_{k_j}$ (with $t \geq 2k_j)$ such that
\[
\mathfrak{C}_r(k_j)=\mathfrak{C}_r(\kappa^{(k_j,t)})=\Bigl(\loc_\ell\bigl(\kappa^{(k_j,t)}_{n}\bigr)\;\Big|\;\text{$n\in\mathcal{N}_{t}$, $\ell\in\mathcal{P}_{t}$, $\ell\nmid n$, $\nu(n)\leq r$}\Bigr).
\]
Then there is an equality
\begin{equation} \label{s-j-C-eq}
s_j\bigl(\mathfrak{C}_r(\kappa^{(k_j,t)})\bigr)=\bigcap_{\kappa^{(k_j,t')}\in\mathfrak{K}_{k_j},t'\geq 2k_j}s_j\bigl(\mathfrak{C}_r(\kappa^{(k_j,t')})\bigr).
\end{equation}
Now we study the right-hand side of \eqref{s-j-C-eq}. The specialization $s_j$ induces the equality $s_j\bigl(\loc_\ell(\kappa_n^{(k_j,t')})\bigr)=\loc_\ell\bigl(\kappa_{j,n}^{(e_j k_j,e_j t')}\bigr)$. Choose $\bar{n}\in\mathcal{N}_{t'}$ with $\nu(\bar{n})= r$ such that $\ind\bigl(\kappa_{j,\bar{n}}^{(e_j k_j,e_j t')}\bigr)=\partial^{(r)}\bigl(\kappa_j^{(e_j k_j,e_j t')}\bigr)$ ($=\min_{i\le r}\partial^{(i)}\bigl(\kappa_j^{(e_j k_j,e_j t')}\bigr)$). Since $t'\ge 2k_j$, Lemma \ref{freeKS} ensures that $\kappa_j^{(e_j k_j,e_jt')}$ is free, so, by Theorem \ref{Rigidity}, $\kappa_j^{(e_j k_j,e_j t')}$ is contained in a free $\cO_j/\pi_j^{e_j k_j}$-module $C$ of rank 1. By Lemma \ref{InfPrimes}, we can choose $\ell\in\mathcal{P}_{t'}$ not dividing $\bar{n}$ such that $\loc_\ell$ is injective on $C$. Then 
\begin{equation*}
\Scale[0.92]
{\begin{split}
\partial^{(r)}\bigl(\kappa_j^{(e_j k_j,e_j t')}\bigr)&=\ind\bigl(\kappa_{j,\bar{n}}^{(e_j k_j,e_j t')}\bigr)=\ind\Bigl(\loc_\ell\bigl(\kappa_{j,\bar{n}}^{(e_j k_j,e_j t')}\bigr),H^1_{\f}(K_\lambda, T_{j,e_j k_j})\!\Bigr)\\
&=\min\Bigl\{\ind\Bigl(\loc_\ell\bigl(\kappa^{(e_j k_j,e_j t')}_{j,n}\bigr),H^1_{\f}(K_\lambda, T_{j,e_j k_j})\Bigr)\;\Big|\;n\in\mathcal{N}_{t'},\;\ell\in \mathcal{P}_{t'},\;\ell\nmid n,\; \nu(n)\leq r\Bigr\}.
\end{split}}
\end{equation*}
This implies that 
\[
s_j\bigl(\mathfrak{C}_r(\kappa^{(k_j,t')})\bigr)\cO_j=\Bigl(\pi_j^{\partial^{(r)}(\kappa_j^{(e_j k_j,e_j t')})}\Bigr),\quad s_j\bigl(\mathfrak{C}_r(\kappa^{(k_j,t)})\bigr)\cO_j=\Bigl(\pi_j^{\delta^{(r)}(\kappa_j^{(e_j k_j)})}\Bigr).
\]
In particular, for $t=t'$ we obtain $\partial^{(r)}\bigl(\kappa_j^{(e_jk_j,e_jt)}\bigr)=\delta^{(r)}\bigl(\kappa_j^{(e_jk_j)}\bigr)$. There are equalities
\[
\cO_j\big/s_j(\mathfrak{C}_r)\cO_j=\cO_j\big/s_j\bigl(I_r(k_j)\bigr)\cO_j=\bigl(\cO_j/\pi_j^{e_jk_j}\cO_j)\big/\bigl(\loc_\ell(\kappa^{(e_jk_j,e_jt)}_{j,n})\bigr)
\]
for $n\in\mathcal{N}_{t}$, $\ell\in\mathcal{P}_t$, $\ell\nmid n$ and $\nu(n)\leq r$. We can assume that $\ell$ is one of the infinitely many primes in $\mathcal{P}_{t}$ such that the restriction of
\[
\loc_\ell:\Sel_{\mathcal{F}(n)}\bigl(K,T_j/(\pi_j^{e_jk_j})\bigr)\longrightarrow H^1_{\f}\bigl(K_\lambda,T_j/(\pi_j^{e_jk_j})\bigr)
\]
to $\bigl(\kappa_{j,n}^{(e_jk_j,e_jt)}\bigr)$ is injective. Now the equalities
\[
\begin{split}
\ind\Bigl(\loc_\ell\bigl(\kappa_{j,n}^{(e_jk_j,e_jt)}\bigr),\cO_j/\pi_j^{e_jk_j}\cO_j\Bigr)&=\ind\Bigl(\kappa_{j,n}^{(e_jk_j,e_jt)},\Sel_{\mathcal{F}(n)}\bigl(K,T_j/(\pi_j^{e_j k_j})\bigr)\!\Bigr)\\
&=\partial^{(r)}\bigl(\kappa_j^{(e_jk_j,e_jt)}\bigr)=\delta^{(r)}\bigl(\kappa_j^{(e_jk_j)}\bigr)=\delta^{(r)}(\kappa_j)
\end{split}
\]
yield the desired claim. Therefore, using \eqref{aj=tj} in both lines below to justify the penultimate equivalences and Proposition \ref{fittingspec} for the first and the last equivalences, and further using \eqref{bjepsilon} and \eqref{bj-epsilon} for the middle one, we obtain 
\[
\beta_i^\epsilon e_\mathfrak{P} j \sim b_j^\epsilon \sim 2\cdot\sum_{r\geq i/2}\bigl(\delta^{(2r+1)}(\kappa_j)-\delta^{(2r+2)}(\kappa_j)\bigr)=t_j^\epsilon\sim a_j^\epsilon \sim \alpha_i^\epsilon e_\mathfrak{P}j
\]
and
\[
\beta_i^{-\epsilon} e_\mathfrak{P} j \sim b_j^{-\epsilon} \sim 2\cdot\sum_{r\geq i/2} \bigl(\delta^{(2r)}(\kappa_j)-\delta^{(2r+1)}(\kappa_j)\bigr)=t_j^{-\epsilon}\sim a_j^{-\epsilon}\sim\alpha_i^{- \epsilon}e_\mathfrak{P}j.
\]
The theorem follows. \end{proof}

\section{Hida families of modular forms} \label{sec4}

In this section, we prove our results on Hida families of modular forms, which are the main contribution of this paper.

\subsection{Hida families}

Let $p$ be an odd prime number. Let  $\cO$ be the valuation ring of  a finite extension of $\Q_p$, set $\Gamma \defeq 1 + p\Z_p$ and denote by $\Lambda\defeq\mathcal{O} \llbracket \Gamma \rrbracket $ the Iwasawa algebra of $\Gamma$ over $\mathcal{O}$, which is isomorphic to the power series ring $\cO\llbracket T\rrbracket$ in one variable. 

\subsubsection{Arithmetic morphisms}

Let $\mathcal{R}$ be a complete local noetherian domain that is finite and flat as a $\Lambda$-algebra. A homomorphism $\eta\in\mathcal{X}\defeq\Hom_{\cont}\bigl(\mathcal{R},\bar{\Q}_p\bigr)$ is \emph{arithmetic} if the composition $\Gamma \rightarrow \mathcal{R}^\times \xrightarrow{\eta} \bar{\Q}_p$ has the form $\gamma \mapsto \psi_\eta(\gamma) \gamma^{k_\eta-2}$ for some integer $k_\eta \geq 2$ and some finite order character $\psi_\eta$ of $\Gamma$, where $\Gamma\hookrightarrow\mathcal{R}^\times$ is the composition of the canonical injection $\Gamma\hookrightarrow \Lambda$ mapping an element $\gamma\in\Gamma$ to the corresponding group-like element in $\Lambda$ and the structural injection $\Lambda\hookrightarrow\mathcal{R}$. We denote by $\mathcal{X}^{\arith}$ the set of arithmetic morphisms of $\mathcal{X}$ and by $\omega$ the Teichm\"uller character modulo $p$.  

\subsubsection{$p$-adic Hida families}

Let $N\geq1$ be an integer such that $p\nmid6N$.

\begin{definition} \label{def-Hida}
A \emph{$p$-adic Hida family} over $\mathcal R$ of tame level $N$ is a formal  power series  
\[
\mathbf{f}=\sum_{n \geq 1} a_n(\mathbf{f})q^n \in \mathcal{R} \llbracket q \rrbracket
\]
such that for all $\eta \in \mathcal X^{\arith}$ the specialization $\mathbf{f}_\eta\defeq\sum_{n \geq 1}\eta\bigl(a_n(\mathbf{f})\bigr)q^n$ is the $q$-expansion of a $p$-ordinary and $p$-stabilized cusp form of weight $k_\eta$, character $\psi_\eta \omega^{2-k_\eta}$ and level $\Gamma_1(Np^{r_\eta})$, where $r_\eta\defeq\max\bigl\{1, \mathrm{cond}(\psi_\eta)\bigr\}$.
\end{definition}

See, \emph{e.g.}, \cite[\S 2.4]{vigni-SNS} for definitions and details on the $p$-stabilization of modular forms and on $p$-stabilized newforms.

\subsubsection{(Big) Galois representations}

For all $\eta\in\Hom_{\cont}(\mathcal{R},\bar{\Q}_p)$, set 
\[
\p_\eta\defeq\ker(\eta)\subset\mathcal R,\quad F_\eta \defeq \mathcal{R}_{\p_\eta}\big/\p_\eta\mathcal{R}_{\p_\eta},
\]
where $\mathcal{R}_{\p_\eta}$ is the localization of $\mathcal{R}$ at $\p_\eta$; 
then $F_\eta$ is a finite extension of $\Q_p$ with valuation ring equal to the integral closure $S_\eta$ of $\mathcal{R}/\p_\eta$ and we fix an embedding $i_\eta: F_\eta \hookrightarrow \bar{\Q}_p$. For $\eta\in\mathcal{X}^\mathrm{arith}$, also denote by $V_\eta$ the self-dual twist of the Galois representation attached to $\mathbf{f}_\eta$ by Deligne; up to base change to $F_\eta$, $V_\eta$ is an $F_\eta$-vector space of dimension $2$ equipped 
with an action of $G_\Q$, unramified outside the prime ideals dividing $Np$ and such that the arithmetic Frobenius at a prime $\ell\nmid Np$ has trace equal to $a_\ell(\mathbf{f}_\eta)/\ell^{k_\eta/2}$, where we write $a_\ell(\mathbf{f}_\eta)\defeq\eta\bigl(a_\ell(\mathbf{f})\bigr)$. 

By \cite[Theorem II]{Hida-Galois}, we can attach to $\mathbf{f}$ a free $\mathcal{R}$-module $\mathbf{T}$ of rank $2$ endowed with a continuous $G_\Q$-action such that the ``big'' representation
\[
\rho_\mathbf{f}:G_\Q\longrightarrow\Aut(\mathbf{T}) \simeq \GL_2(\mathcal{R})
\]
is unramified outside $Np$ and $\tr\bigl(\rho_\mathbf{f}(\Frob_\ell)\bigr)=a_\ell(\mathbf{f})$ for all primes $\ell\nmid Np$, where $\Frob_\ell$ is an arithmetic Frobenius. Let $\mathbf{T}^{\dag}=\mathbf{T}\otimes\Theta^{-1}$ be the \emph{critical twist} of $\mathbf{T}$ 
defined in \cite[Definition 2.1.3]{Howard-Variation}; this definition depends on the choice of a \emph{critical character} $\Theta:G_\Q\rightarrow\Lambda^\times$, which we suppose fixed throughout as in \cite[Remark 2.1.4]{Howard-Variation}. Let 
\[
\rho_{\mathbf{f}}^\dagger:G_\Q\longrightarrow\Aut_\mathcal{R}(\mathbf{T}^\dagger)
\]
be the associated Galois representation. Then, for all $\eta \in \mathcal{X}^\mathrm{arith}$, there is an isomorphism
\[
\mathbf{T}^\dag_{\p_\eta}\big/\p_\eta\mathbf{T}^\dag_{\p_\eta}\simeq V_{\eta}
\]
of representations of $G_\Q$, where $\mathbf{T}^\dag_{\p_\eta}\defeq\mathbf{T}^\dag\otimes_\mathcal{R}\mathcal{R}_{\p_\eta}$. We also set $\mathbf{A}^\dag\defeq\mathbf{T}^\dag \otimes_{\mathcal{R}}\mathcal{R}^\vee$, where $\mathcal{R}^\vee$ is the Pontryagin dual of $\mathcal{R}$. For each $\eta \in \mathcal{X}$, define the $G_\Q$-representation  
\[
T_\eta \defeq \mathbf{T}^\dag/\p_\eta \mathbf{T}^\dag\otimes_{\mathcal{R}/\p_\eta} S_\eta,
\]
which is a Galois-stable lattice inside $V_\eta\defeq T_\eta\otimes_{S_\eta}F_\eta$, and consider the $p$-divisible group $A_\eta\defeq V_\eta/T_\eta$. 

\subsubsection{Arithmetic assumptions and local conditions}

From here on, we work under 

\begin{assumption} \label{Hidaass}
\begin{enumerate}
\item $\mathcal{R}=\Lambda$
\item The semi-simple residual Galois representation attached to $\mathbf{f}$ is absolutely irreducible.
\item The image of $\rho_{\mathbf{f}}^\dagger$ contains $\Z_p^\times$ embedded diagonally (see, \emph{e.g.}, \cite[Assumption 4.13 and Remark 4.14]{zerman2025quaternionic}). 
\end{enumerate}
\end{assumption}

As explained in \cite[(4.3.9)]{NP}, point (1) of the previous assumption is satisfied for example when there exists an arithmetic $\eta_f \in \chi^{\arith}$ such that $f=\mathbf{f}_{\eta_f}$ is a $p$-isolated newform in $S_k\bigl(\Gamma_0(Np))$ for an even integer $k\geq2$.

Let $K$ be an imaginary quadratic field of discriminant $D_K$ such that $(D_K,pN)=1$. 
Assume the Heegner hypothesis holds: all primes diving $N$ are split in $K$.
Recall that the representation $\mathbf{T}$ of $G_{\Q_p}$ is equipped with an ordinary filtration 
\begin{equation} \label{ordinary-seq-eq}
0\longrightarrow \mathbf{T}^+\longrightarrow\mathbf{T}\longrightarrow\mathbf{T}^-\longrightarrow 0
\end{equation}
such that both $\mathbf{T}^+$ and $\mathbf{T}^-$ are free $\Lambda$-modules of rank $1$ and $\mathbf{T}^-$ is the maximal unramified quotient of $\mathbf{T}$ (see, \emph{e.g.}, \cite[Proposition 2.4.1]{Howard-Variation}). Twisting by $\Theta^{-1}$, exact sequence \eqref{ordinary-seq-eq} induces a similar ordinary filtration 
\begin{equation} \label{ordinary-twist-seq-eq}
0\longrightarrow \mathbf{T}^{\dagger,+}\longrightarrow \mathbf{T}^\dagger\longrightarrow \mathbf{T}^{\dagger,-}\longrightarrow 0
\end{equation} 
on $\mathbf{T}^\dagger$. For each prime $v$ of $K$ above $p$, define 
\[
H^1_{\ord}(K_v,\mathbf{T}^\dagger)\defeq\im\Bigl(H^1\bigl(K_v,\mathbf{T}^{\dagger,+}\bigr)\longrightarrow H^1\bigl(K_v,\mathbf{T}^\dagger\bigr)\!\Bigr),
\]
where the map on the right is canonically induced by the inclusion $\mathbf{T}^{\dagger,+}\hookrightarrow\mathbf{T}^\dagger$. For each $\eta\in\mathcal{X}$, exact sequence \eqref{ordinary-twist-seq-eq} gives rise to an ordinary filtration 
\[
0\longrightarrow V_\eta^+\longrightarrow V_\eta\longrightarrow V_\eta^-\longrightarrow 0
\]
for $V_\eta$, and for a prime $v$ of $K$ above $p$ we define $H^1_{\ord}(K_v,V_\eta)$ as above.
By propagation, we then define $H^1_{\ord}(K_v,T_\eta)$ and $H^1_{\ord}(K_v,A_\eta)$. Analogously, for primes $v\nmid p$ of $K$, let $H^1_{\mathrm{unr}}(K_v,V_\eta)$ be the 
unramified cohomology group of $V_\eta$ and define $H^1_{\unr}(K_v,T_\eta)$ and $H^1_{\unr}(K_v,A_\eta)$ by propagation.

\subsubsection{Greenberg Selmer structures}

The Greenberg Selmer structure $\mathcal{F}_{\Gr}$ for $\mathbf{T}^\dag$ is
\[
H^1_{\mathcal{F}_{\Gr}}\bigl(K_v,\mathbf{T}^\dag\bigr)\defeq \begin{cases}
    H^1_{\unr}\bigl(K_v,\mathbf{T}^\dag\bigr) & \text{if $v\nmid p$}\\[2mm]
    H^1_{\ord}\bigl(K_v,\mathbf{T}^\dag\bigr) & \text{if $v\,|\,p$}
\end{cases}
\]
(\emph{cf.} also \cite[Definition 2.4.2]{Howard-Variation}). Moreover, the Selmer structure $\mathcal{F}_\eta$ for $T_\eta$ is
\[
H^1_{\mathcal{F}_\eta}(K_v,T_\eta)\defeq \begin{cases}
    H^1_{\unr}(K_v,T_\eta) & \text{if $v\nmid p$},\\[2mm]
    H^1_{\ord}(K_v,T_\eta) & \text{if $v\,|\,p$}. 
\end{cases}
\]
The $\Lambda$-module $\mathbf{T}^\dagger$ is equipped with a perfect, alternating, $\Lambda$-bilinear and $G_\Q$-invariant pairing $$(\cdot,\cdot):\mathbf{T}^\dagger\times\mathbf{T}^\dagger\rightarrow \Lambda(1).$$
We notice that, equivalently, $\mathbf{T}^\dag$ is equipped with a perfect, symmetric, $\Lambda$-bilinear and $G_K$-invariant pairing 
\[ 
[\cdot,\cdot]:\mathbf{T}^\dagger\times\mathbf{T}^\dagger\longrightarrow\Lambda(1),
\]
obtained by setting $[x,y] \defeq (x,y^{\tau_c})$, that satisfies the formula $\bigl[x^\sigma,y^{\tau_c\sigma\tau_c^{-1}}\bigr]=[x,y]^\sigma$ for all $x,y\in\mathbf{T}^\dagger$ and $\sigma\in G_K$. 

The pairing $(\cdot,\cdot)$ induces for each $\eta\in \mathcal{X}$ a perfect, alternating, $S_\eta$-bilinear and $G_\Q$-invariant pairing
\[
(\cdot,\cdot)_\eta:T_\eta\times T_\eta\longrightarrow S_\eta(1).
\]
To check perfectness, observe that the pairing $(\cdot,\cdot)$ induces an isomorphism of $\Lambda$-modules $\mathbf{T}^\dagger\simeq \Hom_\Lambda(\mathbf{T}^\dagger,\Lambda)$; since $\mathbf{T}^\dagger$ is free of finite rank over $\Lambda$, we also obtain an isomorphism of $S_\eta$-modules ${T}_{\eta}\simeq \Hom_{S_\eta}({T}_{\eta},S_{\eta})$ that is induced by $(\cdot,\cdot)_\eta$, so $(\cdot,\cdot)_\eta$ is perfect. As before, $T_\eta$ is also equipped with a perfect, symmetric, $S_\eta$-bilinear and $G_K$-invariant pairing
\[
{[\cdot,\cdot]}_\eta:T_\eta\times T_\eta\longrightarrow S_\eta(1)
\]
satisfying ${[x^\sigma,y^{\tau_c\sigma\tau_c^{-1}}]}_\eta={[x,y]}_\eta^\sigma$ for all $x,y \in T_\eta$ and all $\sigma \in G_K$.

\subsubsection{Validity of assumptions}

For every integer $s\ge1$, let $\mathcal{P}_s$ be the maximal set of prime numbers such that the pair $\bigl(\mathbf{T}^\dag_s,\mathcal{P}_s\bigr)$ satisfies Assumption \ref{ass:kolyvagin-primes}, where $\mathbf{T}^\dag_s\defeq\mathbf{T}^\dag/\mathfrak{m}_\Lambda^s \mathbf{T}^\dag$. Set $\mathcal{P}\defeq\bigcup_{s\ge 1}\mathcal{P}_s$. We check the validity of the assumptions introduced in the previous section.

\begin{proposition} \label{tripleDVR}
Let $\eta\in\Hom_{\cont}(\mathcal{R},\bar{\Q}_p)$. The triple $(T_\eta,\mathcal{F}_{\eta},\mathcal{P})$ satisfies Assumption \ref{ass:Howard-assumption}.
\end{proposition}

\begin{proof}
The residual Galois representation $\overline{T}_\eta$ of $T_\eta$ is the same as the $G_\Q$-representation 
$\mathbf{T}^\dagger/\mathfrak{m}_\Lambda\mathbf{T}^\dagger$ up to finite base change, where $\mathfrak{m}_\Lambda$ is the 
maximal ideal of $\Lambda$. Thus, condition (1) in Assumption \ref{ass:Howard-assumption} holds by Assumption \ref{Hidaass} (2). 

We show that (2) holds. Thanks to Assumption \ref{Hidaass}, the group $\Bmu_{p-1}$ of $(p-1)$-roots of unity, 
embedded into $\Z_p^\times$ via the Teichm\"uller representatives, 
is contained  via the diagonal embedding in the image of the representation  $\rho_{\mathbf{f}}^\dagger:G_\Q\rightarrow \Aut_\Lambda(\mathbf{T}^\dagger)$. Since it has order coprime with $p$, we also have that $\Bmu_{p-1}$ is also contained in the composition $\bar{\rho}_\eta^\dagger$ of $\rho_{\mathbf{f}}^\dagger$ with the natural maps $\Aut_\Lambda(\mathbf{T}^\dagger)\to\Aut_\cO(\mathbf{T}^\dagger/\mathfrak{m}_\Lambda \mathbf{T}^\dagger)\to\Aut_{S_\eta}(\bar{T}_\eta)$ (as the last one is injective). Therefore, the group $\mathcal{N}\defeq\bigl\{\zeta^2\mid\zeta\in\Bmu_{p-1}\bigr\}$ has order prime to $p$, is contained in $\bar{\rho}_\eta^\dagger(G_K)$ and is non-trivial because $p\geq5$.

Let $\mathcal{M}\defeq\rho_{\mathbf{f}}^\dagger(G_K)\subset\GL_2(\Lambda)$, denote by $H$ the kernel of $\rho_{\mathbf{f}}^\dagger|_{G_K}$ and set $F\defeq\overline{\Q}^H$, 
so that $\Gal(F/K)\simeq\mathcal{M}$. Consider the group $G\defeq\Gal\bigl(FK(\Bmu_{p^\infty})/K\bigr)$, which surjects onto $\mathcal{M}$ via $\rho_{\mathbf{f}}^\dagger$, and the inflation-restriction exact sequence 
\[
0\longrightarrow  H^1\Bigl(G/\mathcal{N},\overline{T}_\eta^\mathcal{N}\Bigr)\longrightarrow H^1\bigl(G,\overline{T}_\eta\bigr)\longrightarrow H^1\bigl(\mathcal{N},\overline{T}_\eta\bigr).
\]
Since $\mathcal{N}$ is non-trivial and has cardinality prime to $p$, $\overline{T}_\eta^\mathcal{N}=0$ and 
$H^1(\mathcal{N},\overline{T}_\eta)=0$ (see, \emph{e.g.}, \cite[Corollary 1]{AW}); as a consequence, 
$H^1(G,\overline{T}_\eta)=0$, which shows that condition (2) in Assumption \ref{ass:Howard-assumption} holds.

Since the local conditions defining $H^1_{\mathcal{F}_\eta}(K_v,T_\eta)$ are obtained by propagation from local conditions for $V_\eta$, condition (3) in Assumption \ref{ass:Howard-assumption} is satisfied by \cite[Lemma 3.7.1]{MR}. 
    
Now we check condition (4). The self-orthogonality of the local Selmer condition for $V_\eta$ is well known for $v\nmid p$ (see, \emph{e.g.}, \cite[Proposition 3.8]{BK}), while for $v\,|\,p$ it is a consequence of the definition of the ordinary condition. Let 
\[
0\longrightarrow V_\eta^+\longrightarrow V_\eta\longrightarrow V_\eta^-\longrightarrow0
\]
denote the ordinary filtration of $V_\eta$. First, $V_\eta^+$ is isotropic for the pairing: this is true because the action on $V_\eta^+$ is via $\delta\chi_\cyc^{k/2}$ for an unramified character $\delta$, so for $\sigma\in G_K$ and $x,y\in V_\eta^+$ there is an equality $\bigl[x^\sigma,y^{\sigma^{-1}}\bigr]_\eta={[x,y]}_\eta$; on the other hand, we have
\[
\bigl[x^\sigma,y^{\sigma^{-1}}\bigr]_\eta=\bigl[x^\sigma,y^{\tau\sigma\tau^{-1}}\bigr]_\eta={[x,y]}_\eta^\sigma
\]
by the fact that $\tau_c$ acts on $G_K$ as the inversion map $\tau_c\sigma\tau_c^{-1}=\sigma^{-1}$; these two equations are clearly incompatible unless ${[x,y]}_\eta=0$. Therefore, the perfect pairing 
\[
V_\eta\times V_\eta\longrightarrow F_\eta(1)
\]
identifies $V^-_\eta$ with $\Hom_{F_\eta}\bigl(V^+_\eta,F_\eta(1)\bigr)$. Therefore, by \cite[Proposition 3.8]{BK}, there is a perfect pairing 
\[
H^1(K_v, V^+_\eta)\times H^1(K_v,V^-_\eta)\longrightarrow F_\eta(1),
\]
so the image of $H^1(K_v,V^+_\eta)$ in $H^1(K_v,V_\eta)$ is maximal isotropic.
Since the local conditions for $T_\eta$ are induced by propagation from those of $V_\eta$, they are still self-dual, which shows that condition (4) in Assumption \ref{ass:Howard-assumption} is satisfied. 

As for condition (5.a), it is a consequence of the existence of the skew-symmetric, $G_\Q$-invariant pairing on $\mathbf{T}^\dag$, while condition (5.b) follows directly from definitions outside $p$ and for $p$ from the definitions and the fact that the ordinary filtration is defined over $\Q$. 
    
Finally, condition (6) follows from the arguments in \cite[Lemma 1.6.2]{HoHeeg}.
\end{proof}

\begin{proposition} \label{tripleLambdaHida}
The triple $\bigl(\mathbf{T}^\dag,\mathcal{F}_{\Gr},\mathcal{P}\bigr)$ satisfies Assumption $\ref{asstripleR}$.
\end{proposition}

\begin{proof} Condition (1) is satisfied thanks to Proposition \ref{tripleDVR}, while 
(2) is clear from the definitions. We check the control theorem condition. The maps $\mathbf{T}^\dag \rightarrow T_\eta$ and $A_\eta \rightarrow \mathbf{A}^\dag$ induce maps
\[
H^1_{\mathcal{F}_{\Gr}}\bigl(K_v,\mathbf{T}^\dag/\p_\eta \mathbf{T}^\dag\bigr) \longrightarrow H^1_{\mathcal{F}_\eta}(K_v,T_\eta),\quad H^1_{\mathcal{F}_\eta}(K_v,A_\eta) \longrightarrow H^1_{\mathcal{F}_{\Gr}}\bigl(K_v,\mathbf{A}^\dag[\p_\eta]\bigr)
\]
with finite kernel and cokernel of cardinality bounded by a constant that depends only on $[S_\eta:\Lambda/\p_\eta]$: for the first map this follows from the arguments in \cite[Lemma 5.3.13]{MR}, \cite[Lemma 2.2.7]{HoHeeg} and \cite[Lemma 4.1.13]{NP}, while for the second map this follows by local duality. Then the statement follows as in \cite[Proposition 5.3.14]{MR}. \end{proof}

\subsection{Universal Kolyvagin systems of big Heegner points} \label{UKSHida}
A universal Kolyvagin system for $(\mathbf{T}^\dagger,\mathcal{F}_\mathrm{Gr},\mathcal{P})$ can be obtained by specializing the anticyclotomic system constructed by B\"{u}y\"{u}kboduk in \cite{Kazim} (under a few more technical assumptions: see \cite[(H.stz) and (H.Tam)]{Kazim}) with respect to the augmentation ideal. 
The same construction can be directly obtained from the Euler system of big Heegner points $\mathfrak{z}_n\in H^1(K[n],\mathbf{T}^\dagger)$ of conductor $n$ built by Howard in \cite{Howard-Variation} by applying Kolyvagin derivatives as in \cite{mastella2025anticyclotomic}. 

From now on, denote by 
\[
\kappa\in\uks\bigl(\mathbf{T}^\dagger,\mathcal{F}_\mathrm{Gr},\mathcal{P}\bigr)
\]
this universal Kolyvagin system. Then $\kappa_1\in\Sel_{\Gr}\bigl(K,\mathbf{T}^\dag\bigr)$ and 
we propose

\begin{conjecture} \label{kappa1nottors}
The class $\kappa_1$ is not $\Lambda$-torsion in $\Sel_{\Gr}\bigl(K,\mathbf{T}^\dag\bigr)$.    
\end{conjecture}

In what follows, we assume the validity of Conjecture \ref{kappa1nottors}.

\begin{proposition}\label{assKappaHida}
The system $\kappa$ satisfies Assumption \ref{asskolyR}.    
\end{proposition}

\begin{proof}
Given $\eta \in \Hom_{\cont}(\Lambda,\bar{\Q}_p)$, $\kappa_{\eta,1}$ is the image of $\kappa_1$ via the composition of maps
\begin{equation} \label{compSpec}
\Sel_{\Gr}(K,\mathbf{T}^{\dag}) \longrightarrow \Sel_{\Gr}(K,\mathbf{T}^{\dag})/\p_\eta \longrightarrow \Sel_{\mathcal{F}_\eta}(K,T_\eta).
\end{equation}
By \cite[Proposition 2.1.5]{MR}, there is a pseudoisomorphism
\[
g:\Sel_{\Gr}(K,\mathbf{T}^{\dag})\longrightarrow\Lambda^r\oplus\Sel_{\Gr}(K,\mathbf{T}^{\dag})_{\tors}
\]
that sends $\kappa_1$ to a pair $\bigl(\kappa_1^{(f)},\kappa_1^{(t)}\bigr)$ with $\kappa_1^{(f)} \neq 0$ since $\kappa_1$ is not torsion by assumption. Since $\Lambda$ is a UFD by \cite[Theorem 20.3]{matsumura}, $\kappa_1^{(f)}$ is divisible only by a finite number of primes of $\Lambda$, so the same fact holds true for $\kappa_1$. We fix a prime $\p_\eta=(f) \notin \Sigma_{\mathrm{CT}}$ of height $1$ of $\Lambda$ such that $f \nmid \kappa_1^{(f)}$ and $f \nmid p$ in $\Lambda$. Then, by Conjecture \ref{kappa1nottors}, the image of $\kappa_1$ via the first map in \eqref{compSpec} is non-trivial. In order to conclude, we assume by contradiction that this image belongs to the kernel of the second map in \eqref{compSpec}. Since this kernel is finite by Proposition \ref{tripleLambdaHida}, it would imply that there exists an integer $z \geq 1$ such that $f\,|\,z \cdot\kappa_1$. Then $f\,|\,z \cdot \kappa_1^{(f)}$ but this contradicts our assumption.

As for the second part of Assumption \ref{asskolyR}, recall that $\kappa$ is built in \cite{Kazim} by applying Kolyvagin derivatives to the big Heegner points of \cite{Howard-Variation}. Then the result follows from \cite[Proposition 2.3.5]{Howard-Variation} together with the formula
\[
\tau_c D_n = (-1)^{\nu(n)}D_n\tau_c \mod p^s,
\]
which holds in $\Z[\Gal(K^{\mathrm{ab}}/\Q)]$ for every $s\ge1$ and $n\in\mathcal{N}_s$, where $D_n$ is the $n$-th Kolyvagin derivative. \end{proof}

\begin{remark} \label{N^-}
More generally, suppose that there is a factorization $N=N^+N^-$ where all prime factors of $N^+$ (respectively, $N^-$) split (respectively, are inert) in $K$; in the setting of this article, where all primes dividing $N$ split in $K$, one has $N^-=1$. One may relax the Heegner hypothesis and replace it with the following \emph{generalized Heegner hypothesis}: $N^-$ is a square-free product of an \emph{even} number of primes. Assuming $N^->1$, an analogue of the classes $\kappa_n^{(s,t)}$ have been constructed by the last named author in \cite{zerman2025quaternionic} by making use of the Euler system of quaternionic big Heegner points of \cite{LV-MM} built on towers of Shimura curves attached to quaternion algebras of discriminant $N^-$. When $N^->1$, it should be possible to prove a quaternionic analogue of \cite[Proposition 2.3.5]{Howard-Variation} describing the action of complex conjugation on big Heegner points, from which one would deduce, as before, that the class $\kappa_n^{(s,t)}$ belongs to $\Sel_{\mathcal{F}(n)}(K,T_s)^{\epsilon_n}$ for all $s\geq 1$, all $\kappa^{(s,t)}\in\mathfrak{K}_s$ and all $n\in\mathcal{N}_t$. Granting that, the arguments in this section (including the counterparts of Theorems \ref{DVRtheorem} and \ref{HidaTheorem}) hold when $N^->1$ as well. However, to the best of our knowledge, no generalization of \cite[Proposition 2.3.5]{Howard-Variation} is currently available in the literature: in order to keep the size of the present paper under control, we decided to postpone a careful discussion of this issue to a future project.  
\end{remark}

In particular, there exists a finite set $\Sigma_{\EX}$ of height-$1$ primes of $\Lambda$ such that $\kappa_\eta$ satisfies Assumption \ref{asskappaDVR} whenever $\mathfrak{p}_\eta \notin \Sigma_{\EX}$.

\subsection{Main results}

First, we can prove our result on Selmer groups of the specializations of $\mathbf{f}$. Let us fix $\eta \in \Hom_{\cont}(\Lambda,\overline{\Q}_p)$.

\begin{theorem} \label{DVRtheorem}
Assume $\p_\eta\notin\Sigma_{\EX}$. There are isomorphisms
\[
\Sel_{\mathcal{F}_\eta}^\epsilon(K,A_\eta)\simeq(F_\eta/S_\eta)\oplus\bigoplus_{i\in\N}\bigl(S_\eta/\pi_\eta^{d_i^\epsilon}S_\eta\bigr)^2,\quad\Sel_{\mathcal{F}_\eta}^{-\epsilon}(K,A_\eta) \simeq\bigoplus_{i\in\N}\bigl(S_\eta/\pi_\eta^{d_i^{-\epsilon}}S_\eta\bigr)^2,
\]
where $d_i^\epsilon\defeq\delta^{(2i+1)}(\kappa_\eta)-\delta^{(2i+2)}(\kappa_\eta)$ and $d_i^{-\epsilon}\defeq\delta^{(2i)}(\kappa_\eta)-\delta^{(2i+1)}(\kappa_\eta)$.
\end{theorem}

\begin{proof} This is a consequence of Theorem \ref{thmDVR}. \end{proof}

Now we can prove our main result on the structure of Shafarevich--Tate groups of Hida families. Set $X^\pm \defeq \Sha^\pm_{\Gr}(K,\mathbf{A}^\dag)$. For all even $i\in\N$, choose ideals $\mathfrak{D}_i^\pm$ of $\Lambda$ as in \S\ref{constrideals}, using the universal Kolyvagin system $\kappa$ introduced in \S \ref{UKSHida}. 

\begin{theorem}\label{HidaTheorem}
Let $i\in\N$. Then
\[
\Fitt_i(X^\pm) \sim \begin{cases} (\mathfrak{D}_i^\pm)^2 & \text{if $i$ is even}, \\[2mm]
\mathfrak{D}^\pm_{i-1} \cdot \mathfrak{D}^\pm_{i+1} & \text{if $i$ is odd}.
\end{cases}
\]
\end{theorem}

\begin{proof} Combine Theorems \ref{evenoddFitt} and \ref{mainThmR}. \end{proof}

\bibliographystyle{amsplain}
\bibliography{Bibliography}

\end{document}